\documentclass[a4paper,11pt]{article}
\usepackage{amsmath,amsthm,amssymb,enumitem,xcolor}
\usepackage{tikz}

\usepackage[hang]{footmisc}
\usepackage{makecell}
\usepackage{caption}
\usepackage{subcaption}

\usepackage{tkz-euclide}

\usepackage[nosort,nocompress,noadjust]{cite}

\usepackage[bookmarks=false,hyperfootnotes=false,colorlinks,
    linkcolor={red!60!black},
    citecolor={blue!50!black},
    urlcolor={blue!80!black}]{hyperref}

\renewcommand{\eqref}[1]{\hyperref[#1]{(\ref{#1})}}

\usepackage{faktor}

\newlist{enumlist}{enumerate}{1}
\setlist[enumlist]{labelindent=0cm,label=\arabic*.,labelwidth=2.5ex,labelsep=0.5ex,leftmargin=3ex,align=left,topsep=0.5ex,itemsep=1ex,parsep=1ex}

\newlist{itemlist}{itemize}{1}
\setlist[itemlist]{labelindent=0cm,label=$\bullet$,labelwidth=2.5ex,labelsep=0.5ex,leftmargin=3ex,align=left,topsep=0.5ex,itemsep=1ex,parsep=1ex}

\numberwithin{equation}{section}

{\theoremstyle{definition}\newtheorem{definition}{Definition}[section]
\newtheorem*{definition*}{Definition}
\newtheorem{notation}[definition]{Notation}

\newtheorem{remark}[definition]{Remark}
\newtheorem{example}[definition]{Example}
\newtheorem*{example*}{Example}
\newtheorem*{examples*}{Examples}
}

\newtheorem{proposition}[definition]{Proposition}
\newtheorem{lemma}[definition]{Lemma}
\newtheorem{theorem}[definition]{Theorem}
\newtheorem{corollary}[definition]{Corollary}
\newtheorem*{corollary*}{Corollary}

\newtheorem{letterthm}{Theorem}

{\theoremstyle{definition}}

\newcommand{\C}{\mathbb{C}}
\newcommand{\cC}{\mathcal{C}}

\newcommand{\ot}{\otimes}

\newcommand{\Z}{\mathbb{Z}}

\newcommand{\N}{\mathbb{N}}

\newcommand{\Tr}{\operatorname{Tr}}

\newcommand{\E}{\mathbb{E}}

\newcommand{\Ad}{\operatorname{Ad}}

\newcommand{\cF}{\mathcal{F}}
\newcommand{\T}{\mathbb{T}}

\newcommand{\cB}{\mathcal{B}}
\newcommand{\cW}{\mathcal{W}}

\newcommand{\cM}{\mathcal{M}}

\newcommand{\cL}{\mathcal{L}}

\newcommand{\cE}{\mathcal{E}}
\newcommand{\cFC}{\mathcal{FC}}

\newcommand{\cP}{\mathcal{P}}

\newcommand{\cT}{\mathcal{T}}

\newcommand{\bP}{\mathbb{P}}

\newcommand{\G}{\mathbb{G}}

\newcommand{\FT}{\mathbb{FT}}

\DeclareMathOperator{\spann}{span}

\DeclareMathOperator{\End}{End}

\DeclareMathOperator{\mult}{mult}

\DeclareMathOperator{\TL}{TL}
\DeclareMathOperator{\FC}{FC}

\DeclareMathOperator{\Part}{Part}
\DeclareMathOperator{\Mo}{Mo}
\DeclareMathOperator{\Ba}{Ba}

\DeclareMathOperator{\Ind}{Ind}
\DeclareMathOperator{\Wk}{Wk}

\begin{document}

\begin{center}
{\boldmath\Large\bf Traces On Diagram Algebras I: \\
Free Partition Quantum Groups, Random Lattice Paths And Random Walks On Trees}

\bigskip

{\sc by Jonas Wahl\footnote{\noindent Hausdorff Center for Mathematics, Bonn (Germany).\\ E-mail: wahl@iam.uni-bonn.de.}}

\end{center}

\begin{abstract}
\noindent We classify extremal traces on the seven direct limit algebras of noncrossing partitions arising from the classification of free partition quantum groups of Banica-Speicher \cite{BS09} and Weber \cite{We13}. For the infinite-dimensional Temperley-Lieb-algebra (corresponding to the quantum group $O^+_N$) and the Motzkin algebra ($B^+_N$), the classification of extremal traces implies a classification result for well-known types of central random lattice paths. For the $2$-Fuss-Catalan algebra ($H_N^+$) we solve the classification problem by computing the \emph{minimal or exit boundary} (also known as the \emph{absolute}) for central random walks on the Fibonacci tree, thereby solving a probabilistic problem of independent interest, and to our knowledge the first such result for a nonhomogeneous tree. In the course of this article, we also discuss the branching graphs for all seven examples of free partition quantum groups, compute those that were not already known, and provide new formulas for the dimensions of their irreducible representations.
\end{abstract}

\section{Introduction}

The classification problem for extremal traces on direct limits of finite-dimensional $C^*$-algebras has a decennia long history in mathematics that is intertwined with many different fields such as probability theory, symmetric functions, $K$-theory or representation theory \cite{BO16} just to name a few. For instance, Thoma's classification of extremal traces on $\C[S_{\infty}]$ \cite{Th64}, the group algebra of the infinite symmetric group, has spurned a vast literature of beautiful results ranging from applications to determinantal point processes \cite{BO98} \cite{O03} to free probability \cite{Bn98} to solutions of the Yang-Baxter equation \cite{LPW19}.
 
In this article, we are interested in a family of direct limit algebras that arises from the theory of partition (a.k.a. easy) quantum groups initiated by Banica and Speicher in their seminal article \cite{BS09}. More precisely, Banica and Speicher introduced the notion of \emph{categories of partitions} $\cC = (\cC(k,l))_{k,l \geq 0}$ which model the representation theory of the partition quantum groups. The elements of $\cC(k,k), \ k \geq 0$ have a simple diagrammatical representation and serve as a basis for an inductive sequence of finite-dimensional algebras 
\[ \ldots \subset A_{(\cC,\delta)}(k) \subset A_{(\cC,\delta)}(k+1) \subset \dots \]
depending on an additional loop parameter $\delta > 0$. If the loop parameter $\delta$ is chosen generically, the algebras become semisimple and thus admit a limit object $A_{(\cC,\delta)}(\infty)$. Banica and Speicher further divided the categories of partitions in several subfamilies that are amenable to classification. In this article, we are concerned with the limit algebras of categories of \emph{noncrossing} partitions of which there are exactly seven (corresponding to the compact quantum groups $\G = O_N^+,S_N^+, B_N^+, H_N^+, S_N^{'+}, B_N^{'+},B_N^{ \# +}$ when $\delta =N$) as shown by Weber \cite{We13}.

The most prominent limit algebras appearing in this setting are the infinite Temperley-Lieb algebra $\TL(\infty,\delta)$ (in the three cases $\G = O_N^+,S_N^+,S_N^{'+}$, see Remark \ref{rem.modifiedthesame}) and the $2$-Fuss-Catalan-algebra $\FC_2(\infty,\delta)$ (when $\G = H_N^+$) introduced by Bisch and Jones \cite{BiJo95} in their analysis of intermediate subfactors, see also \cite{La01}. 
The extremal traces of the infinite Temperley-Lieb algebra have been classified in the PhD thesis of A. Wassermann \cite{Was81} with each extremal trace corresponding to a parameter value $\lambda \in [1/2,1]$. Wassermann also pointed out that this classification problem is equivalent to the computation of the \emph{minimal boundary} in the sense of Vershik and Kerov of the \emph{semi-Pascal graph} (drawn below in Figure \ref{fig.semiPascalintro}). Loosely speaking, this graph describes how the irreducible components of the algebras $\TL(0,\delta) \subset \TL(1,\delta)\subset \TL(2,\delta) \subset \dots$ are nested inside each other. Its minimal boundary consists of probability measures on the graph that satisfy a Gibbs-type consistency condition and are extremal among all probability measures satisfying this condition, see Subsection \ref{subsec.minimalboundary}.
\begin{figure}[h!]
\begin{center}
\includegraphics[scale=0.25]{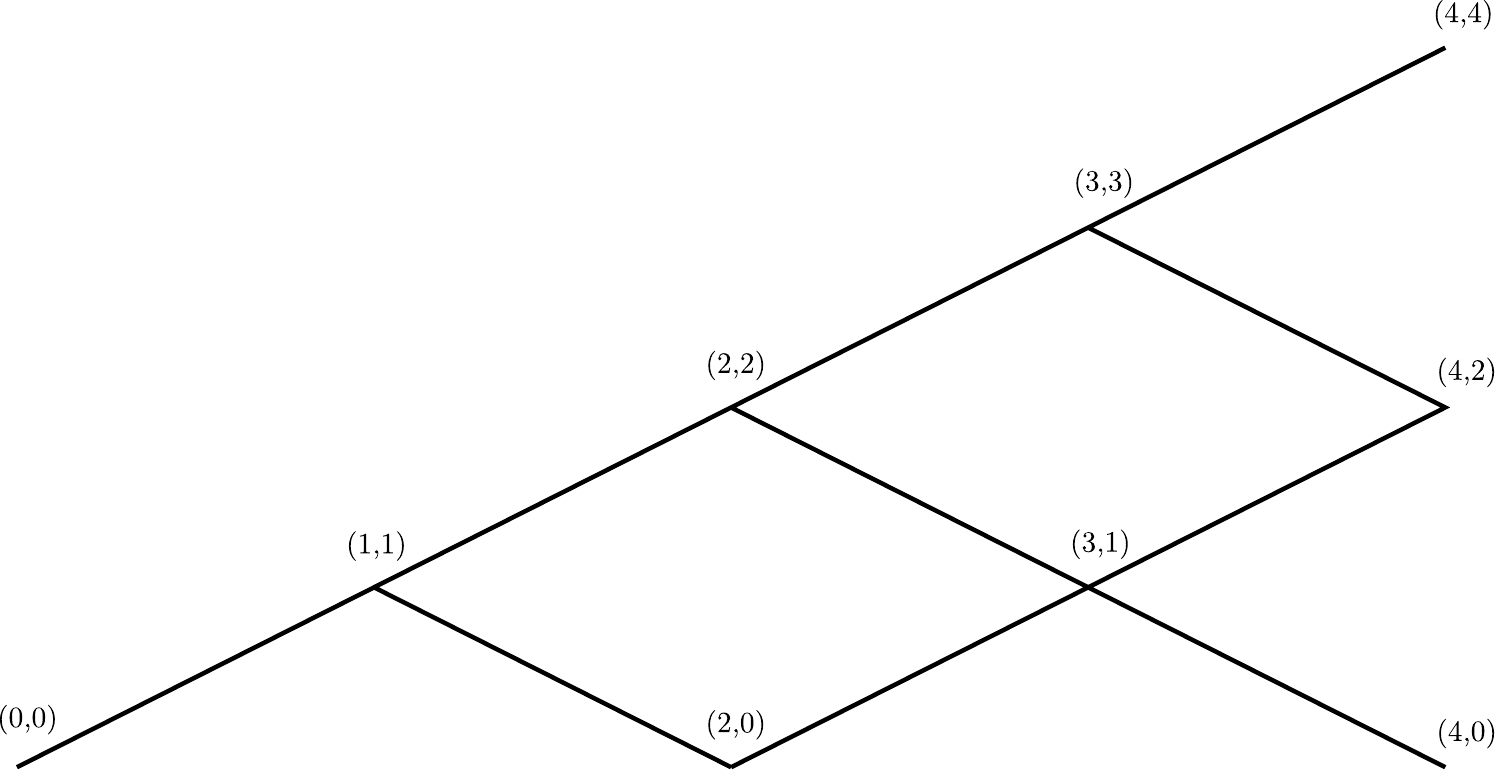}
\end{center}
\caption{\label{fig.semiPascalintro} The first five levels of the semi-Pascal graph.}
\end{figure}

In fact, the relationship between traces on the infinite Temperley-Lieb algebra and boundary measures on the semi-Pascal graph is an instance of a more general correspondence between traces on inductive limits of finite dimensional $C^*$-algebras and boundary measures on their \emph{branching graphs} or \emph{Bratteli diagrams}, which is well-known to experts and will be fleshed out in Section \ref{sec.branchinggraphs}. By combining this correspondence with results due to Vershik and Kerov \cite{VK81}, Vershik and Malyutin \cite{VM15}, Vershik and Nikitin \cite{VN06} and Wassermann \cite{Was81}, we develop a general methodology to classify traces on infinite diagram algebras. We will describe this methodology thoroughly in Section \ref{sec.generalmethodology}, following this introduction. Of particular importance to our approach is the process of \emph{pascalization} of an $\N$-graded graph, see Section \ref{sec.branchinggraphs}. In a nutshell, the branching graphs associated to diagram algebras all arise  through this process of pascalization from significantly smaller and more simple graphs. This is of great use when computing relevant combinatorial quantities such as path counts between vertices. The technique of pascalization, though not the terminology, is also well-known in subfactor theory, see e.g. \cite{JS97}, where the smaller graph is called the \emph{principal graph} of the larger one.

Concretely, for the different examples of infinite dimensional diagram algebras derived from categories of noncrossing partitions, we arrive at the following results which in some cases go beyond a mere classification of traces:
\paragraph*{Infinite Temperley-Lieb algebra.} We recap Wassermann's trace classification, rephrase it as a classification of \emph{central random ballot paths} and make his result much more explicit in this probabilistically natural context. A ballot path (the terminology originates from Bertrand's famous ballot problem) is a lattice path on $\N^2$ starting at $(0,0)$ that is allowed to take steps $(1,1)$ and $(1,-1)$. Consequently, a random ballot path is a probability measure $\mu$ on the space of infinite ballot paths which we require to satisfy the following memory-loss condition (called centrality): given the event that the random path passes through a point $(n,k)$ on the lattice, every path from $(0,0)$ to $(n,k)$ is chosen with the same probability. Wassermann's result then implies that every central random ballot path is the mixing of Markov chains $M_{\lambda}, \ \lambda \in [1/2,1] $ (Theorem \ref{thm.randomballot}). Expanding Wassermann's work, we provide explicit transition probabilities for these Markov chains in Section \ref{sec.templieb}.

\paragraph*{Infinite Motzkin algebra.} In Section \ref{sec.freebistochasic}, we address the trace classification problem for the direct limit algebra $A_{(\cB^+,\delta)}(\infty)$ associated to the free bistochastic quantum groups $B_N^+$ and $B_N^{'+}$. Again, the problem admits an equivalent formulation in terms of random lattice paths and the type of lattice paths appearing is known in the combinatorial literature as \emph{Motzkin paths}. These resemble ballot paths but allow for additional even level steps $(1,0)$. We establish this connection by an analysis of the representations of the algebras $A_{(\cB^+,\delta)}(k), \ k \geq 0$. In particular, this yields a description of the branching graph of the sequence $A_{(\cB^+,\delta)}(0) \subset A_{(\cB^+,\delta)}(1) \subset \dots$ (see also \cite{BH14} for another approach). By combining our reinterpretation of the branching graph of the infinite Motzkin algebra with lattice path counting results from the algebraic combinatorics literature (our main reference here is \cite{Kra15}), we obtain explicit formulas for the dimensions of all irreducible representations of $A_{(\cB^+,\delta)}(k)$. Next, we prove in Theorem \ref{thm.randomMotzkinclass} that every central random Motzkin path is a mixture of Markov chains $M_{(\lambda_1,\lambda_2)}$ indexed by parameters $\lambda_1 \leq \lambda_2$ with $0 \leq \lambda_1 + \lambda_2 \leq 1$. By virtue of the correspondence between traces on $A_{(\cB^+,\delta)}(\infty)$, this also yields a classification of traces on this algebra, see Corollary \ref{cor.Motzkintraceclassification}. Lastly, we provide formulas (in terms of $\lambda_1, \lambda_2$) for the probability that the chain $M_{(\lambda_1,\lambda_2)}$ returns to the root after a given number of steps and a recursion that determines the transition probabilities of this chain uniquely.

\paragraph*{Infinite diagram algebra $A_{(\cB^{\#+},\delta)}(\infty)$ and infinite Fuss-Catalan algebra.} In Section \ref{sec.freelymodified}, we determine the representation theory of the finite dimensional  algebras $A_{(\cB^{\#+},\delta)}(0) \subset A_{(\cB^{\#+},\delta)}(1) \subset \dots $ associated to the freely modified bistochastic quantum group $B_N^{\#+}$. We show that the branching graph of this tower of algebras is the pascalization of the \emph{derooted Fibonacci tree}, see Figure \ref{fig.Fibintro}. In fact, this case closely resembles the last remaining case, the infinite $2$-\emph{Fuss-Catalan algebra} $\FC_2(\infty,\delta)$ which is associated to the free hyperoctahedral group $H_N^+$ and whose branching graph is the pascalization of the (full) Fibonacci tree \cite{BiJo95}. Due to this similarity, we will discuss dimension formulas and the trace classification problem for both infinite dimensional diagram algebras jointly in Section \ref{sec.FussCat}. 

\begin{figure}[h!]
\begin{center}
\includegraphics[scale=0.4]{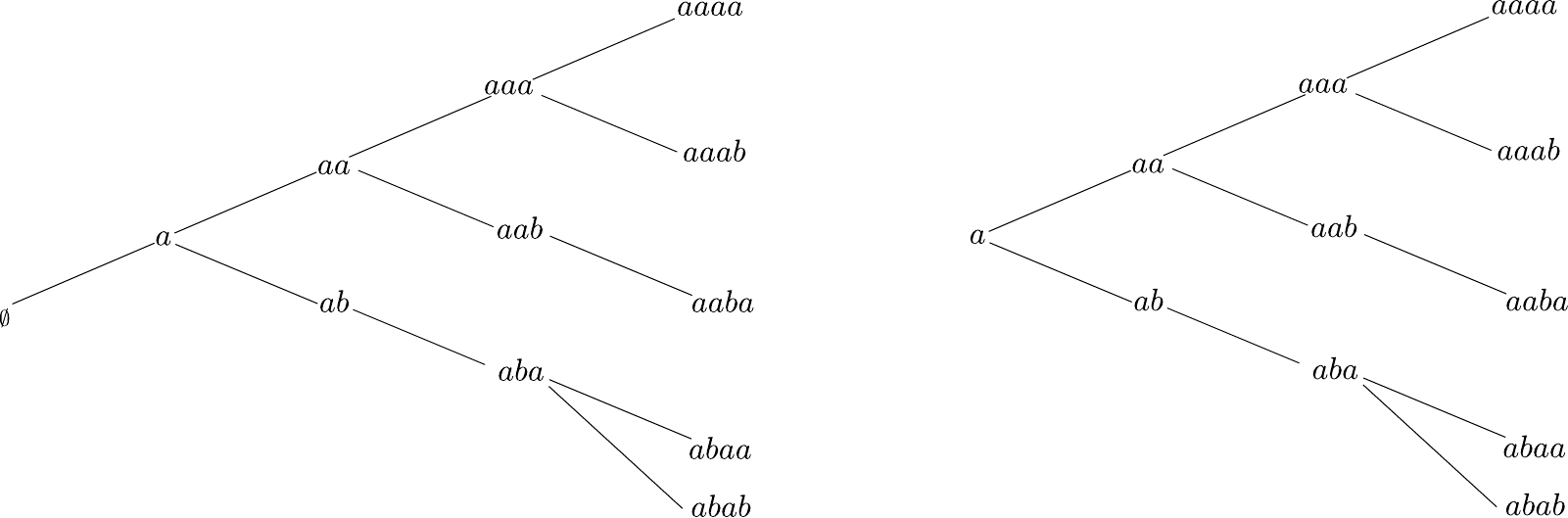}
\end{center}
\caption{\label{fig.Fibintro} The first levels of the Fibonacci tree $\FT$ and the derooted Fibonacci tree $\FT^*$.}
\end{figure}

In both cases, the description of the branching graph as the pascalization of a version of the Fibonacci tree translates the trace classification problem into an equivalent probabilistic problem of independent interest: the computation of the \emph{exit boundary for random walks} on the Fibonacci tree (to follow the terminology of \cite{VM15}). In \cite{VM15}, this problem was solved for homogeneous trees by making clever use of their symmetric nature. For the less symmetric Fibonacci tree, the combinatorics of this problem become more involved and as a result, the computation of the exit boundary becomes more subtle. For a discussion on exit boundaries in a different setting, see also \cite{VM18}, where the exit boundary is called the \emph{absolute}.

In essence, we prove that every element of the exit boundary, that is to say every ergodic central Markov chain on the pascalization $\cP(\FT)$ of the Fibonacci tree $\FT$ is the lift of a transient random walk $S_{(t,\eta)}$  on $\FT$. This random walk is uniquely determined by a \emph{structure constant} $\eta \in [0,4/27]$, which describes the probability to cross edges once in both directions, and an \emph{end} $t \in \partial \FT$ of  the tree $\FT$ to which $S_{(t,\eta)}$ converges almost surely. 
In summary, we obtain the following theorem.
\begin{letterthm}[see also Theorem \ref{thm.classificationFC}] \label{thmA}
The ergodic central Markov measures on $\cP(\FT)$ are the lifts of random walks $S_{(t,\eta)}$ with 
\[(t,\eta) \in \partial \FT \times [0,4/27].\]
\end{letterthm}

The random walks $S_{(t,\eta)}$ and their lifts to the path space of $\cP(\FT)$ are defined in Section \ref{sec.deFintheorem}. Their transition probabilities admit nice formulas on all edges that do not lie on the specified end $t$ and these formulas involve the evaluation at the value $\eta$ of the generating function $G(z) = \sum_{n=0}^{\infty} C^2_n z^n$ of the \emph{Fuss-Catalan numbers } $(C^2_n)_{n \geq 0} $. This explains the upper bound $4/27$ for the parameter $\eta$ which is exactly the radius of convergence of  the power series $G(z)$. More generally, we define these random walks for general $s$-Fuss-Catalan trees with $s \geq 2$ in which case the critical value for the structure constant becomes $ \tfrac{s^{s}}{(s+1)^{s+1}}$. 

By the general correspondence between traces on infinite diagram algebras and the minimal boundary of their branching graphs, Theorem \ref{thmA} thus results in a classification of traces on $\FC_2(\infty,\delta)$. The one-to-one correspondence between the random walks $S_{(t,\eta)}$ and the extremal traces on $\FC_2(\infty,\delta)$ is explicitly spelled out in Corollary \ref{cor.classFC}.
An important ingredient in the proof of Theorem \ref{thmA} and an interesting result on its own, is the law of large numbers (Theorem \ref{thm.llnexit}) for the exit times $N_k$ at the vertex $t_k$, the number of steps after which $S_{(t,\eta)}$ leaves $t_k$ forever.
The proof of Theorem \ref{thmA} can be copied almost verbatim for the derooted Fibonacci tree with only a slight adaptation of the formula for the random walks $\tilde{S}_{(t,\eta)}$, see Subsection \ref{subsubsec.randomwalks}. 

By combining the results explained in the previous paragraphs, we obtain a classification of traces on all infinite diagram algebras associated to free partition quantum groups. We summarize the main results of this work in the following table.

\begin{center}
\begin{tabular}{|c|c|c|c|c|}
\hline
\makecell{Infinite \\ diagram \\ algebra} & \makecell{Quantum group \\ at $\delta = \sqrt{N}$ } & \makecell{Branching graph \\ is pascalization \\ of} & \makecell{Extremal traces \\ parametrized \\ by} & \makecell{Probabilistic \\ description} \\
\hline \hline
\makecell{Temperley-Lieb \\ algebra \\ $\TL(\infty,\delta)$} & \makecell{ $O_N^+,S_N^+, S_N^{'+}$ \\ \cite{BS09} \\ Remark \ref{rem.modifiedthesame}} & \makecell{ half-line $\N$ \\ \cite{Jo83} \cite{Was81} \\ Section \ref{sec.templieb} } & \makecell{$ \lambda \in [1/2,1]$ \\ \cite{Was81} \\ Section \ref{sec.templieb} } & \makecell{random ballot \\ paths \\ Subsection \ref{subsubsec.Temptraces}} \\
\hline 
\makecell{Motzkin \\ algebra \\ $\Mo(\infty,\delta)$} & \makecell{ $B_N^+, B_N^{'+}$ \\ \cite{BS09} \\ Remark \ref{rem.modifiedthesame}} & \makecell{ladder \\ Remark \ref{rem.ladder} } & \makecell{$\lambda_1,\lambda_2$ s.t. \\ $0 \leq \lambda_2 \leq \lambda_1 \leq 1$, \\ $0 \leq \lambda_1 + \lambda_2 \leq 1$, \\ Corollary \ref{cor.Motzkintraceclassification} } & \makecell{random Motzkin \\ paths \\ Subsection \ref{subsubsec.tracesMotzkin}} \\
\hline
\makecell{$2$-Fuss-Catalan \\ algebra \\ $\FC_2(\infty,\delta)$} & \makecell{ $H_N^+$ \\  \cite{BBC07} \cite{BS09} } & \makecell{Fibonacci tree  \\ $\FT$ \\  \cite{BiJo95}} & \makecell{$(t,\eta) \in$ \\ $ \partial \FT \times [0,4/27]$ \\ Corollary \ref{cor.classFC} } & \makecell{random walks \\ on $\FT$ \\ Subsection \ref{subsubsec.randomwalks}}  \\
\hline
$A_{(\cB^{\#+},\delta)}(\infty)$ & \makecell{$B_N^{\#+}$ \\ \cite{BS09} } & \makecell{derooted \\ Fibonacci tree $\FT^*$ \\ Corollary \ref{cor.pascofderootedFT} } & \makecell{$(t,\eta)\in$ \\ $\partial \FT \times [0,4/27]$ \\ Corollary \ref{cor.classFC} } & \makecell{random walks \\ on $\FT^*$ \\ Subsection \ref{subsubsec.randomwalks}}  \\
\hline
\end{tabular}
\end{center}


There are several other classes of infinite diagram algebras appearing in the work of Banica and Speicher \cite{BS09}, Weber \cite{We13} and others, e.g. \cite{RaWe16} \cite{FlP18} \cite{VV19}, for which the trace classification problem is still open and for which in many cases even the branching graphs are unknown. In the companion article \cite{Wa20}, we will classify the extremal traces on the next prominent class of infinite diagram algebras of \cite{BS09}, namely those containing the crossing partition. Originally, these algebras arose in the context of Schur-Weyl duality for compact groups \cite{W46}, and once more, it will be advantageous to think about the involved branching graphs as pascalizations of smaller principal graphs. In contrast to the principle graphs in this article, however, the underlying principle graphs in \cite{Wa20} will be variations of the famous Young graph, the branching graph associated to the infinite symmetric group $S_{\infty}$.


\paragraph*{Acknowledgements}
I am grateful to A. Bufetov, E. Peltola and P. Tarrago for stimulating discussions on the subject matter. I would also like to thank M. Weber for his thoughful comments on a previous draft of this article. Lastly, I would like to thank the anonymous referee whose suggestions have significantly improved this article. This work was supported by the Deutsche Forschungsgemeinschaft (DFG, German Research Foundation) under Germany’s Excellence Strategy – EXC 2047 “Hausdorff Center for Mathematics”.

\section{General Methodology} \label{sec.generalmethodology}

In this section, we will describe a rough general recipe for classifying traces on infinite diagram algebras.

\subparagraph*{Step 1:} Given a tower of finite dimensional diagram algebras $A_1 \subset A_2 \subset \dots$, the first step is to determine the irreducible representations of the algebras $A_k$ and their decomposition rules when restricted to the subalgebra $A_{k-1}$. A general approach to this tailored to diagram algebras was presented by Freslon and Weber \cite{FrWe16}. The data of irreducible representations plus restriction rules is then graphically encoded in the branching graph.  

\subparagraph*{Step 2:} The next step is to determine the principle graph of the diagram algebra, i.e. the smaller graph from which the branching graph can be obtained by pascalization. Although there is an explicit recipe to compute the principle graph from the branching, see e.g. \cite{JS97}, in our case it is often easy to guess it and then check that the correct branching graph is obtained after pascalization. \\  

For the Temperley-Lieb algebras, the representations, the branching graph and the principal graph are known \cite{Jo83} and we describe them in Subsection \ref{subsubsec.semi-Pascal}. For the Motzkin algebras, we compute the branching graph and read of the principal graph in Subsection \ref{subsubsec.freebistochbranching}. For the algebras $A_{(\cB^{\#+},\delta)}(k)$, Steps 1 and 2 are carried out in Subsection \ref{sec.freelymodified}. For the Fuss-Catalan algebras, branching and principal graph are once again known \cite{BiJo95} and we summarize these results in Subsection \ref{sec.FussCat}.

\subparagraph*{Step 3:} By Theorem \ref{thm.boundarytraces} in the preliminaries, the classification of traces on an infinite diagram algebra is equivalent to the classification of central measures on its associated branching graph, see Definition \ref{def.centralmeasure}. In order to classify these measures, we want to apply the Vershik-Kerov ergodic method (Theorem \ref{thm.ergodicmethod}) in Step 4, which relies on path counting formulas on the branching graph. Hence Step 3 is to find these path counting formulas. For the Temperley-Lieb and the Motzkin algebras, we infer them from results in the lattice path literature, see Theorems \ref{thm.numberballotpaths} and \ref{thm.Motzkinnumbers}. For the Fuss-Catalan algebras and the algebras $A_{(\cB^{\#+},\delta)}(k)$, the relevant path counting formulas follow form the work of Landau \cite{La01} and are presented in Subsection \ref{subsec.walkcount}.

\subparagraph*{Step 4:} To obtain candidates for the ergodic central measures, we now apply the Vershik-Kerov ergodic method, Theorem \ref{thm.ergodicmethod}, using the path counting formulas computed and collected in Step 3. According to the ergodic method, finding candidate measures comes down to determining the limit behaviours of fractions of path counting numbers. These computations have to be done case-by-case.

 \subparagraph*{Step 5:} The last step is to show that the candidate measures obtained in the Step 4 are in fact ergodic. To do so, there is no general method available and this once again has to be approached on a case-by-case basis. \\
 
For the Temperley-Lieb algebras, Step 4 and 5 have been carried out in Wassermann's thesis \cite{Was81} and we summarize the results in Subsection \ref{subsubsec.Temptraces}. For the Motzkin algebras, Steps 4 and 5 are executed in Subsection \ref{subsubsec.tracesMotzkin}, where we essentially imitate Wassermann's algebraic approach to Step 5. The remaining two cases (Fuss-Catalan and $A_{(\cB^{\#+},\delta)}(k)$) require by far the most work, see Subsections \ref{sec.deFintheorem}, \ref{subsubsec.randomwalks} and \ref{subsubsec.LLN}. Our approach here is probabilistic and the ergodic measures will be described as the laws of random walks on the Fibonacci tree. Step 5 then will follow from a law of large number for the exit times of these random walks proved in Subsection \ref{subsubsec.LLN}. 

\section{Preliminaries} \label{sec.preliminaries}

\subsection{Lattice paths} \label{subsec.latticepaths}
We will now introduce the lattice path models whose random behaviour we would like to study in the first part of this article. The literature on the combinatorics of lattice paths is vast and we will mostly refer to the survey \cite{Kra15}.

Generally speaking, a lattice path in $\Z^2$ is a (finite or infinite) sequence of points in the lattice $\Z^2$ that is only allowed to change according to an a priori specified set of steps (a step is the difference between two consecutive elements of the sequence). The lattice paths we will focus on in this article, will not be allowed to cross the $x$- and $y$-axes and thus will be restricted to the upper right quadrant $\N \times \N$, where we include $0$ in $\N$ by convention. We will call the tuples $(1,1), (1,-1), (1,0)$ an \emph{up-step, down-step} and \emph{level-step} respectively.

\begin{definition}
\begin{itemize}
\item[(1)] A \emph{ballot path} from $(a,b) \in \N \times \N$ to $(c,d) \in \N \times \N$ is a lattice path starting at $(a,b)$ and ending at $(c,d)$ that takes only up- or down-steps. We will also allow for infinite ballot paths, that is to say ballot paths with a specified starting point performing an infinite number of steps.
\item[(2)] A \emph{Motzkin path} from $(a,b) \in \N \times \N$ to $(c,d) \in \N \times \N$ is a lattice path starting at $(a,b)$ and ending at $(c,d)$ that is allowed to take up-steps, down-steps and level-steps. Again, we will also be interested in infinite Motzkin paths for which only a starting point is specified and which perform infinitely many up-, down-, or level-steps.
\end{itemize}
\end{definition}
\begin{figure}[h!]
\centering
\begin{subfigure}[b]{0.4\textwidth}
\begin{tikzpicture}[scale=0.6]
   \tkzInit[xmax=8,ymax=8,xmin=0,ymin=0]
   \tkzGrid
   \tkzAxeXY
   \draw[ thick,-latex,red] (0,0) -- (1,1) node[anchor=south west] {};
   \draw[ thick,-latex,red] (1,1) -- (2,2) node[anchor=south west] {};
   \draw[ thick,-latex,red] (2,2) -- (3,1) node[anchor=south west] {};
   \draw[ thick,-latex,red] (3,1) -- (4,2) node[anchor=south west] {};
   \draw[ thick,-latex,red] (4,2) -- (5,3) node[anchor=south west] {};
   \draw[ thick,-latex,red] (5,3) -- (6,2) node[anchor=south west] {};
   \draw[ thick,-latex,red] (6,2) -- (7,1) node[anchor=south west] {};
   \draw[ thick,-latex,red] (7,1) -- (8,2) node[anchor=south west] {};
  \end{tikzpicture}
  \subcaption{A ballot path from $(0,0)$ to $(8,2)$.}
\end{subfigure}
\hfill
\begin{subfigure}[b]{0.4\textwidth}
\begin{tikzpicture}[scale=0.6]
   \tkzInit[xmax=8,ymax=8,xmin=0,ymin=0]
   \tkzGrid
   \tkzAxeXY
   \draw[ thick,-latex,red] (0,0) -- (1,1) node[anchor=south west] {};
   \draw[ thick,-latex,red] (1,1) -- (2,1) node[anchor=south west] {};
   \draw[ thick,-latex,red] (2,1) -- (3,0) node[anchor=south west] {};
   \draw[ thick,-latex,red] (3,0) -- (4,1) node[anchor=south west] {};
   \draw[ thick,-latex,red] (4,1) -- (5,2) node[anchor=south west] {};
   \draw[ thick,-latex,red] (5,2) -- (6,2) node[anchor=south west] {};
   \draw[ thick,-latex,red] (6,2) -- (7,3) node[anchor=south west] {};
   \draw[ thick,-latex,red] (7,3) -- (8,2) node[anchor=south west] {};
  \end{tikzpicture}
  \subcaption{A Motzkin path from $(0,0)$ to $(8,2)$.}
  \end{subfigure} 
  \caption{Ballot and Motzkin paths.}
\label{fig.paths}  
\end{figure}
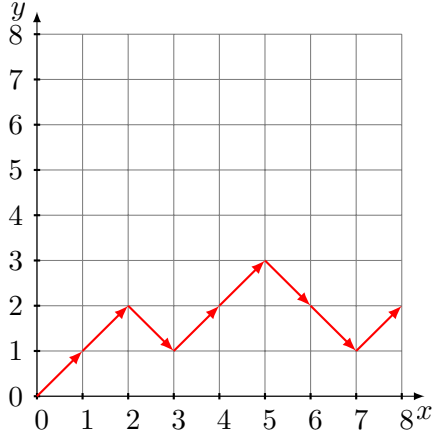
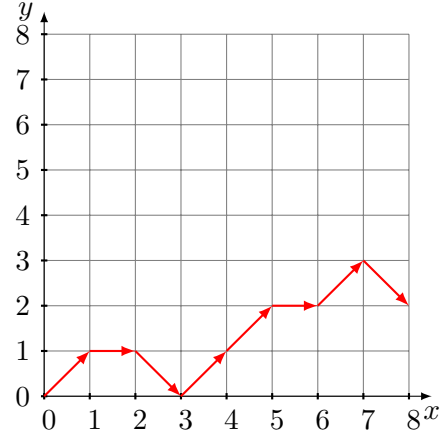

\begin{remark} \label{rem.paths}
\begin{itemize}
\item[(1)] Ballot paths that start at the root $(0,0)$ and end in a point on the $x$-axis are commonly referred to as \emph{Dyck paths}.
\item[(2)] It is also common to define ballot paths somewhat differently as lattice paths that take steps $(0,1)$ and $(1,0)$, start above the diagonal $x=y$ and do not cross it. Clearly, a 45 degree clockwise rotation translates this setting into ours. 

\item[(3)] Ballot paths can also be identified with walks on the one sided line (i.e. the graph with vertex set $\N$ and edges between neighboring non-negative integers). Similarly, Motzkin paths are in one-to-one correspondence with trajectories of a \emph{lazy walker} on $\N$ that is to say, one that is allowed to move to the left or to the right or to rest in place. 
\end{itemize}
\end{remark}

We will denote the set of ballot paths (resp. Motzkin paths) from $(a,b)$ to $(c,d)$ by $\Ba((a,b),(c,d))$ (resp. $\Mo((a,b),(c,d))$). For the following result, see \cite[Theorem 10.3.1]{Kra15}. Note that there, this theorem is formulated in the setting mentioned in Remark \ref{rem.paths}.

\begin{theorem} \label{thm.numberballotpaths}
The number of ballot paths from $(a,b)$ to $(c,d)$ is
\begin{align*}
|\Ba((a,b),(c,d))| = \binom{c-a}{(c-a-d+b)/2} - \binom{c-a}{(c-a+d+b+2)/2}
\end{align*}
where by convention a binomial coefficient is $0$ if its bottom component is not an integer. 
\end{theorem}

The analogous result for Motzkin paths is the following (see \cite[Theorem 10.6.1]{Kra15}).

\begin{theorem} \label{thm.Motzkinnumbers}
The number of Motzkin paths from $(a,b)$ to $(c,d)$ is 
\begin{align*}
|\Mo((a,b),(c,d))| = \sum_{k=0}^{c-a} \binom{c-a}{k} \bigg( \binom{c-a-k}{(c-a-k+d-b)/2} - \binom{c-a-k}{(c-a-k+b+d+2)/2} \bigg).
\end{align*}
\end{theorem}

\begin{remark}
\begin{enumerate}
\item The numbers 
\begin{align*}
m_n := |\Mo((0,0),(n,0))| = \sum_{l=0}^{\lfloor \tfrac{n}{2} \rfloor} \frac{n!}{(n-2l)!(l+1)!l!}
\end{align*}  
are known in the literature as the \emph{Motzkin numbers} (sequence A001006 in the OEIS) and they satisfy the identity $m_n = \sum_{k=0}^{\lfloor \tfrac{n}{2} \rfloor} \binom{n}{2k} C_k$, where $C_k = \frac{1}{k+1} \binom{2k}{k}$ are the Catalan numbers.
\end{enumerate}
\end{remark}





\subsection{Branching graphs} \label{sec.branchinggraphs}

A \emph{branching graph} or \emph{Bratteli diagram} $\Gamma = (V,E)$ is a locally finite $\N$-graded rooted graph, that is to say, a graph whose set of vertices $V$ can be subdivided into levels $V=\sqcup_{n \in \N} V_n$ with $V_0 = \{ \emptyset \}$ containing the root $\emptyset$, and whose edges can only connect vertices of adjacent levels. All concrete examples of branching graphs in this article will have finite level sets, i.e. $|V_n| < \infty$ for all $n \in \N$.

Branching graphs are typically used to encode the induction/restriction rules of inductive sequences $G_0= \{e \} \hookrightarrow G_1 \hookrightarrow G_2 \hookrightarrow \dots$ of finite groups  of more generally of inductive sequences $A_0 = \C \hookrightarrow A_1 \hookrightarrow A_2 \hookrightarrow \dots $ of semisimple $*$-algebras. To recall how this graph is obtained, let us fix the following notation. If $A \subset B$ is an inclusion of semisimple algebras and $M$ is a left $B$-module, we denote by $M^{\downarrow}$ the left $A$-module obtained by restricting the action of $B$ to $A$.

\begin{definition} \label{def.indrestrgraph}
Let $A_0 = \C \hookrightarrow A_1 \hookrightarrow A_2 \hookrightarrow \dots $ be an inductive sequence of semisimple $*$-algebras. The \emph{induction/restriction graph} of $(A_n)_{n \geq 0}$ is the branching graph $\Gamma = (V,E)$ defined in the following way.
\begin{itemize}
\item The $n$-th level vertex set $V_n$ is the set of (equivalence classes of) simple modules of $A_n$.
\item There are exactly $\mult(v, w^{\downarrow})$ edges between the simple $A_n$-module $v \in V_n$ and the simple $A_{n+1}$-module $w \in V_{n+1}$ where  $\mult(v, w^{\downarrow})$ denotes the multiplicity of $v$ in the decomposition of $w^{\downarrow}$ into simple $A_n$-modules.
\end{itemize}
\end{definition}

\begin{remark} \label{rem.dimensions}
The branching graph of a tower $A_0 = \C \hookrightarrow A_1 \hookrightarrow A_2 \hookrightarrow \dots $ produces a useful combinatorial interpretation of the dimensions of simple $A_n$-modules for all $n \in \N$. The number of paths on the graph starting at a simple $A_m$-module $v \in V_m$ and ending at a simple $A_n$-module $w$, $n > m$, is exactly the multiplicity of $v$ in the decomposition into simple summands of $w$ considered as an $A_m$-module. In particular, when $v = \emptyset$ is the root vertex corresponding to the one-dimensional trivial simple module $\cong \C$ of $A_0 = \C$, the number of paths from $\emptyset$ to $w$ encodes the multiplicity of $\C$ in the decomposition of $w$ as a $\C$-module. Thus, the number of rooted paths ending at $w$ is exactly $\dim(w)$.
\end{remark}

\begin{notation} \label{not.dimensions}
For arbitrary branching graphs, we will make use of the notation $\dim(v,w)$ for the number of paths leading from $v \in V_m$ to $w \in V_n, \ m > n$ and we will shorten $\dim(\emptyset,w)$ to $\dim(w)$.
\end{notation}

In most of the examples that we will consider in this article, the branching graphs in question are generated by smaller ones through a process dubbed \emph{pascalization} in \cite{VN06}.

\begin{definition}
Let $\Gamma = (V,E)$ be a branching graph. The \emph{pascalized graph} $\cP(\Gamma) = (\cP(V), \cP(E))$ of $\Gamma$ is defined in the following way.
\begin{itemize}
\item The vertex set of level $n$ is $\cP(V)_n = \{ (n,v) \ ; \ v \in V_k, \ k \leq n, \ k \equiv n \mod 2 \}$. We will regularly make use of the projection $\pi: \cP(V) \to V, \ (n,v) \mapsto v$.
\item In $\cP(\Gamma)$, the number of edges between $(n,v) \in \cP(V)_n$ and $(n+1,w) \in \cP(V)_{n+1}$ is the number of edges between $v \in V_k$ and $w \in V_l$ in the original graph $\Gamma$. In particular, this number can only be non-zero if $v$ and $w$ are  on neighbouring levels of $\Gamma$, that is $|k-l|=1$.
\end{itemize}
\end{definition}

\begin{example} \label{ex.pascalization}
\begin{enumerate}
\item[(1)] If we consider the set of integers $\Z$ as a branching graph with $\N$-grading $n \mapsto |n|$ and edges connecting neighboring integers, then $\cP(\Z)$ is the \emph{Pascal graph}, whence the name pascalization, see Figure \ref{fig.pascalization}.
\item[(2)] Similarly if we consider the set of non-negative integers $\N$ as a branching graph, its pascalization is the semi-Pascal graph $\cP(\N)$. This branching graph is the Bratteli diagram of the inductive family of Temperley-Lieb algebras at a generic parameter, see e.g \cite{Jo83} \cite{GHJ89}. We will discuss this example in more detail in Section \ref{sec.free}.
\item[(3)] More generally, the procedure of pascalization will be familiar to readers with a background in \emph{subfactor theory} although the name pascalization is not commonly used there. It is exactly the method by which one constructs the Bratteli diagram of irreducible bimodules of a subfactor $N \subset M$ from its principal graph, see e.g. \cite{JS97}.    
\end{enumerate}
\end{example}

\begin{figure}[h!]
\begin{center}
\includegraphics[scale=0.3]{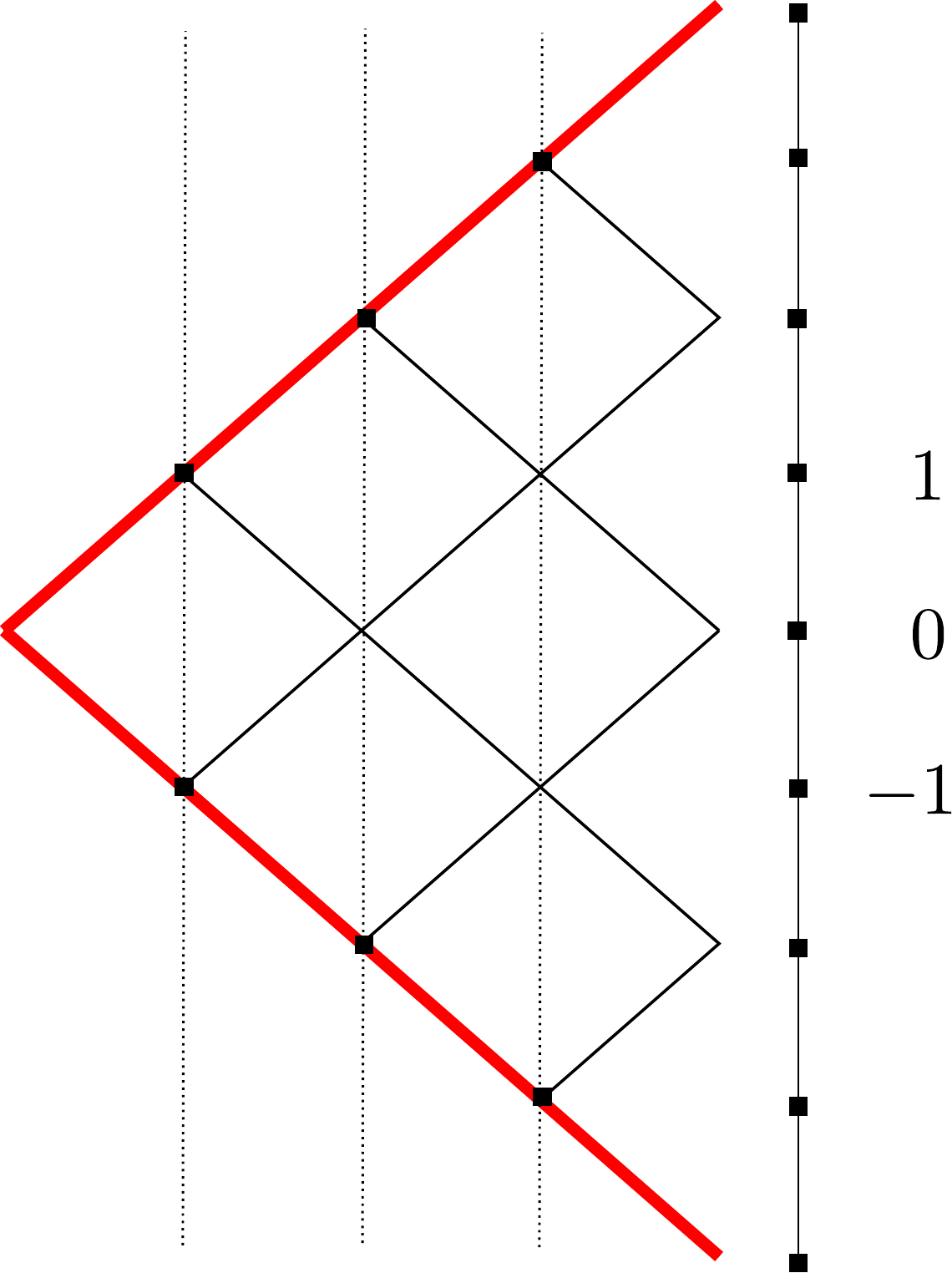}
\end{center}
\caption{\label{fig.pascalization} The principal graph $\Z$ (in red) and its pascalization, the Pascal graph. Every vertex on level $n$ that does not lie on the principal graph is obtained as a reflection of a vertex on level $n-2$ along the vertical line passing through level $n-1$.}
\end{figure}

\subsection{The minimal boundary of a branching graph} \label{subsec.minimalboundary}
In this section, we will recall the fundamentals on the minimal boundary of branching graphs. The methods explained here go back to the work of Vershik and Kerov \cite{VK81} \cite{VK82} and they are heavily used in asymptotic representation theory of inductive limit groups such as $S(\infty)$ or $U(\infty)$, see e.g. the excellent book \cite{BO16}. 

For a Bratteli diagram $\Gamma = (V,E)$, let us denote by $(\Omega,\cF):= (\Omega_{\Gamma}, \cF_{\Gamma})$ the space $\Omega \subset \prod_{n\geq 0} V_n$ of infinite paths on $\Gamma$ equipped with the restriction $\cF$ of the product $\sigma$-algebra to $\Omega$. The $n$-th coordinate projection will be denoted by $X_n: \Omega \to V_n, \ X_n((v_k)_{ \geq 1}) = v_n$. We will say that two infinite paths $x,y \in \Omega$ are \emph{tail equivalent} ($x \sim y$) if they coincide up to finitely many steps. The corresponding equivalence relation, the \emph{tail relation}, is Borel and will be referred to by $\cT$.

\begin{definition} \label{def.centralmeasure}
A probability measure $\nu$ on $(\Omega,\cF)$ will be called \emph{central} if it is invariant under $\cT$ or equivalently if for all $k\geq 0$, $v \in V_k$, and every path $v_0=\emptyset, v_1,\dots, x_k=v$ from the root to $v$, we have 
\[\nu(\{x \in \Omega \ ; \ x_1 = v_1 ,\dots, x_k= v \}) = \frac{\nu(\{ x_k = v \})}{\dim v}. \]
The central measure $\nu$ will be called \emph{ergodic} if it is ergodic w.r.t. $\cT$, that is to say if the only $\cT$-invariant subsets of $\Omega$ are those that have $\nu$-measure $0$ or $1$.
\end{definition}  

The topological space of ergodic central measures on $(\Omega,\cF)$ (equipped with the weak topology) is known as the \emph{minimal boundary} of $\Gamma$ and we will denote it by $\partial \Gamma$. The set $\cM_c(\Gamma)$ of central probability measures on $(\Omega,\cF)$ forms a Choquet simplex (w.r.t. the weak topology) whose extremal points are given by the boundary. An arbitrary central probability $\nu'$ can therefore always be decomposed into its ergodic components by Choquet's theorem, that is to say, there exists a (unique up to nullsets) probability measure $\mu$ on $\partial \Gamma$ such that $\nu'(A) = \int_{\partial \Gamma} \nu(A) \ d \mu(\nu)$.

Observe also that the marginal measures 
\[ \nu_k(v):= \nu(\{x \in \Omega \ ; x_k= v \}), \qquad (v \in V_k, \ k \geq 0),
\] of a central measure $\nu$ form a \emph{coherent system}, i.e. they satisfy the consistency condition
\begin{align*}
\nu_k(v) = \sum_{w: v \nearrow w} \frac{\dim(v)}{\dim(w)} \nu_{k+1}(w) \qquad \text{for all } \qquad v \in V_k, \ k \geq 0,  
\end{align*}
where the sum runs over all vertices $w \in V_{k+1}$ such that there is an edge from $v$ to $w$. Conversely, any coherent system of probability measures $(\nu_k)_{k \geq 0}$, where $\nu_k$ lives on $V_k$, uniquely determines a central measure $\nu$ on $(\Omega,\cF)$ with the family $(\nu_k)_{k \geq 0}$ as its marginals. This essentially follows by an application of the Kolmogorov extension theorem, see \cite[Chapter 7]{BO16} for a detailed argument. The one-to-one correspondence between coherent systems and central measures also interacts nicely with the weak topology:

\begin{lemma} \label{lem.weakcoherent}
Let $(\nu^{(i)})_i$ be a net of central measures on $(\Omega,\cF)$ and let $(\nu^{(i)}_k)_{k \geq 0}$ be the coherent system of $\nu^{(i)}$. Then $(\nu^{(i)})_i$ converges weakly to a central measure $\nu$ with coherent system $(\nu_k)_{k \geq 0}$ if and only if for every $k \geq 0, \  v\in V_k$, we have $\nu^{(i)}_k(v) \to \nu(v)$.
\end{lemma}

\begin{proof}
Since this result is once again well known to experts, we will only sketch the proof. First as the sets $V_k$ are finite, one observes that by Tychonoff's theorem, $\Omega$ is a compact topological space equipped with the restriction of the product topology. The first implication which states that weak convergence of central measures implies pointwise convergence of coherent systems is immediate. For the converse direction, consider the family of characteristic functions $(\mathbf{1}_{Z(v_1,\dots,v_k)})_{(v_1 \nearrow \dots \nearrow v_k)}$ of the cylinder sets $Z(v_1,\dots,v_k)= \{x \in \Omega \ ; \ x_1 = v_1 ,\dots, x_k= v_k \}$ where the index sets runs over all finite paths on $\Omega$ of arbitrary length $k$. One checks that this family separates points and vanishes nowhere and is therefore dense in the space of continuous functions $C(\Omega)$ by the Stone-Weierstrass theorem. Since $\int_{\Omega} \mathbf{1}_{Z(v_1,\dots,v_k)}(x) d\nu(x) = \frac{\nu_k(v_k)}{\dim(v_k)} $ by centrality, the assertion follows.
\end{proof}

 The following theorem is well-known to experts, see e.g \cite[Proposition 7.8]{BO16} in the context of characters of compact groups. 

\begin{theorem} \label{thm.boundarytraces}
Let $\C = A_0 \subset A_1 \subset \dots$ be a sequence of finite-dimensional semisimple $*$-algebras with inductive limit algebra $A_{\infty}$. Further, let $\Gamma$ be the induction/restriction branching graph of this sequence, so that the vertices $v \in V_n$ correspond to the irreducible summands of $A_n$. Then, the Choquet simplex of tracial states on the envelopping $C^*$-algebra $C^*(A_{\infty})$ is homeomorphic to $\cM_c(\Gamma)$. More precisely, this homeomorphism takes the central measure $\nu$ to the tracial state $\tau_{\nu}$ determined by
\begin{align*}
\tau_{\nu}(x) = \sum_{i=1}^m \nu(X_n = v_i) \frac{\tau_i(x)}{\dim(v_i)} \qquad (x \in A_n),
\end{align*}
where $\tau_i$ is the extremal trace on $A_n$ belonging to the irreducible summand $v_i$. Under this homeomorphism, extremal central measures are mapped onto extremal tracial states.
\end{theorem} 

\begin{proof}[Sketch of the proof]
In order to check that $\tau_{\nu}$ is well defined, one needs to show that the traces $\tau_{\nu}^{(n)}$ defined by 
\[ \tau_{\nu}^{(n)}(x) = \sum_{i=1}^m \nu_n(v_i) \frac{\tau_i(x)}{\dim(v_i)}, \qquad (x \in A_n), \] satisfy $\tau_{\nu}^{(n)}|_{A_{n-1}} = \tau_{\nu}^{(n-1)}$ for all $n \geq 1$. However, a quick calculation shows that is exactly guaranteed by the centrality of the measure $\nu$ or equivalently the consistency of the coherent system $(\nu_n)_{n \geq 0}$. The injectivity of the map $\nu \mapsto \tau_{\nu}$ is then not hard to check. To show that this map is also surjective, one need only note that the restriction $\tau|_{A_n}$ of any tracial state on $C^*(A_{\infty})$ to the subalgebra $A_n$ admits a decomposition 
\begin{align*}
\tau|_{A_n} = \sum_{i=1}^m p_i^{(n)} \frac{\tau_i(x)}{\dim(v_i)},
\end{align*}
where $\sum_{i=1}^m p_i^{(n)} =1, p_i^{(n)} \geq 0$. The fact that the restriction of $\tau|_{A_n}$ to $A_{n-1}$ is nothing but $\tau|_{A_{n-1}}$ then yields that the family of probability measures defined by $\nu_n(v_i) := p_i^{(n)}, \ v_i \in V_n, $ forms a coherent system. Finally, checking that the map $\nu \mapsto \tau_{\nu}$ is continuous is once again straightforward, using e.g. Lemma \ref{lem.weakcoherent}. Since both tracial states and central measures form a Choquet simplex, the map $\nu \mapsto \tau_{\nu}$ is a homeomorphism.
\end{proof}

Thanks to the works of Vershik and Kerov \cite{VK81} \cite{VK82}, there is a well-developed machinery available to compute the minimal boundary of a branching graph $\Gamma$. For an ergodic central measure $\nu \in \partial \Gamma$ and connected vertices $v_n \in V_n, \ v_{n+1} \in V_{n+1}$, we will denote by $p_{\nu}(v_n,v_{n+1})$ the transition probability $\nu(X_{n+1}=v_{n+1} | X_{n}=v_{n} )$.  The following theorem is known as the \emph{(Vershik-Kerov-)ergodic method}, see e.g. \cite{VM15}.

\begin{theorem} \label{thm.ergodicmethod}
Let $\nu \in \partial\Gamma$ be an ergodic central measure on $(\Omega_{\Gamma}, \cF_{\Gamma})$.
\begin{itemize}
\item[(i)] For every finite path $(v_0 =\emptyset,v_1,\dots,v_n)$  and $\nu$-almost every infinite path $(\omega_i)_{i \geq 0}$, the sequence $\left(\frac{\dim(v_n,\omega_i)}{\dim(\omega_i)} \right)_{i \geq 0}$ has a limit and this limit is equal to
\begin{align*}
\lim_{i \to \infty} \frac{\dim(v_n,\omega_i)}{\dim(\omega_i)} = \frac{\nu(X_n = v_n)}{\dim(v_n)} = \prod_{l=0}^{n-1} p_{\nu}(v_l,v_{l+1}).
\end{align*}
\item[(ii)] Similarly, for every edge $(v,w)$ on $\Gamma$ and $\nu$-almost every infinite path $(\omega_i)_{i \geq 0}$, the sequence $\left(\frac{\dim(w,\omega_i)}{\dim(v,\omega_i)} \right)_{i \geq 0}$ has a limit and this limit is equal to
\begin{align*}
\lim_{i \to \infty} \frac{\dim(w,\omega_i)}{\dim(v,\omega_i)} =  p_{\nu}(v,w).
\end{align*}
\end{itemize}
 
\end{theorem}

The above theorem identifies the minimal boundary as a subset of another space of central measures, the so-called \emph{Martin boundary}, see \cite{BO16}. The ergodic method does not make any statements on whether or not the measures obtained from it are in fact ergodic (they need not be) and one typically has to check this 'by hand' for the example in question.

\subsection{Random walk interpretation of the boundary of a pascalized graph} \label{subsec.walkinterpretation}

If one is interested in the boundary $\partial \cP(\Gamma)$ of the pascalization of a branching graph $\Gamma$, it can be difficult to target this problem directly as the pascalized graph can be significantly larger than the original one. However, it is possible to reduce the study of the boundary $\partial \cP(\Gamma)$ to a problem on $\Gamma$ if one notes the following, see e.g. \cite{VM15}.

By an \emph{infinite (rooted) walk} on $\Gamma$, we mean a sequence of vertices $(v_n)_{n\geq 0}$  such that $v_n$ and $v_{n+1}$ are connected by an edge for all $n \geq 0$. We turn the space $\cW_{\Gamma}$ of walks on $\Gamma$ that start at $\emptyset$ into a Borel space by equipping it with the product $\sigma$-algebra $\cE_{\Gamma}$ inherited from the inclusion $\cW_{\Gamma} \subset \prod_{n\geq 1} V_{\leq n}, \ V_{\leq n} = \cup_{l \leq n \atop l =  n \ \mathrm{mod} \ 2} V_l $. The proof of the following result is straighforward.

\begin{lemma} \label{lem.pathwalkident}
The projection $\pi:((n,w_n))_{n \geq 0} \mapsto (w_n)_{n \geq 0}$ yields a one-to-one correspondence between infinite paths on $\cP(\Gamma)$ and rooted walks on $\Gamma$. More strongly, this map is an isomorphism of Borel spaces $(\Omega_{\cP(\Gamma)},\cF_{\cP(\Gamma)}) \cong (\cW_{\Gamma},\cE_{\Gamma})$. 
\end{lemma} 

\begin{proof}
The fact that the map $\pi$ is bijective is trivial. To see that $\pi$ respects the Borel structure, we simply note that since all $\cF_{\cP(\Gamma)}$ is generated by cylinder sets $\{x \in \Omega_{\cP(\Gamma)} \ ; \ x_0=(0,\emptyset), x_1 = (1,w_1) ,\dots, x_n= (n,w_n) \}$ where $(0,\emptyset) \nearrow (1,w_1) \nearrow \dots \nearrow (n,w_n)$ is an arbitrary finite path on $\cP(\Gamma)$. Similarly, by definition, $\cE_{\Gamma}$ is generated by cylinders $\{y \in \cW_{\Gamma} \ ; \ y_0=\emptyset, y_1 = w_1 ,\dots, y_n= w_n \}$ for arbitrary finite rooted walks $(\emptyset,w_1,\dots,w_n)$ on $\Gamma$. As $\pi$ maps cylinder sets onto cylinder sets, the assertion follows. 
\end{proof}

Lemma \ref{lem.pathwalkident} asserts in particular that every central measure on $(\Omega_{\cP(\Gamma)},\cF_{\cP(\Gamma)})$ can be interpreted as a random walk on $\Gamma$. This random walk can be time-inhomogeneous in the sense that transition probabilities $p((n,w_n),(n+1,w_{n+1}))$ can depend on $n$ and not only on the vertices $w_n, w_{n+1}$ themselves. Because of Lemma \ref{lem.pathwalkident}, we will freely switch between the interpretations of sequences $((n,w_n))_{n \geq 0}$ as paths on $\cP(\Gamma)$ and infinite walks on $\Gamma$.

\subsection{Preliminaries on diagram algebras}

The branching graphs to be discussed in this article will always be derived from towers of finite-dimensional algebras whose basis will be given by \emph{set partitions} and whose operations can be encoded diagrammatically. These algebras originate from Schur-Weyl duality of compact groups and have been used in \cite{BS09} by Banica and Speicher to introduce a class of compact quantum groups which they called \emph{easy}. These easy (or partition) quantum groups have also been extensively studied in \cite{We13}. Together, the articles \cite{BS09} and \cite{We13} have achieved a full classification of three important subclasses of easy quantum groups: the \emph{free} easy quantum groups \cite{BS09} \cite{We13}, the \emph{half-liberated} easy quantum groups \cite{We13} and the easy (classical) groups \cite{BS09}. Later, the full classification of all easy quantum groups was finalized in \cite{RaWe16}. The main focus of this paper is the subclass of free easy quantum groups. The branching graphs associated to easy groups will be addressed in the forthcoming article \cite{Wa20}. 

Let us recall that by a set partition with $k$ upper and $l$ lower points, we mean a decomposition of the set $\{1,2,\dots,k,1',\dots,l'\}$ into disjoint subsets (the \emph{blocks} of the partition). Order the boundary points as $1 < 2 < \dots < k < l' < (l-1)' < \dots < 1'$. A set partition is called \emph{noncrossing} if w.r.t. this order, the following holds: if $a<b<c<d$ and $a,c$ belong to the same block and $b,d$ belong to the same block, then all of them must belong to the same block. A diagrammatic depiction of a noncrossing partition can be found below in Figure \ref{fig.partition}. The set of all partitions with $k$ upper and $l$ lower points will be denoted by $\Part(k,l)$ and the set of noncrossing partitions  with $k$ upper and $l$ lower points will be denoted by $\mathrm{NC}(k,l)$. See \cite{BS09} or \cite[Definition 1.4]{We13} for more details on the operations in the following definition.

\begin{figure}[h!]
\begin{center}
\includegraphics[scale=0.3]{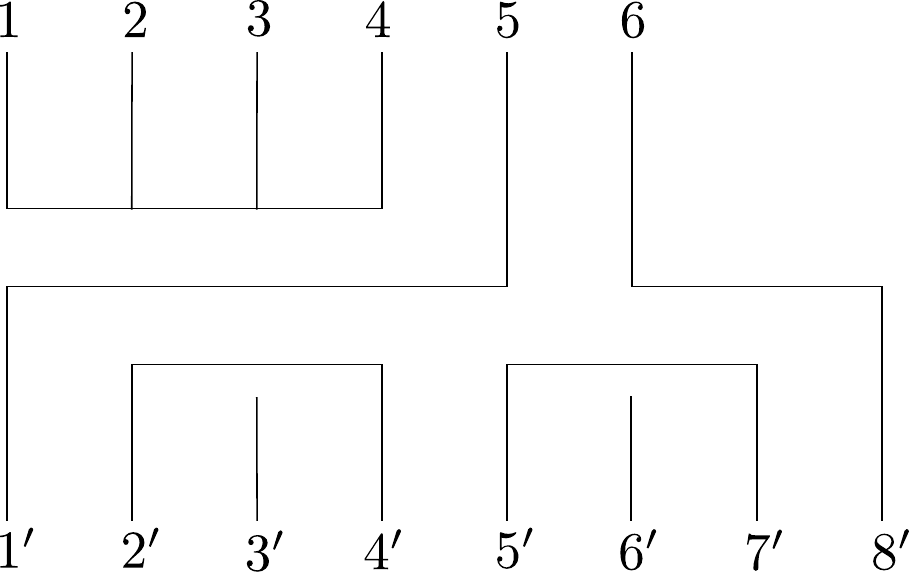}
\end{center}
\caption{\label{fig.partition} A partition in  $\mathrm{NC}(6,8)$.}
\end{figure}

\begin{definition}
A \emph{category of partitions} $\cC$ is a collection $(\cC(k,l))_{k,l \in \N}$ of subsets $\cC(k,l) \subset \Part(k,l)$ such that
\begin{itemize}
\item $\cC(1,1)$ contains the identity partition that connects the upper and the lower point.
\item the family is invariant under the category operations tensor product (i.e. vertical concatenation of diagrams), rotation, involution (i.e. reflecting a diagram along a horizontal line in the middle) and composition (i.e. vertical concatenation of compatible partitions $p_1 \in \cC(l,m), \ p_2 \in \cC(k,l)$, see Figure \ref{fig.multiplication} below).
\end{itemize} 
\end{definition}
  
\begin{figure}[h!]
\begin{center}
\includegraphics[scale=0.4]{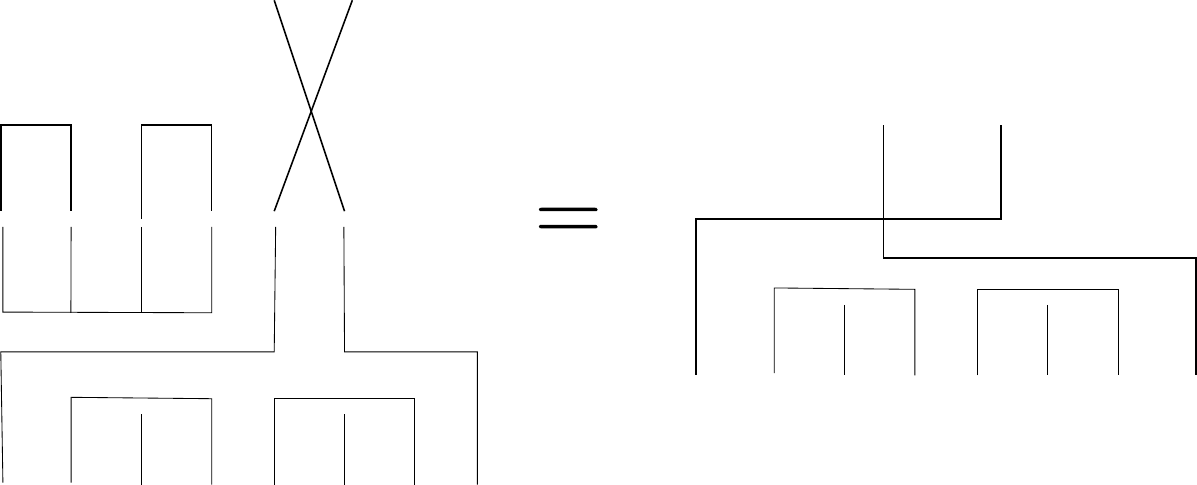}
\end{center}
\caption{\label{fig.multiplication} Vertical concatenation of partitions $p_1 \in \Part(6,8)$ and $p_2 \in \Part(2,6)$.}
\end{figure}

Given the pair $(\cC,\delta)$ of a category of partitions $\cC$ and a loop parameter $\delta$ and $k \geq 0$, we shall consider the free $\C$-vector space
\begin{align*}
 A_{(\cC,\delta)}(k) := \spann \{ e_p, \ p \in \cC(k,k)  \}
\end{align*}
spanned linearly by the elements of $\cC(k,k)$. We can define a multiplication on $A_{(\cC,\delta)}(k)$ by specifying the product of two basis vectors $e_{p_1}, e_{p_2},$ and extending this multiplication bilinearly to the whole space. To define $e_{p_1} \cdot e_{p_2}$, we concatenate $p_1$ and $p_2$ vertically, that is we draw $p_2$ on top of $p_1$ and then connect blocks of $p_1$ and $p_2$ that meet in the middle. We then erase every closed loop appearing in the middle of the picture that is not connected to any upper or lower points, so that we obtain a new partition $p_3$. We also record the number of loops that we erased and define the multiplication of base vectors $e_{p_1}\cdot e_{p_2} := \delta^{\# erased \ loops} e_{p_3}$. As mentioned above, the involution $p^*$ of a diagram $p$ is given by reflecting $p$ along a horizontal line in the middle, whence we get an involution $e_p^* = e_{p^*}$ on $A_{(\cC,\delta)}(k)$.

\begin{definition} \label{def.diagramalgebra}
The involutive algebra $A_{(\cC,\delta)}(k)$ defined in the previous paragraph is called the \emph{k-th diagram algebra of the pair} $(\cC,\delta)$ or alternatively the \emph{k-th diagram algebra of} $\cC$ \emph{at loop parameter} $\delta$.
\end{definition}

Recall that a finite-dimensional algebra $A$ is semisimple if and only if it possesses a \emph{positive} involution, i.e. one for which $x^*x=0$ implies $x=0$, see e.g. \cite[Appendix II]{GHJ89}. In addition, if there is a positive involution on $A$, there is also a unique $C^*$-norm on $A$ turning it into a $C^*$-algebra. Lastly, if $A \subset B$ is an inclusion of finite-dimensional semisimple algebras, for any positive involution on $A$, there is a positive involution on $B$ extending it. Therefore, when given an inductive sequence of finite-dimensional semisimple algebras, we can always consider its inductive limit in the category of $C^*$-algebras which is the viewpoint we want to take here. 

\begin{definition}
Let $(\cC,\delta)$ be a category of partitions at loop parameter $\delta$ and assume that for all $k\geq 1$, $A_{(\cC,\delta)}(k)$ is semisimple.  Let $\alpha_k: A_{(\cC,\delta)}(k) \to A_{(\cC,\delta)}(k+1)$ be the embedding obtained by mapping $e_p \to e_{p'}$, where $p'$ is obtained from $p$ by adding a through-string on the right of the diagram. The inductive limit (in the category of $C^*$-algebras) w.r.t. these embeddings will be denoted $A_{(\cC,\delta)}(\infty)$.   
\end{definition}

\begin{figure}[h!]
\begin{center}
\includegraphics[scale=0.3]{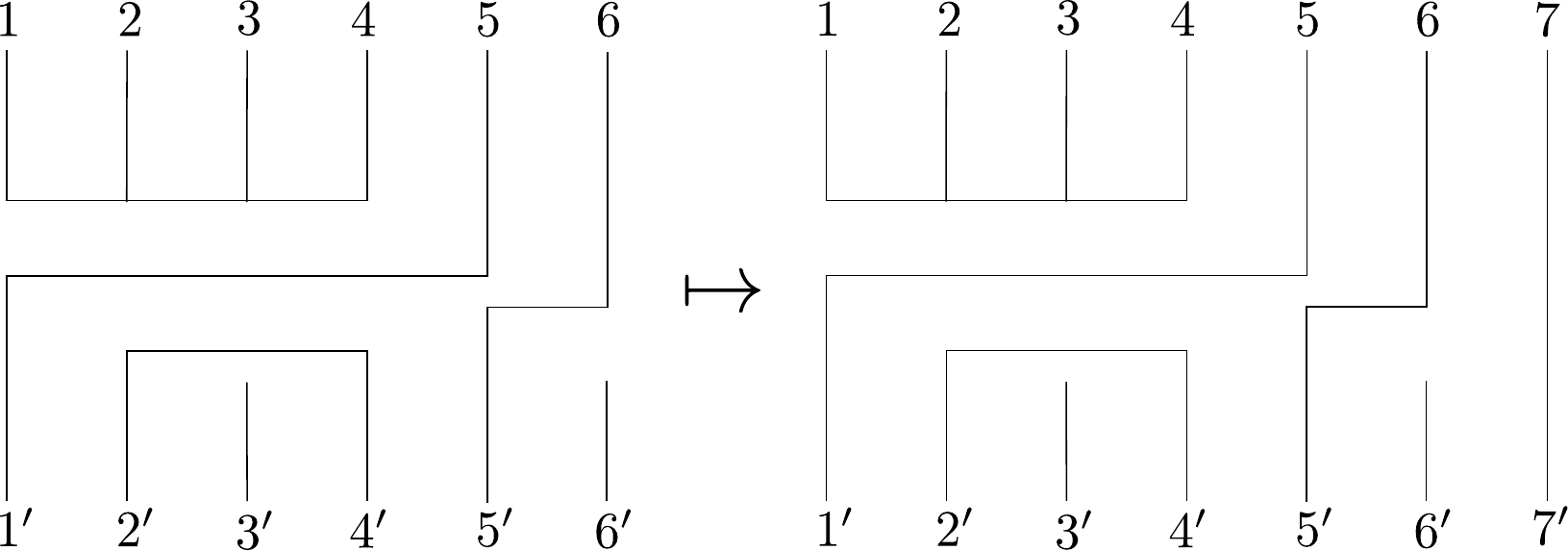}
\end{center}
\caption{\label{fig.embedding} The graphical embedding from $\cC(6,6)$ into $\cC(7,7)$ by adding a through-string.}
\end{figure}

Note that the set $\cC(\infty):= \bigcup_{k \geq 1} \cC(k,k)$ yields a natural basis for $A_{(\cC,\delta)}(\infty)$. Also we remark that by \cite[Theorem 5.13]{HR05}, the set of generic parameters (that is parameter values for which $A_{(\cC,\delta)}(k)$ is semisimple for all $k \geq 0$) is always co-countable. For more precise results on genericity of the loop parameter, see \cite{FM20}.

\subsection{The classification of free partition quantum groups} \label{subsec.classfreeeasyqg}

We will now recall the classification of categories of noncrossing set partitions due to Banica, Speicher \cite{BS09} and Weber \cite{We13}. These categories describe the representation theory of the \emph{free} easy quantum groups in the setting of Banica and Speicher and there are exactly seven of them. For details on this and the exact relationship between categories of partitions and compact quantum groups, we refer the reader to \cite{BS09} \cite{We13}.

\begin{theorem}[\cite{BS09},\cite{We13}] \label{thm.freecategories}
There are exactly seven categories of noncrossing partitions, namely
\begin{enumerate}
\item The category $\cC_{S^+} = \mathrm{NC}$ of all noncrossing partitions, corresponding to the quantum permutation groups $S_N^+$;
\item The category $\cC_{O^+} = \mathrm{NC}_2$ of all noncrossing pair partitions, corresponding to the free orthogonal groups $O_N^+$;
\item The category $\cC_{B^+}$ of all noncrossing partitions with blocks of size one or two, corresponding to the bistochastic quantum groups $B_N^+$;
\item The category $\cC_{H^+}$ of all noncrossing partitions with blocks of even size, corresponding to the hyperoctahedral quantum groups $H_N^+$;
\item The category $\cC_{S'^{+}}$ of all noncrossing partitions with an even number of blocks of odd size, corresponding to the modified quantum permutation groups $S_N'^{+}$;
\item The category $\cC_{B'^{+}}$ of all noncrossing partitions with an even number of blocks of size one and an arbitrary number of blocks of size two, corresponding to the modified bistochastic quantum groups $B_N'^{+}$;
\item The category $\cC_{B^{\#+}}$ of all noncrossing partitions with boundary points labelled by alternating symbols $a,b$ with an even number of blocks of size one and an arbitrary number of blocks of size two. Each block of size two has to connect a point labelled $a$ to a point labelled $b$. This category  corresponds to the freely modified bistochastic quantum groups $B_N^{\#+}$.
\end{enumerate}
\end{theorem}

We can right away observe the following result:

\begin{lemma} \label{lem.modifiedthesame}
For all $k \geq 0$, we have $A_{(\cC_{S^+},\delta)}(k) = A_{(\cC_{S^{'+}},\delta)}(k)$ and $A_{(\cC_{B^+},\delta)}(k) = A_{(\cC_{B^{'+}},\delta)}(k)$.
\end{lemma}

\begin{proof}
We simply observe that the difference between the categories of partitions $\cC_{S^{+'}}$ and $\cC_{S^+}$ is not visible in the sets $\cC(k,k), \ k\geq 0$, as the total number $k+k = 2k$ of boundary points is even. Indeed, the defining rule for $\cC_{S^{+'}}$ (see Theorem \ref{thm.freecategories}) which states that the number of blocks of odd size must be even, is automatically satisfied for all noncrossing partitions with an even number of boundary points. Since the diagram algebras $A_{(\cC_{S^+},\delta)}(k), A_{(\cC_{S^{'+}},\delta)}(k)$ belonging to these categories only make reference to the sets $\cC(k,k), \ k\geq 0$, they must be the same. As a consequence, also the branching graphs associated to these categories must be the same. This argument is also applicable to the categories $\cC_{B^+}$ and $\cC_{B^{+'}}$, which therefore produce the same diagram algebras and thus yield the same branching graph.
\end{proof}

\section{Traces on diagram algebras associated to free easy quantum groups} \label{sec.free}

We will now approach the new results of this article, namely the description of the traces on the algebras $A_{(\cC,\delta)}(\infty)$ at the generic parameter. We will begin with the diagram algebra associated to the category $\cC_{O^+} = \mathrm{NC}_2$ of all non-crossing pair partitions, which is better known as the \emph{infinite Temperley-Lieb algebra} $\TL(\infty,\delta)$. The main result of Subsection \ref{sec.templieb} is Theorem \ref{thm.randomballot}, which rephrases the classification of extremal traces on $\TL(\infty,\delta)$ as a classification result for random ballot paths. As explained in Remark \ref{rem.modifiedthesame}, Theorem \ref{thm.randomballot} will also settle the  trace classification problem for the examples $\cC_{S^+}$ and $\cC_{S'^{+}}$. The trace classification for the diagram algebras of the categories of partitions $\cC_{B^+}$ and $\cC_{B'^{+}}$ or equivalently, the classification of central random Motzkin paths, will be addressed in Subsection \ref{sec.freebistochasic}. Before stating the classification result (Theorem \ref{thm.randomMotzkinclass}), we will first carry out the algebraic legwork of computing the relevant branching graph which yields dimension formulas for irreducible representations of the quantum groups $B_N^+$ as a side product. In Subsection \ref{sec.freelymodified} , we compute the branching graph associated to the category $\cC_{B^{\#+}}$ and in Subsection \ref{sec.FussCat}, we review known results on the branching graph of $\cC_{H^+}$. There we will also deal with more general \emph{Fuss-Catalan-algebras}. In Subsection \ref{sec.FussCat}, we will moreover deduce new dimension formulas for irreducible representations of the quantum groups $B_N^{\#+}$. The trace classification for the remaining two examples $\cC_{B^{\#+}}$ and $\cC_{H^+}$ is in our eyes the most interesting and it will follow from the stochastic results on the boundary of the Fibonacci tree that we will obtain in Section \ref{sec.deFintheorem}.

\subsection{Traces on the infinite-dimensional Temperley-Lieb algebra} \label{sec.templieb}

The algebra spanned by noncrossing pair partitions in $\mathrm{NC}_2(k,k)$ between $k$ upper and $k$ lower points is the \emph{Temperley-Lieb algebra} $\TL(k,\delta)$ and it is typically defined as the associative unital algebra generated by elements $e_i, \ i = 1,\dots,k-1$ satisfying the relations 
\begin{align*}
e_i^2 &= e_i \qquad &i=1,\dots k-1 \\
e_i e_j &= e_j e_i \qquad &|i-j| > 1 \\
\delta e_i e_{j} e_i &= e_i \qquad &|i-j| = 1.
\end{align*}
The set of generic parameter values $\delta \in \C$ for the Temperley-Lieb algebras was determined by Jones \cite{Jo83} and contains in particular the interval $[2,\infty)$. If $\delta \in [2,\infty)$, the declaration $e_i^* = e_i, \ i=1,\dots, k-1$ defines a positive involution on $\TL(k,\delta)$ for all $k$, thus turning the generating idempotents into projections. 
The $k$-th Temperley-Lieb algebra admits a diagrammatical interpretation as the $k$-th diagram algebra of the category of partitions $\mathrm{NC}_2$ at loop parameter $\delta$, see e.g. \cite{BiJo95}. In this interpretation, the generator $e_i$ is identified with the element $\delta^{-1/2} e_{p_i}$, where $p_i$ is the partition that consists of the blocks $\{i,i+1\}, \{i',(i+1)'\}$ and $\{j,j' \}$ for $j\neq i,i+1$, see Figure \ref{fig.TLprojection} below.

\begin{figure}[h!]
\begin{center}
\includegraphics[scale=0.3]{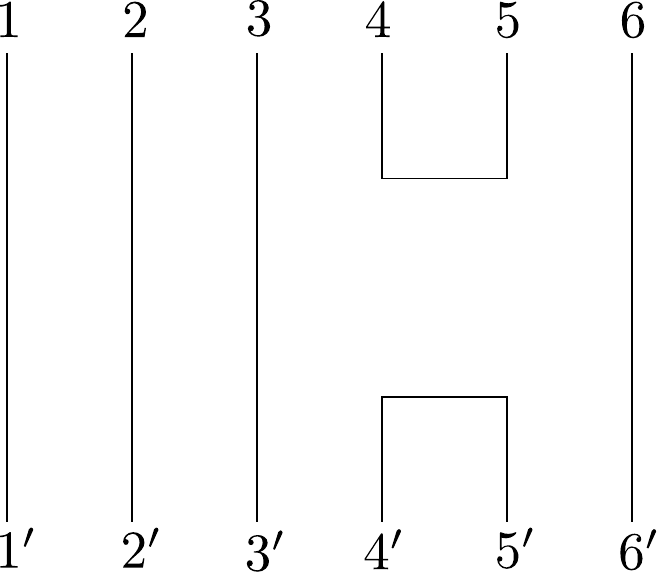}
\end{center}
\caption{\label{fig.TLprojection} The diagram representing the generator $e_4 \in \TL(6,\delta)$ up to a constant $\delta^{-1/2}$.}
\end{figure}

The $k$-th Temperley-Lieb algebra $\TL(k,\delta)$ embeds naturally into $\TL(k+1,\delta)$ by mapping $e_i \in \TL(k,\delta)$ to $e_i \in \TL(k+1,\delta)$ or equivalently by adding a through string on the right of every diagram in the diagrammatical interpretation. The inductive limit algebra under these embeddings will be denoted $\TL(\infty,\delta)$. The images of $\TL(k,\delta)$ at parameter value $\delta=N = 2,3\dots$ under the Banica-Speicher representation in \cite{BS09}, yield the endomorphism spaces the fundamental representations of the easy quantum groups $O_N^+$.

\subsubsection{The semi-Pascal graph} \label{subsubsec.semi-Pascal}

In the generic case, the Bratteli diagram of $\TL(\infty,\delta)$ does not depend on the choice of the parameter $\delta$ (as long as it is generic) and is given by the semi-Pascal graph, see Example \ref{ex.pascalization}. Let us label the vertices of $\cP(\N)_m$ by $(m,s), \ s=0,2,\dots,m$ if $m$ is even and by $(m,s), \ s=1,3,\dots m,$ if $m$ is odd such that we have an edge between $(m,s)$ and $(m+1,t)$ if and only if $t \in \{s-1,s+1\}$. A picture of the first few levels of the semi-Pascal graph $\cP(\N)$ is drawn in Figure \ref{fig.semiPascal} below.

\begin{figure}[h!]
\begin{center}
\includegraphics[scale=0.4]{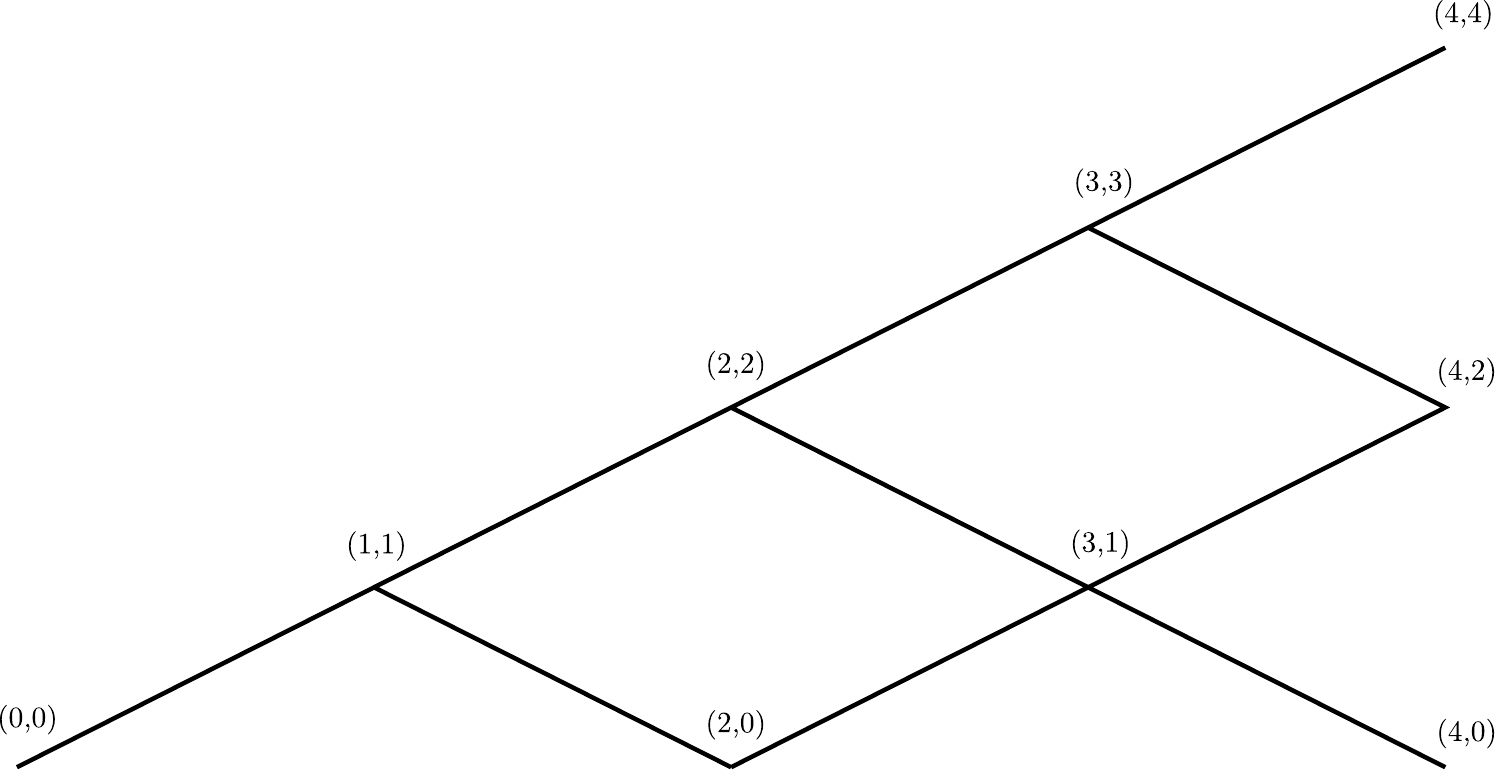}
\end{center}
\caption{\label{fig.semiPascal} The first five levels of the semi-Pascal graph.}
\end{figure}

From the description of the Bratteli diagram as the semi-Pascal graph, it is immediate that its infinite rooted paths exactly conincide with infinite ballot paths starting at $(0,0)$. Central measures on the semi-Pascal graph therefore translate into random ballot paths such that the conditional probability measure given that the random path passes through $(n,k)$ is the uniform distribution on $\Ba((0,0),(n,k))$.

\subsubsection{Traces on the Temperley-Lieb algebra in terms of random ballot paths} \label{subsubsec.Temptraces}

The boundary of the semi-Pascal graph (and thus the trace simplex on $\TL(\infty,\delta)$) has been studied in the PhD thesis of Wassermann \cite{Was81}, where it is shown that  the extremal trace simplex is homeomorphic to the half interval $[1/2,1]$. An algebraic explanation of this goes as follows. $\TL(\infty,\delta)$ is a quotient of the infinite Hecke algebra $H(\infty,q), \ q+q^{-1} = \delta$, a $q$-deformation of the symmetric group algebra. The extremal traces/ II$_1$-factor representations of $H(\infty,q)$ coincide with those of $S_{\infty}$ as $C^*(H_q(\infty)) \cong C^*(S_{\infty})$. The only II$_1$-factor representations that factor through the quotient map onto the Temperley-Lieb algebra are the ones corresponding to the Thoma parameters satisfying $\alpha_1 + \alpha_2 = 1$. The extremal trace simplex is thus homeomorphic to the half interval $[1/2,1]$. \\ 

Recall that by Theorem \ref{thm.boundarytraces}, traces on $\TL(\infty,\delta)$ are in natural correspondence to central measures on the semi-Pascal graph which in turn correspond to random ballot paths by the previous subsection. We will now describe the ergodic central measures (and thus the extremal traces) explicitly.



Let $\lambda \in (1/2,1]$. Define the Markov chain $M^{\lambda}$ on the space $(\Omega_{\cP(\N)}, \cF_{\cP(\N)})$ of infinite rooted ballot paths with transition probabilities
\begin{align*}
p_{\lambda}((m,s),(m+1,s+1)) \ &= \ \frac{(1-\lambda)^{s+2} - \lambda^{s+2}}{(1-\lambda)^{s+1} - \lambda^{s+1}}, \\
p_{\lambda}((m,s),(m+1,s-1)) \ &= \ 1- \frac{(1-\lambda)^{s+2} - \lambda^{s+2}}{(1-\lambda)^{s+1} - \lambda^{s+1}}.
\end{align*}
For $\lambda=1/2$, define the transition probabilities of $M_{1/2}$ by 
\begin{align*}
p_{1/2}((m,s),(m+1,s+1)) \ &= \ \frac{1}{2} \cdot \frac{s+2}{s+1}, \\
p_{1/2}((m,s),(m+1,s-1)) \ &= \ 1- \frac{1}{2} \cdot \frac{s+2}{s+1}.
\end{align*}

Note that the Markov chains $M^{\lambda}$ are time-homogeneous in the sense that the transition probabilities do not depend on $m$.  Denote the law of $M^{\lambda}$ by $\nu_{\lambda}$. The following theorem can be interpreted as a one-sided de Finetti theorem for discrete stochastic processes conditioned to stay non-negative.

\begin{theorem} \label{thm.randomballot}
Every central rooted random ballot path $M$ with law $\nu_M$ is a mixing of the Markov chains $M^{\lambda}, \ \lambda \in [1/2,1]$, that is to say there is a probability measure $\mu$ on $[1/2,1]$ such that 
\[ \nu_M = \int_{1/2}^1 \nu_{\lambda} \ d\mu(\lambda). \]
\end{theorem}

\begin{proof}
We need to show that the family $\nu_{\lambda}, \ \lambda \in [1/2,1]$ exhausts the ergodic central measures on $(\Omega_{\cP(\N)}, \cF_{\cP(\N)})$. The theorem then follows from the ergodic decomposition theorem. \\
Thus, let $\nu$ be an ergodic central measure on $(\Omega_{\cP(\N)}, \cF_{\cP(\N)})$. We invoke the ergodic method of Theorem \ref{thm.ergodicmethod} to compute the transition probabilities $p_{\nu}((m,s),(m+1,s+1))$. Set $k= (m-s)/2$, let $(N,S)$ be another vertex of $\cP(\N)$ and set $K = (N-S)/2$. By Theorem \ref{thm.numberballotpaths}, we have
\begin{align*}
\frac{\dim_{\cP(\N_0)}((m+1,s+1),(N,S))}{\dim_{\cP(\N_0)}((m,s),(N,S))} = \frac{\binom{N-m-1}{K-k} - \binom{N-m-1}{N-K-k+1}}{\binom{N-m}{K-k} - \binom{N-m}{N-K-k+1}}.
\end{align*} 
A quick computation shows that for large values of $N$ this expression is asymptotically equal to 
\begin{align*}
\frac{N-K}{K} \ \cdot \ \frac{1-\big(\frac{K}{N-K} \big)^{m-2k+2}}{1-\big(\frac{K}{N-K} \big)^{m-2k+1}}.
\end{align*}
This expression converges to the desired value
\begin{align*}
\frac{(1-\lambda)^{s+2} - \lambda^{s+2}}{(1-\lambda)^{s+1} - \lambda^{s+1}},
\end{align*}
if and only if $\tfrac{K}{N} \to \lambda \in (1/2,1]$  and to $\tfrac{1}{2} \cdot \tfrac{s+2}{s+1}$ if and only if $\tfrac{K}{N} \to \tfrac{1}{2}$.
\\
On the other hand, it is easily checked that the measures defined by the Markov chains $M_{\lambda}$ are indeed central. Their ergodicity is proven in \cite[p. 119-122]{Was81} by showing that the corresponding trace $\tau_{\lambda}$ on $\TL(\infty,\delta)$ is the restriction of the product state $\bigotimes_{n\geq 0} \Tr( \ \cdot \ \mathrm{diag}(\lambda, 1-\lambda))$  on $\bigotimes_{n\geq 0} M_2(\C)$ to $\TL(\infty,\delta)$ and are therefore extremal. We will not present the details of this argument here, but we will invoke a similar argument in the classification of central ergodic Motzkin paths in the next section. If one desires a more probabilistic argument, one can alternatively adapt the proof presented in Section \ref{sec.FussCat} to this simpler setup. 
\end{proof}




\begin{remark} \label{rem.modifiedthesame}
\begin{enumerate}
\item For the category $\cC_{S^+} = \mathrm{NC}$ of all noncrossing partitions, it has been observed in the literature that $A_{(\cC_{S^+},\delta)}(k) \cong A_{(\cC_{O^+},\sqrt{\delta})}(2k) = \mathrm{TL}(2k,\sqrt{\delta})$ through the so-called fattening isomorphism, see e.g \cite{LT16}. Since this isomorphism respects the inclusion $A_{(\cC_{S^+},\delta)}(k) \subset A_{(\cC_{S^+},\delta)}(k+1)$ to $\mathrm{TL}(2k,\sqrt{\delta}) \subset \mathrm{TL}(2k+2,\sqrt{\delta})$, it follows that, up to a change of the loop parameter, the limit algebra of $(\cC_{S^+},\delta)$ is isomorphic to the infinite Temperley-Lieb algebra, that is
\begin{align*}
A_{(\cC_{S^+},\delta)}(\infty) \cong \mathrm{TL}(\infty,\sqrt{\delta}).
\end{align*}

In particular, the trace simplex of $A_{(\cC_{S^+},\delta)}(\infty)$ is homeomorphic to that of the infinite Temperley-Lieb algebra by this identification. By Lemma \ref{lem.modifiedthesame}, the same holds true for $A_{(\cC_{S^{'+}},\delta)}(\infty)$ as it is exactly the same algebra as $A_{(\cC_{S^+},\delta)}(\infty)$.

On the level of graphs, we see that the branching graph of $A_{(\cC_{S^+},\delta)}(\infty)$ is the graph of even levels $\cP(\N_0)_{2n}$ of the semi-Pascal graph. In this graph, there is an edge from $v$ to $w$ for every two step path from $v$ to $w$ in $\cP(\N_0)$. The minimal boundary is thus homeomorphic to $[1/2,1]$ and the transition probabilities $\tilde{p}_{\lambda}(v,w)$ of the ergodic central Markov process $\tilde{M}_{\lambda}$ indexed by $\lambda$ is the probability that the process $M_{\lambda}$ defined above moves from $v \in \cP(\N_0)_{2n}$ to $w \in \cP(\N_0)_{2n+2}$ in two steps on the semi-Pascal graph.
\end{enumerate}
\end{remark}

\subsection{Traces on the infinite-dimensional Motzkin algebra} \label{sec.freebistochasic}

The next direct limit algebra that we would like to study is the one associated to the category $\cB=\cC_{B^+}$ which describes the representation theory of the free bistochastic quantum group in the setting of Banica and Speicher.  Recall from Theorem \ref{thm.freecategories} above that $\cC_{B^+}$ is formed by the noncrossing partitions with blocks of size one and two. In the next section we will compute the Bratteli diagram of $A_{(\cB,\delta)}(\infty)$. We learned after these computations were carried out that this branching graph had already been computed in \cite{BH14}, where the algebras $A_{(\cB,\delta)}(n)$ appear under the name \emph{Motzkin algebras}. We decided to nevertheless include our computations in the present article since they are a nice application of the recipe to compute the representation theory of free easy quantum group presented in \cite{FrWe16}. 

\subsubsection{The branching graph of $A_{(\cB,\delta)}(\infty)$} \label{subsubsec.freebistochbranching}

We will describe the representations of the finite-dimensional algebras $A_{(\cB,\delta)}(n), \ n\geq 1$ in a way that is similar to the standard description of Temperley-Lieb representations through link states. 

\begin{definition} \label{def.linkstate}
\begin{itemize}
\item A $n$-\emph{link state} for $\cB$ is a partition of $n$ points into pairs and two types of singletons: \emph{proper singletons} and \emph{defects}. Pictorially, a defect is labelled by an arrowhead while a proper singleton is not. A singleton that is braced by a pair is not allowed to be a defect, i.e. if $j$ is a singleton, $i<j < k$ and $i$ and $k$ are paired, then $j$ must be a proper singleton. A singleton that is not braced by a pair is allowed to be either a defect or a proper singleton.
\item A $(n,d)$-link state for $\cB$ is a $n$-link state with $d$ defects and the set of $(n,d)$-link states is denoted by $\cL_{(n,d)}$.  
\end{itemize}
\end{definition}

$n$-Link states for $\cB$ arise as upper halves of projective partitions in $\cB(n,n)$ (as defined in \cite{FrWe16}) when we cut them across a horizontal line in the middle of the diagram. A diagramatic example in which defects are marked by arrows is depicted in Figure \ref{fig.link} below.

\begin{figure}[h!]
\begin{center}
\includegraphics[scale=0.4]{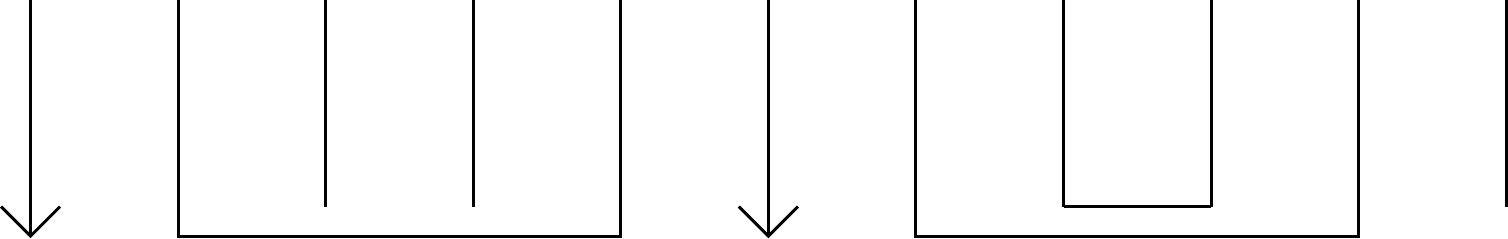}
\end{center}
\caption{\label{fig.link} A $(11,2)$-link state.}
\end{figure}

Analogous to the multiplication of partition diagrams, we can define an action $v \mapsto p \cdot v$ of a partition $p \in \cB(n,n)$ on an $n$-link state $v$ by drawing $p$ above $v$, connecting lines and deleting any lines that do not connect to points on the upper boundary. For any closed block (that is any line without defect) deleted in this way, we pay by multiplying the resulting diagram with the loop parameter $\delta$. Any line with a defect that is not connected to upper points may however be deleted right away (that is without multiplying by a scalar). Let $M_n$ denote the complex span of all $n$-link states. We can linearly extend the action of partitions on link states to a linear representation $A_{(\cB,\delta)}(n) \to \End(M_n)$, turning $M_n$ into a left $A_{(\cB,\delta)}(n)$-module.

\begin{definition} \label{def.linkmodule}
 We denote by $M_{(n,d)} \subset M_n$ the complex span of all $(n,d')$-link states with $d' \leq d$. For notational convenience, we also set $M_{(n,-1)} := \emptyset$.
\end{definition}    

It is clear that $M_{(n,d)}$ is a submodule of $M_n$. Moreover, by definition we have $M_{(n,d)} \subset M_{(n,d+1)}$ for all $d=0,\dots, n-1$. Let us denote the quotient modules by
\begin{align*}
V_{(n,0)}:= M_{(n,0)}, \qquad V_{(n,d)} := \faktor{M_{(n,d)}}{M_{(n,d-1)}}, \ d=1,\dots n.
\end{align*}

\begin{proposition}
$\{V_{(n,0)},\dots V_{(n,n)} \}$ is a full set of inequivalent irreducible modules for $A_{(\cB,\delta)}(n)$, that is to say we have a decomposition 
\begin{align*}
A_{(\cB,\delta)}(n) \cong \bigoplus_{d=0}^{n} \End(V_{(n,d)}).
\end{align*}
Moreover, the set $B_{(n,d)}=\{ v + M_{(n,d-1)}; \ v \in \cL_{(n,d)} \}$ is a basis of $V_{(n,d)}$ for $d=0,\dots,n$.
\end{proposition}

\begin{proof}
The claim about the basis follows directly from the definition of our modules $V_{(n,d)}$. 
To see that the module $V_{(n,d)}$ is irreducible for any $d=0,\dots,n$, it suffices to note that the equivalence class of any $(n,d)$-link state $v$ is a cyclic vector for $V_{(n,d)}$. Indeed, given another $(n,d)$-link state $w$, we can define the partition $p(v,w)$ by drawing $w$ on top of $v$, flipping $v$ upside down and connecting the defects to obtain $d$ through-strings. Since $p(v,w) \cdot v = w$, it follows that the equivalence class of $v$ is cyclic and  $V_{(n,d)}$ is thus irreducible. To see, that $V_{(n,d)}$ and $V_{(n,c)}$ cannot be equivalent for $d < c$, it suffices to note that $p(v,v) (v+M_{(n,d)}) = (v+M_{(n,d)})$, while $p(v,v) (w+M_{(n,c)}) = 0 \in V_{(n,c)}$ as the number of defects is lowered. Thus, a $A_{(\cB,\delta)}(n)$-covariant linear  map $T: V_{(n,d)} \to V_{(n,c)}$ can never be invertible.
Finally, to show that we have found a full family of irreducible modules, we can invoke a dimension argument. To do so, simply note that the map $\cL_{(n,d)} \times \cL_{(n,d)} \to \cB(n,n)_d, \ (v,w) \mapsto p(v,w)$ is a bijection, where $\cB(n,n)_d$ denotes the set of partitions in $\cB(n,n)$ with $d$ through-strings. The inverse of this map is obtained by cutting up a partition $p$ in the middle so that the upper and lower part of the diagram form $d$-link states. Therefore, $\sum_{d=0}^n \dim(V_{(n,d)})^2 = \sum_{d=0}^n |\cL_{(n,d)} |^2 = |\cB(n,n)| = \dim(A_{(\cB,\delta)}(n))$.  
\end{proof}

To determine the branching graph of $A_{(\cB,\delta)}(\infty)$, we need to determine the decomposition of the modules $V_{(n,0)},\dots V_{(n,n-1)} $ of $A_{(\cB,\delta)}(n)$, when considered as modules of $A_{(\cB,\delta)}(n-1)$. Here, as always, we work with the natural embedding $A_{(\cB,\delta)}(n-1) \to A_{(\cB,\delta)}(n)$ that adds a through string to any partition diagram. When considered as an $A_{(\cB,\delta)}(n-1)$-module, we write $V_{(n,d)}$ as $V_{(n,d)}^{\downarrow}$.

\begin{proposition} \label{prop.decompbistochastic}
The $A_{(\cB,\delta)}(n-1)$-module $V_{(n,d)}^{\downarrow}$ decomposes as
\begin{align*}
V_{(n,d)}^{\downarrow} \cong V_{(n-1,d-1)} \oplus V_{(n-1,d)} \oplus V_{(n-1,d+1)}
\end{align*}
for $d=1,\dots,n-2$. In the remaining cases $d=0, \ d=n-1, \ d=n$, we have
\begin{align*}
V_{(n,0)}^{\downarrow} \cong V_{(n-1,0)} \oplus V_{(n-1,1)}, \quad V_{(n,n-1)}^{\downarrow} \cong V_{(n-1,n-2)} \oplus V_{(n,n-1)}, \quad V_{(n,n)}^{\downarrow} \cong V_{(n-1,n-1)}.
\end{align*}
\end{proposition} 

\begin{proof}
For a $(n-1,d)$-link state $v$, let us define $\dot{v}$ to be the $(n,d)$-link state obtained by adding a proper singleton on the right. Similarly, we define $\vec{v}$ as the $(n,d+1)$-link state obtained from $v$ by adding a defect on the right. Lastly, if $d >0$, we define $\breve{v}$ to be the $(n,d-1)$-link state that we obtain from $v$ by adding a point on the right and pairing it with the rightmost defect of $v$. For $d=1,\dots,n-2$, we define the map
\begin{align*}
 V_{(n-1,d-1)} \oplus V_{(n-1,d)} \oplus V_{(n-1,d+1)} \to V_{(n,d)}^{\downarrow}, \ (u,v,w) \mapsto \vec{u} +\dot{v} +\breve{w},
\end{align*}
where $u \in B_{(n-1,d-1)}, v \in B_{(n-1,d)}, w \in B_{(n-1,d+1)} $ and where we ignored the equivalence classes in the quotient to lighten the notation (note that this is well-defined). It is now easy to check that this map indeed extends to an isomorphism $V_{(n,d)}^{\downarrow} \cong V_{(n-1,d-1)} \oplus V_{(n-1,d)} \oplus V_{(n-1,d+1)}$ of $A_{(\cB,\delta)}(n-1)$-modules. The cases $d=0,n-1,n$ are analogous.
\end{proof}

A pictorial depiction of the first levels of the branching graph $\Gamma_{\cB^+}$ associated to $A_{(\cB,\delta)}(\infty)$ is shown in Figure \ref{fig.Motzkin}.

\begin{figure}[h!]
\begin{center}
\includegraphics[scale=0.4]{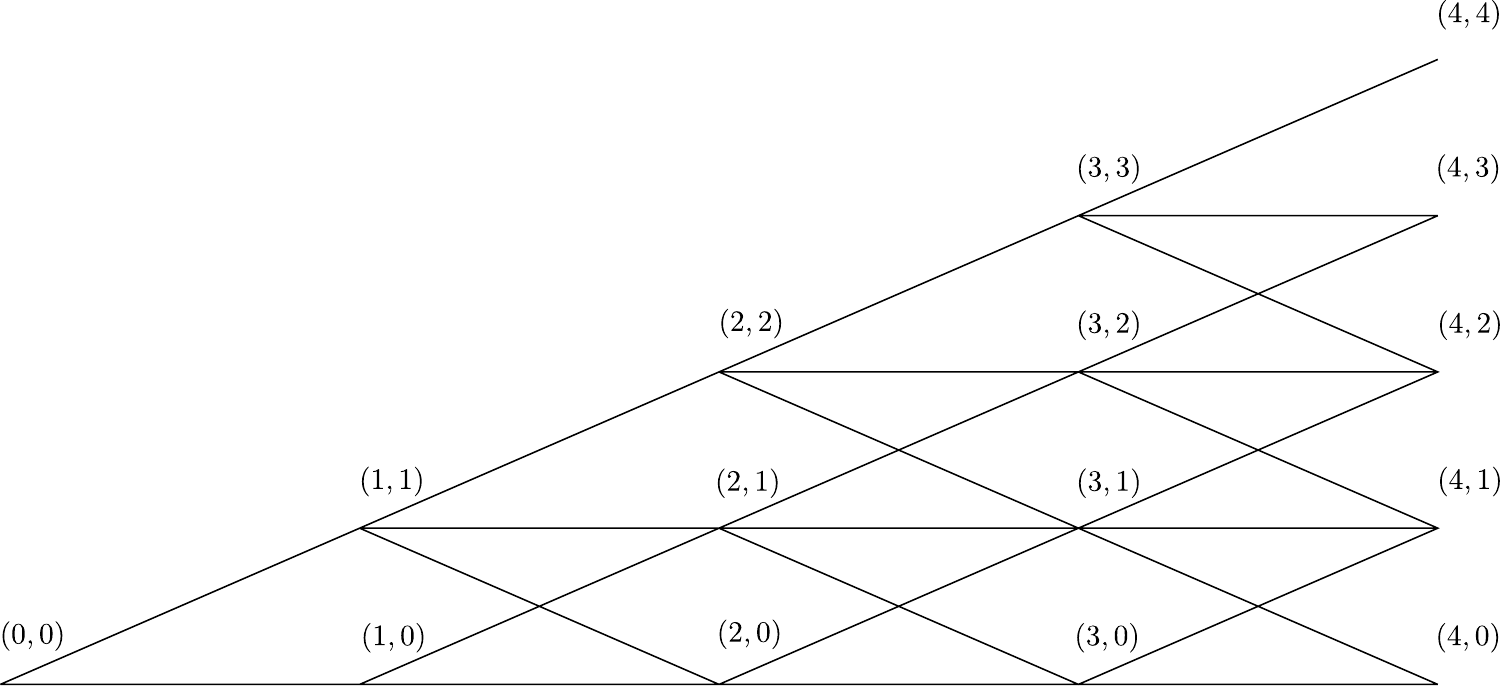}
\end{center}
\caption{\label{fig.Motzkin} The first five levels of the branching graph $\Gamma_{\cB^+}$.}
\end{figure}



The following observation is now immediate.

\begin{proposition} \label{prop.Motzkinpaths}
Let $m >n$. There is a natural bijection between paths on the branching graph $\Gamma_{\cB^+}$ starting at vertex $(n,d)$ and ending at vertex $(m,c)$ and Motzkin paths starting at $(n,d)$ and ending at $(m,c)$.
\end{proposition}

\begin{proof}
For any edge $(n,d) \to (n+1,d')$ in $\Gamma_{\cB}$ we have $(n+1,d') - (n,d) \in \{ (1,1),(1,0),(1,-1) \}$ corresponding to up-steps, level-steps and down-steps.
\end{proof}

As was the case for the ballot path interpretation of the branching graph of $\TL(\infty,\delta)$, Proposition \ref{prop.Motzkinpaths} yields a convenient reinterpretation
 of central measures on the space $(\Omega_{\Gamma_{\cB^+}}, \cF_{\Gamma_{\cB^+}})$ of infinite paths on $\Gamma_{\cB^+}$. A central measure $\mu$ is a \emph{random (infinite) Motzkin paths} $(w_i)_{i \geq 0}$ starting at $(0,0)$ such that the conditional distribution given $w_n = (n,d)$ is uniform on $\Mo((0,0),(n,d))$.



\begin{corollary} \label{cor.freebistochasticdimensions}

The dimension of the simple module $V_{(n,d)}$ of $A_{(\cB,\delta)}(n)$ is 
\begin{align*}
\dim V_{(n,d)} = \sum_{k=0}^{n} \binom{n}{k} \bigg( \binom{n-k}{(n-k+d)/2} - \binom{n-k}{(n-k+d+2)/2} \bigg)=: m_{n,d}.
\end{align*}
In particular the dimension of $V_{(n,0)}$ is the $n$-th Motzkin number $m_n$.
\end{corollary}
\begin{proof}
The dimension of $V_{(n,d)}$ is equal to the number of paths in $\Gamma_{\cB}$ starting at the root and ending at $(n,d)$, see Remark \ref{rem.dimensions}. Hence, $\dim V_{(n,d)} = |\Mo((0,0),(n,d))|$ by Proposition \ref{prop.Motzkinpaths} and the result follows from the enumeration formula for Motzkin paths of Theorem \ref{thm.Motzkinnumbers}.
\end{proof}

\begin{remark}
\begin{itemize}
\item[(1)] More generally, our description of the branching graph through Motzkin paths together with Theorem \ref{thm.Motzkinnumbers} provides an exact formula for the multiplicity of $V_{(n,d)}$ in the induced representation $\Ind_{A_{(\cB,\delta)}(m)}^{A_{(\cB,\delta)}(n)} V_{(m,c)}$ for arbitrary choices of $(m,c), \ (n,d),$ as this multiplicity is equal to $|\Mo((m,c),(n,d))|$.
\item[(2)] Alternatively, one can deduce the formula for $\dim V_{(n,d)}$ from known results on easy quantum groups: the bistochastic group $B_N^+$ with fundamental representation $u_{B}$ is isomorphic as a matrix quantum group to  the free orthogonal group $O_{N-1}^+$ with representation $u_{O} \oplus 1$, where $u_O$ denotes the fundamental representation of $O_{N-1}^+$, see \cite[Theorem 4.1]{Ra12}. Since the dimension of $V_{n,d}$ is equal to the multiplicity of the irreducible representation $u_d$ in $(u_O \oplus 1)^{\ot k} = \sum_{k=0}^n \binom{n}{k} u_O^{\ot k}$, $\dim V_{(n,d)}$ can be calculated by using the well-known formulas for the multiplicities in $O_{N-1}^+$. These exactly correspond to the dimensions of the simple modules of the Temperley-Lieb algebras.  
\end{itemize}
\end{remark}

\begin{remark} \label{rem.ladder}
Although we will not make explicit use of it, the branching graph $\Gamma_{\cB} = \cP(\mathbb L)$ can also be recognized as the pascalization of the graph $\mathbb L$ depicted in Figure \ref{fig.ladder}. For obvious reasons we call this graph the \emph{ladder}.
\end{remark}

\begin{figure}[h!]
\begin{center}
\includegraphics[scale=0.4]{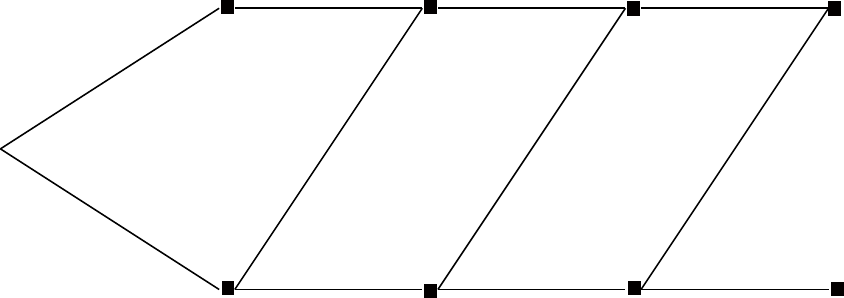}
\end{center}
\caption{\label{fig.ladder} The ladder $\mathbb L$.}
\end{figure}

\subsubsection{Traces on the Motzkin algebra in terms of random Motzkin paths} \label{subsubsec.tracesMotzkin}

We will now classify ergodic central random Motzkin paths (and thereby extremal traces on $A_{(\cB,\delta)}(\infty)$ ) in analogy to Theorem \ref{thm.randomballot}. Given an ergodic central probability measure $\nu$ let us denote the probability of arriving at vertex $(n,d)$ by 
\begin{align*}
v_{n,d} \ := \ \nu(\{ \omega \in \Omega_{\cB^+} \ ; \ \omega_n = (n,d) \}).
\end{align*}
Then, centrality implies that $\nu$ is uniquely determined by the values $v_{n,0}, \ n \geq 0$
and the recursive rules
\begin{align*}
\frac{m_{n-1,0}}{m_{n,1}} \ v_{n,1} = v_{n-1,0} \left(1- \frac{m_{n-1,0}}{m_{n,0}} v_{n,0} \right)
\end{align*}
and
\begin{align*}
\frac{m_{n-1,d}}{m_{n,d+1}} \ v_{n,d+1} = v_{n-1,d}  \left(1- m_{n-1,d} \left(\frac{v_{n,d-1}}{m_{n,d-1}} + \frac{v_{n,d}}{m_{n,d}} \right) \right)
\end{align*}
for $1 < d <n$. To see that these recursions are in fact correct, one simply notes that when dividing by $v_{n-1,d}$ one gets an expression for the probability under $\nu$ of the transition $(n-1,d) \to  (n,d+1)$.
 
Now, let $\lambda_1, \lambda_2 \in [0,1]$. By $M^{(\lambda_1, \lambda_2)} = (M^{(\lambda_1, \lambda_2)}_n)_{n \geq 0}$, we will denote the unique random Motzkin path that is Markov and whose transition probabilities are determined by
\begin{align*}
v^{(\lambda_1, \lambda_2)}_{n,0} = \bP(M^{(\lambda_1, \lambda_2)}_n = (n,0)) =  \sum_{l=0}^{\lfloor \tfrac{n}{2} \rfloor} \frac{n!}{(n-2l)!(l+1)!l!} \lambda_1^l \lambda_2^l (1-\lambda_1 - \lambda_2)^{n-2l}
\end{align*}
and the recursion above. Denote the law of $M^{(\lambda_1, \lambda_2)} $ by $\nu_{(\lambda_1, \lambda_2)}$.

Let 
\[ U:= \{ (\lambda_1,\lambda_2) \in [0,1] \times [0,1] \ ; \ \lambda_1 \geq \lambda_2, \ 0 \leq \lambda_1 + \lambda_2 \leq 1 \} \]
and note that for two distinct elements of $U$, the associated random Motzkin paths are distinct.

\begin{theorem} \label{thm.randomMotzkinclass}
A (rooted) random Motzkin path is central and ergodic if and only if it is equal to $M^{(\lambda_1, \lambda_2)}$ for some $(\lambda_1,\lambda_2) \in U $.

In particular, every central random Motzkin path $M$ with law $\nu_M$ is a mixing of the Markov chains $M^{(\lambda_1, \lambda_2)}, \ (\lambda_1,\lambda_2) \in U$, that is to say there is a probability measure $\mu$ on $U$ such that 
\[ \nu_M = \int_{U} \nu_{(\lambda_1, \lambda_2)} \ d\mu(\lambda_1,\lambda_2). \]
\end{theorem}

\begin{proof}
We will adapt the methods of \cite{Was81} to our situation to classify the traces on $A_{(B,\delta)}(\infty)$. Let us first extend the standard representation $\pi: SU(2) \to M_2(\C)$ of $SU(2)$ to a representation \[\tilde{\pi}:SU(2) \to M_3(\C), \qquad \tilde{\pi}(g) = \begin{pmatrix}
 \pi(g) & 0 \\ 0 & 1 
\end{pmatrix}. \]
Note that the compact matrix quantum group $(B_N^+,u_B)$ (with standard fundamental representation $u_B$) is isomorphic to $(O_{N-1}^+,u_{O} \oplus 1_O)$ (with standard representation $u_{O}$ and trivial representation $1_O$) and that the parameter $\delta = N$ is generic for all $N \geq 3$. Hence we can just work with the choice $N=3$. Since $O_{2}^+ \cong (SU(2),\pi)$, it follows that $A_{(\cB,3)}(\infty)$ is the fixed point algebra of the infinite tensor product action
\begin{align*}
\bigotimes_{n=0}^{\infty} \Ad(\tilde{\pi})(g): \ \bigotimes_{n=0}^{\infty} M_3(\C) \to \bigotimes_{n=0}^{\infty} M_3(\C), \qquad \qquad g \in SU(2),
\end{align*}
where $\Ad$ denotes the adjoint action.
By \cite[Theorem 4]{Was81}, a trace on $A_{(B,3)}(\infty)$ is extremal if and only if it is the restriction of a product state 
\begin{align*}
\varphi_{(\lambda_1,\lambda_2)} := \bigotimes_{n=0}^{\infty}  \Tr_3(\ \cdot \ \mathrm{diag}(\lambda_1,\lambda_2,\lambda_3)),
\end{align*}
with $\lambda_1, \lambda_2 \geq 0, \ \lambda_1 + \lambda_2 \leq 1, \ \lambda_3 = 1-\lambda_1 - \lambda_2$.
We will denote the extremal trace on $A_{(B,3)}(\infty)$ arising as the restriction of $\varphi_{(\lambda_1,\lambda_2)}$ by $\tau_{(\lambda_1,\lambda_2)}$.

Since the central projection $p_{(n,0)}$ in the irreducible module $V_{(n,0)} \subset A_{(\cB,3)}(n)$ is the projection onto the space of invariant vectors of  the representation $\tilde{\pi}^{\otimes n}$ of $SU(2)$, $p_{(n,0)}$ is given by the formula
\begin{align*}
p_{(n,0)} = \int_{SU(2)} \tilde{\pi}^{\otimes n}(g)^{\ot n} \ dg,
\end{align*} 
where integration is understood w.r.t. the Haar state on $SU(2)$. Using the standard parametrisation $\begin{pmatrix} \alpha & \beta \\ -\overline{\beta} & \overline{\alpha}\end{pmatrix}, \ |\alpha|^2 + |\beta|^2 = 1$, for elements of $SU(2)$, we arrive at
\begin{align*}
\tau_{(\lambda_1,\lambda_2)}(p_{(n,0)}) &= \int_{SU(2)} \langle \mathrm{diag}(\lambda_1,\lambda_2,\lambda_3),\tilde{\pi}(g) \rangle^n \ dg \\
&= \int_{SU(2)} \Tr \Bigg(\begin{pmatrix} \lambda_1 & 0 & 0 \\ 0 & \lambda_2 & 0 \\ 0 & 0 & 1-\lambda_1 - \lambda_2 \end{pmatrix},\begin{pmatrix} \alpha & \beta & 0 \\ -\overline{\beta} & \overline{\alpha} & 0 \\ 0 & 0 & 1 \end{pmatrix} \Bigg)^n \ dg(\alpha,\beta) \\
&= \int_{SU(2)} (\lambda_1 \alpha + \lambda_2 \overline{\alpha}+ 1-\lambda_1 - \lambda_2 )^n \ dg(\alpha,\beta).
\end{align*}
To this expression, we now apply the formula
\begin{align*}
\int_{SU(2)} f(\alpha,\beta) \ dg(\alpha,\beta) = \frac{1}{16 \pi^2} \int_{-2\pi}^{2\pi} \int_{0}^{2\pi} \int_{0}^{\pi} f(\alpha,\beta) \sin(\theta) \ d\theta \ d\phi \ d\psi
\end{align*}  
for the Haar state on $SU(2)$, with the reparametrisation $\alpha = \cos (\theta/2) \exp(i (\phi+ \psi)/2), \ \beta = i \sin (\theta/2) \exp(i (\phi - \psi)/2)$, see e.g. \cite[p.159]{Vi68}. After a routine computation, this yields
\begin{align*}
\tau_{(\lambda_1,\lambda_2)}(p_{(n,0)}) &=  \sum_{l=0}^{\lfloor \tfrac{n}{2} \rfloor} \frac{n!}{(n-2l)!(l+1)!l!} \lambda_1^l \lambda_2^l (1-\lambda_1 - \lambda_2)^{n-2l}.
\end{align*}
Since the process $(M_n^{(\lambda_1,\lambda_2)})_n$ with law $\nu_{(\lambda_1,\lambda_2)}$ satisfies
\begin{align*}
\bP(M_n^{(\lambda_1,\lambda_2)} = (n,0)) = \tau_{(\lambda_1,\lambda_2)}(p_{(n,0)}),
\end{align*}
the theorem now follows by decomposing central measures into ergodic components.
\end{proof}

\begin{corollary} \label{cor.Motzkintraceclassification}
The Choquet simplex of extremal traces on the infinite Motzkin algebra $\Mo(\infty,\delta)$ at a generic loop parameter is homeomorphic to the simplex $U \subset [0,1] \times [0,1]$ with the standard topology. 

Explicitly, the homeomorphism is obtained by composing the homeomorphisms 
\[ \psi: U \to \partial \Gamma_{\mathcal{B}^+}, \quad \lambda  = (\lambda_1,\lambda_2) \mapsto \nu_{\lambda} \]
described above and the homeomorphism $\nu_{\lambda} \mapsto \tau_{\lambda}$ of Theorem \ref{thm.boundarytraces}.
\end{corollary}

\begin{proof}
Theorem \ref{thm.randomMotzkinclass} shows that $\psi$ is one-to-one. Moreover if $\lambda \to \eta \in U$, then $v_{n,0}^{\lambda}$ converges to $v_{n,0}^{\eta}$ and using the recursion formula before Theorem \ref{thm.randomMotzkinclass}, it follows that $v_{n,d}^{\lambda} \to v_{n,d}^{\eta}$ for all $n$ and $d$. Therefore, it follows from Lemma \ref{lem.weakcoherent} and the fact that $U$ and $\Gamma_{\mathcal{B}^+}$ are compact that $\psi$ is a homeomorphism. 
\end{proof}

\subsection{Traces on the infinite diagram algebra $A_{(\cB^{\#},\delta)}(\infty)$ and the infinite $2$-Fuss-Catalan algebra} \label{sec.tracesFCandrest}

In this section, our goal is to describe the traces on the infinite dimensional diagram algebras $A_{(\cB^{\#},\delta)}(\infty)$ and $\FC_2(\infty,\delta)$. The reason that these algebras are addressed together in the same section is that once the algebraic work is done and the representation theory is computed, their principal graphs will be almost identical. Therefore, the exact same arguments apply for the computation of extremal central measures on their branching graphs.  

\subsubsection{The branching graph of the algebra $A_{(\cB^{\#},\delta)}(\infty)$} \label{sec.freelymodified}

We will first turn towards the branching graph of the algebras $A_{(\cB^{\#},\delta)}(n)$ that come with the category of partitions $\cB^{\#} := \cC_{B^{\# +}}$. Recall that $A_{(\cB^{\#},\delta)}(n)$ is spanned by noncrossing set partitions with $n$ upper and $n$ lower boundary points whose blocks are of size one or two with a counterclockwise labelling of the boundary points alternating between two symbols $a$ and $b$. The labelling starts with $a$ in the upper right corner a block of size two always consists of two points which are labelled differently, see part (7) of Theorem \ref{thm.freecategories}. We call set partition of this form $ab$-\emph{admissible}.

Our goal is to describe the branching graph of the sequence of finite-dimensional algebras $A_{(\cB^{\#},\delta)}(n), \ n \geq 0$ and we will do so in Corollary \ref{cor.pascofderootedFT}. Unsurprisingly, the computations resemble those for $A_{(\cB^{+},\delta)}(n), \ n \geq 0$ in Subsection \ref{subsubsec.freebistochbranching} with the additional caveat of having to deal with the labelling. We will again describe the irreducible modules of $A_{(\cB^{\#},\delta)}(n)$ by refining the notion of link states of Definition \ref{def.linkstate} to $ab$-admissible partitions.


\begin{definition} \label{def.modifiedlinkstate}
A $n$-\emph{link state of} $\cB^{\#}$ is an $ab$-admissible set partition of $n$ points into pairs and singletons of two types called \emph{proper singletons} and \emph{defects}. A defect is not allowed to be braced by a pair, i.e. if $i<j < k$ and $i$ and $k$ are paired, then $j$ must be a proper singleton. Moreover, a defect is called $a$-defect, respectively $b$-defect, if it is labelled $a$, resp. $b$.

A $(n,k,l)$-link state of $\cB^{\#}$ is a $n$-link state with $k$ $a$-defects and $l$ $b$-defects.  
\end{definition}

\begin{figure}[h!]
\begin{center}
\includegraphics[scale=0.4]{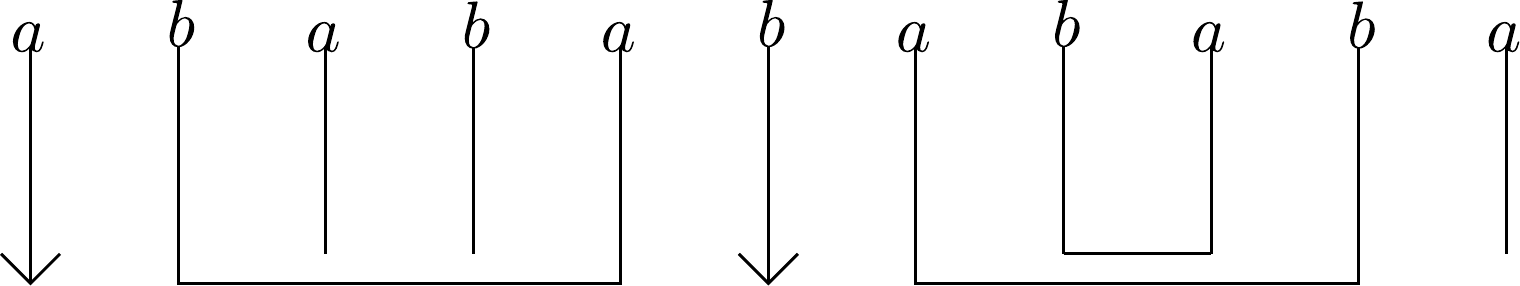}
\end{center}
\caption{\label{fig.admissiblelink} A $(11,1,1)$-link state of $\cB^{\#}$.}
\end{figure}

As in the previous cases, a partition $p$ in $A_{(\cB^{\#},\delta)}(n)$ acts on an $n$-link state $v$ by drawing $p$ above $v$ and connecting lines (without any additional restrictions by the labelling). Note that a point of $v$ labelled $a$ always meets a lower boundary point of $p$ labelled $b$ and the other way round. The fact that this operation returns an $n$-link state is easily verified in the exact same way as in the proof that composition of partitions in $\cB^{\#}$ is well-defined,  see \cite[Proof of Prop. 2.7]{We13}.

\begin{definition}
We denote the vector space freely spanned by all $n$-link states by $M_n$ and the subspace of $M_n$ spanned by all $n$-link states with at most $d$ defects by $M_{(n,d)}$ where $d=0,\dots n$. Note that $M_{(n,n)} = M_n $.
\end{definition}

Let $\alpha(n)$ be the word of length $n$ with alternating letters $a,b$ that ends in $a$. If $n=0$, we set $\alpha(0) = \emptyset$. 
\begin{definition}
We denote by $W_{\alpha(n)}$ the set of subwords of $\alpha(n)$, that is the set of words that are obtained from $\alpha(n)$ by deleting letters. The inverted word $\bar{w}$ of $w$ is the word arising from $w$ by applying the map $ a \mapsto b, \ b\mapsto a$ to every letter of $w$. In particular, if $w \in W_{\alpha(n)}$, then $\bar{w} \in W_{\overline{\alpha(n)}}$.
\end{definition}

\begin{definition} \label{def.wordoflink}
Let $v$ be an $n$-link state. The \emph{boundary word} $w(v)$ is the subword of $\alpha(n)$ that labels the defects of $v$ from left to right. 
\end{definition}

\begin{lemma} \label{lem.words}
Let $v_1$ and $v_2$ be $n$-link states with the same number of defects $d$. There exists a partition $p$ in $\cB^{\#}(n,n)$ such that $p \cdot v_1 = v_2 $ if and only if $w(v_1) = w(v_2)$.
\end{lemma}

\begin{proof}
Suppose first that a partition as in the assertion exists. If $m$ is a defect of $v_2$ then there must be a through-string of $p$ connecting of $p$ whose lower boundary point is a defect for $v_1$. Since the number of defects is the same for $v_1$ and $v_2$, all defects of $v_1$ are met. Moreover, since $p$ is noncrossing, the $k$-th defect of $v_2$ has to connect to the $k$-th defect of $v_1$. Since the word labelling lower boundary points of $p$ that meet defects of $v_1$ is $\overline{w(v_1)}$ and since through strings always connect opposite labels, it follows that $w(v_2) = \overline{\overline{w(v_1)}} = w(v_1)$. 
  
Conversely, if $v_1$ and $v_2$ are such that $w(v_1) = w(v_2)$ we can define the partition $p$ by placing $v_2$ above $v_1$, flipping $v_1$ horizontally and inverting the boundary labelling of $v_1$. We then connect the $k$-th defect of $v_1$ to the $k$-th defect of $v_2$ to not produce any crossing. Since $w(v_1) = w(v_2)$ and since through-strings connect boundary points of opposite labels (as we inverted $w(v_1)$), $p$ is well-defined.
\end{proof}

In analogy to Subsection \ref{subsubsec.freebistochbranching} we make $M_n$ into a left $A_{(\cB^{\#},\delta)}(n)$-module by extending the action of the base partitions of $A_{(\cB^{\#},\delta)}(n)$ on $n$-link states linearly to all of $A_{(\cB^{\#},\delta)}(n)$ and $M_n$. Since the action of base partitions of $A_{(\cB^{\#},\delta)}(n)$ can never add a defect to a link state, the vector spaces $M_{(n,d)}, \ d=0,\dots n$ also inherit a left $A_{(\cB^{\#},\delta)}(n)$-module structure. \\

We write  $W^{(k,l)}_{\alpha(n)} \subset  W_{\alpha(n)}$ for the set of subwords of $\alpha(n)$ consisting of $k$ letters $a$ and $l$ letters $b$ where $k=0,\dots \lceil \tfrac{n}{2} \rceil, \ l=0,\dots \lfloor  \tfrac{n}{2} \rfloor$.

\begin{theorem} \label{thm.modifiedmodules}
\begin{itemize}
\item[(1)] The module $V_{(n,0,0)}:= M_{(n,0)}$ is irreducible.
\item[(2)]The quotient module $V_{(n,d)} := \faktor{M_{(n,d)}}{M_{(n,d-1)}}, \ d=1,\dots n$
decomposes as a direct sum
\begin{align*}
V_{(n,d)} = \bigoplus_{k,l: \ k+l=d} V_{(n,k,l)}
\end{align*}
of submodules $V_{(n,k,l)}, \ k=0,\dots \lceil \tfrac{n}{2} \rceil, \ l=0,\dots \lfloor  \tfrac{n}{2} \rfloor, \ k+l=d$. The submodule $V_{(n,k,l)}$ is spanned by the equivalence classes of $(n,k,l)$-link states in $V_{(n,d)}$.
\item[(3)] The  submodule $V_{(n,k,l)}$ decomposes further into inequivalent irreducible submodules
\begin{align*}
V_{(n,k,l)} = \bigoplus_{w \in  W^{(k,l)}_{\alpha(n)}} V_{w},
\end{align*}
where $V_w$ is spanned by the equivalence classes of $(n,k,l)$-link states $v$ with $w(v)=w$.
\item[(4)] The modules $V_w, \ w \in W_{\alpha(n)}$ form a full family of inequivalent irreducible modules for $A_{(\cB^{\#},\delta)}(n)$.
\end{itemize}
\end{theorem}

\begin{proof}
We note first that, for $d > 0$ ,the equivalence classes $v +  M_{(n,d-1)}, \ v \in \mathcal{L}_{n,d}$, where $\mathcal{L}_{n,d}$ is the set of $n$-link states with $d$ defects, form a basis of $V_{(n,d)}$. For $d=0$, the same statement is true with the equivalence classes may be dropped since we are not in a quotient. We show first that $V_{(n,k,l)}$ is in fact invariant under the action of $A_{(\cB^{\#},\delta)}(n)$ on $V_{(n,d)}$ and thus a submodule for $k+l=d$. Let $p \in \cB^{\#}(n,n)$ a partition and $v$ a $(n,k,l)$-link state. Since $p$ has $2n$ boundary points, the rightmost lower point is labelled $b$. Thus, when $p$ acts on $v$, an $a$-labelled boundary point of $v$ meets a $b$-labelled lower point of $p$. There are now two possibilities: either a defect of $v$ is erased by the action of $p$ in which case $p \cdot v$ is annihilated by the quotient or all defects are kept, which means that they must connect to an upper point of $p$. By the observation we just made on the matching of boundary points, the upper point to which a defect connects, carries the same label as the defect. Therefore, $p \cdot v$ is also a $(n,k,l)$-link state and $V_{(n,k,l)}$ is invariant under the action of $A_{(\cB^{\#},\delta)}(n)$. By counting dimensions, it follows that $V_{(n,d)}$ decomposes in the stated manner so that (2) holds. The refined decomposition of $V_{(n,k,l)}$ stated in (3), then follows from Lemma \ref{lem.words}, as the orbits of $(n,k,l)$-link states under the action of $A_{(\cB^{\#},\delta)}(n)$ are fully classified by the words in $W^{(k,l)}_{\alpha(n)}$. This also shows that the module $V_w$ is irreducible for arbitrary $w \in W_{\alpha(n)}$. In particular, this proves (1) since $V_{n,0,0} = V_{\emptyset}$.
 
To finish the proof of (3) and (4), we still have to show that the modules $V_w, \ w \in W_{\alpha(n)}$ are pairwise inequivalent. To see this, we invoke \cite[Theorem 5.5]{FrWe16} which tells us that every minimal projection in $A_{(\cB^{\#},\delta)}(n)$ is a projective partition. The minimal projections that act nontrivially on the module $V_w$ are the projections onto $\langle v \rangle$ for $n$-link states $v$ with $w(v) = w$. By Lemma \ref{lem.words} two minimal projective partitions mapping onto $\langle v_1 \rangle, \langle v_2 \rangle$ respectively, cannot be equivalent if $w(v_1) \neq w(v_2) $.  
\end{proof}

As always, we define the inclusion $A_{(\cB^{\#},\delta)}(n-1) \hookrightarrow A_{(\cB^{\#},\delta)}(n) $ by adding a through string on the right of every base partition (and extending this mapping linearly). Under this embedding the labelling of a boundary point is inverted: a boundary point that was labelled $a$ is now labelled $b$ and the other way round.

For a word $w \in W_{\alpha(n)}$, we write $|w|$ for its length, $ll(w)$ for its last letter and $w\backslash ll(w)$ for the word obtained from $w$ by deleting its last letter.

\begin{proposition} \label{prop.decompositionmodifiedmodules}
As an $A_{(\cB^{\#},\delta)}(n-1)$-module, $V_w, \ w \in W_{\alpha(n)}$ decomposes as the direct sum of irreducible modules
\begin{align*}
V_w^{\downarrow} = \delta_{ll(w)=a} V_{\overline{w \backslash ll(w)}} \oplus \delta_{|w|<n-1} V_{\overline{w}b} \oplus \delta_{|w|<n} V_{\bar{w}},
\end{align*}
where $\delta_{ll(w)=a}$ (resp. $\delta_{|w|<m})$ is the characteristic functions of the set $\{ w \in W_{\alpha(n)} \ ; \ ll(w)= a\}$ (resp. $\{ w \in W_{\alpha(n)} \ ; \ |w|< m\})$.
\end{proposition}

\begin{proof}
The isomorphism 
\[\delta_{ll(w)=a} V_{\overline{w \backslash ll(w)}} \oplus \delta_{|w|<n-1} V_{\overline{w}b} \oplus \delta_{|w|<n} V_{\bar{w}} \to V_w^{\downarrow} \]
is defined in the same way as in Proposition \ref{prop.decompbistochastic}. We leave out the details of the proof.
\end{proof}

We are now ready to describe the branching graph $\Gamma_{\cB^{\#}}$ of $A_{(\cB^{\#},\delta)}(\infty)$. Since the inversion of words that appears in Proposition \ref{prop.decompositionmodifiedmodules} is somewhat inconvenient, we will do the following. Define the word
\begin{align*}
\beta(n) = \begin{cases} \alpha(n) \quad n \text{ odd} \\ \overline{\alpha(n)} \quad n \text{ even} \end{cases}.
\end{align*}
In other words, $\beta(n)$ is simply the alternating word $abab\dots$ with $n$ letters starting with $a$. We also identify $W_{\alpha(n)}$ with $W_{\beta(n)}$ through the bijection $w \mapsto \overline{w}$ for even $n$. Under this identification, $u \in W_{\beta(2n-1)}$ is connected to $w \in W_{\beta(2n)}$ for $n \geq 1$ if and only if 
\begin{align*}
w = u, \qquad \text{or} \qquad w=ub \qquad \text{or} \qquad u=wa.
\end{align*}
$u \in W_{\beta(2n)}$ is connected to $w \in W_{\beta(2n+1)}$ for $n\geq 0$ if and only if 
\begin{align*}
w = u, \qquad \text{or} \qquad w=ua \qquad \text{or} \qquad u=wb.
\end{align*}
Note that after our change of notation from $\alpha(n)$ to $\beta(n)$, the first connection rule $w=u$ exactly corresponds to the last summand in the decomposition of Proposition \ref{prop.decompositionmodifiedmodules}. The second connection rule $w=ub$, respectively $w=ua$, corresponds to the middle summand in this decomposition and the last connection rule $u=wa$, respectively $u=wb$ corresponds to the first summand.
The first few levels of $\Gamma_{\cB^{\#}}$ are depicted in Figure \ref{fig.modifiedbranching} below.

\begin{figure}[h!]
\begin{center}
\includegraphics[scale=0.3]{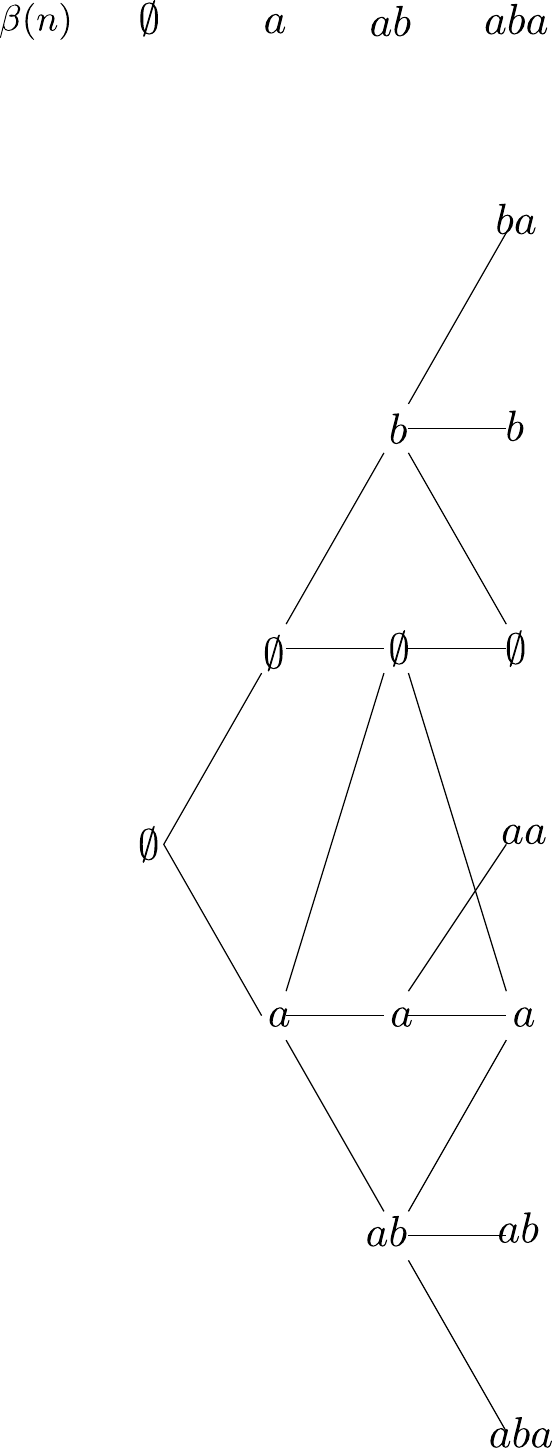}
\end{center}
\caption{\label{fig.modifiedbranching} The first four levels of $\Gamma_{\cB^{\#}}$ and the words $\beta(n)$.}
\end{figure}

Again, the graph $\Gamma_{\cB^{\#}}$ can be thought of the pascalization of a smaller graph which is the \emph{derooted Fibonacci tree} $\FT^*$. Let us therefore remind the reader that the (non-derooted) Fibonacci tree $\FT$ is given as follows. On the $0$-th level, we start with the root $\emptyset$ ('a newborn pair of rabbits'). The rabbits now age one month creating a vertex on level one and are now able to procreate so that they have two descendants on level two corresponding to themselves and a new pair. This process continues inductively with each new pair only being able to procreate after having lived for one month. More formally, let us describe the vertices of $\FT$ through words in letters $a$ and $b$.  The $n$-th level set $\FT_n, \ n \geq 1$ contains words of length $n$ starting with $a$ such that two copies of the letter $b$ never appear next to each other. A word $w_1 \in \FT_n$ is connected to a word in $w_2 \in \FT_{n+1}$ if $w_2=w_1 b$ or $w_2=w_1 a$. The derooted Fibonacci tree $\FT^*$ is just $\FT$ with the root deleted, i.e. $\FT^*_n = \FT_{n+1}, \ n \geq 0$. Note that any word in $\FT^*$ begins with the letter $a$ and in particular, the root of $\FT^*$ is  labelled by the single letter word $a$.
The first few levels of $\FT$ and $\FT^*$ are depicted in Figure \ref{fig.Fib}.

\begin{figure}[h!]
\begin{center}
\includegraphics[scale=0.4]{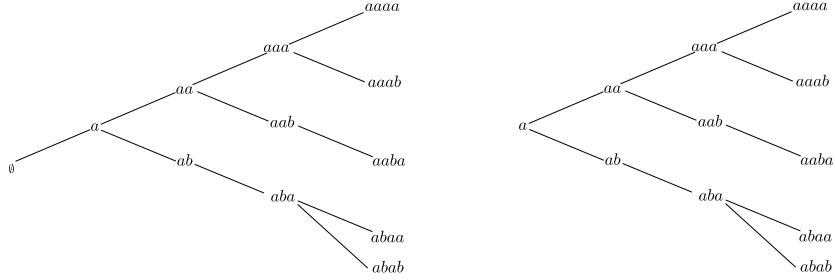}
\end{center}
\caption{\label{fig.Fib} The first levels of the Fibonacci tree $\FT$ and the derooted Fibonacci tree $\FT^*$.}
\end{figure}

\begin{corollary} \label{cor.pascofderootedFT}
The branching graph $\Gamma_{\cB^{\#}}$ is the pascalization of the derooted Fibonacci tree $\FT^*$.
\end{corollary}

The easiest way for the reader to convince themselves that corollary is this true, is to simply draw the derooted Fibonacci graph and pascalize it graphically. The formal proof goes as follows.

\begin{proof}
Since every word $w$ in $\FT^*$ begins with the letter $a$, we can simply delete this letter from all words. We thus identify the word $w$ with the new word $\tilde{w}$  where $\tilde{w} = aw$. We denote the graph with the new labelling  by $\widetilde{\FT}^*$. Let us define a map $\phi_n : \widetilde{\FT}^*_n \to W_{\beta(n)}, \ \tilde{w} \mapsto \phi_n(\tilde{w}) $ in the following way. First replace every letter $a$ of $\tilde{w}$ that is in an even position by the auxiliary letter $c$ and delete all letters $b$. Then replace all instances of $c$ again by $b$.
The map $\phi_n$ is injective for all $n$ and its image is the set
\begin{align*}
\phi_n(\widetilde{\FT}_n^*) = W_{\beta(n)} \bigg\backslash \left( \bigcup_{k <n, \atop k \equiv n \ \mathrm{mod} \ 2} W_{\beta(k)} \right).
\end{align*}
 Therefore, we can define a new branching graph $\phi(\tilde{\FT}^*)$ with $\phi(\widetilde{\FT}^*)_n = \phi_n(\widetilde{\FT}_n^*)$ and the connection rule $\phi_n(u) \to \phi_{n+1}(w)$ if and only if $\tilde{u}\to \tilde{w}$ in $\widetilde{\FT}^*$ for $\tilde{u} \in \widetilde{\FT}^*_n, \tilde{w} \in \widetilde{\FT}^*_{n+1}$. If $\tilde{w} = \tilde{u}b$, then $\phi_{n+1}(\tilde{w}) = \phi_{n}(\tilde{u})$. If $n$ is even and $\tilde{w} = \tilde{u}a$, then  $\phi_{n+1}(\tilde{w}) = \phi_{n}(\tilde{u})a$. Lastly, if $n$ is odd and $\tilde{w} = \tilde{u}a$, then  $\phi_{n+1}(\tilde{w}) = \phi_{n}(\tilde{u})b$. Now, combining the definition of pascalization with the connection rules found for $\Gamma_{\cB^{\#}}$, it follows that $\Gamma_{\cB^{\#}} = \cP(\phi(\tilde{\FT}^*)) \cong \cP(\FT^*)$. 
\end{proof}

\subsubsection{The branching graph of the Fuss-Catalan algebras} \label{sec.FussCat}

The hyperoctahedral quantum group $H_N^+$ is isomorphic to the free wreath product $\Z_2 \wr_* S_N^+$ and the endomorphism spaces $\End(w^{\ot k})$ of its fundamental representation $w$ are known to be isomorphic to the $s$-\emph{Fuss-Catalan algebras} of Bisch and Jones for $s=2$ and a suitable choice of loop parameters, see e.g. \cite{TW18}. The $s$-Fuss-Catalan algebras $\FC_s(k,\delta_1,\dots, \delta_s), s \geq 2$ were introduced in \cite{BiJo95} and further analyzed in \cite{La01} in the context of intermediate subfactors. For higher values of $s$, these algebras still admit an interpretation as commutants of the fundamental representation of a compact quantum group, namely the free wreath product quantum group $\Z_s \wr_* S_N^+$.

As in the Temperley-Lieb case, there is a nice description of these algebras in terms of generators and relations, see \cite{La01}, but we will stick to their description as diagram algebras, as this makes it easier to identify $\FC_2(k,\delta)$ with $A_{(\cC_{H^+},\delta)}(k)$. Many fundamental aspects of the representation theory of the Fuss-Catalan algebras including their branching graph (for fixed $s$ and increasing $k$) have been computed in \cite{BiJo95} and \cite{La01}. We will thus start this section with a summary of the results needed for a discussion of the minimal boundary of the branching graphs.  

\begin{definition} \label{def.FCdiagram}
Let $s \geq 2, \ k \geq 1$ and let $p$ be a noncrossing pair partition diagram with $s \cdot k$ upper and equally many lower points. Fix a set of labels $a_1,\dots a_s$ and label the boundary points of $p$ clockwise starting at the upper left corner by
\begin{align*}
a_1, a_2, \dots, a_s, a_s, \dots a_2, a_1, a_1, a_2 \dots. 
\end{align*}
We call $p$ a $(s,k)$\emph{-Fuss-Catalan diagram} if for any two boundary points of $p$ that are in the same block, their labels coincide. We will denote the set of all $(s,k)$-Fuss-Catalan-diagrams by $\cFC(s,k)$.
\end{definition}

The Fuss-Catalan algebras can be introduced as follows. Note that when multiplying $(s,k)$-Fuss-Catalan diagrams $p$ and $q$ in the usual way, i.e. by stacking $q$ on top of $p$ connecting lines and erasing loops, every loop generated in the middle of the diagram inherits a unique label $a_i$ from the blocks that it is build from.

\begin{definition} \label{def.FCalgebra}
Let $s \geq 2, \ k \geq 1$ and fix $\delta_1,\dots, \delta_s \in \C\backslash \{0 \}$. The $s$-\emph{Fuss-Catalan-algebra} $\FC_s(k,\delta_1,\dots, \delta_s)$ is the $\C$-vector space with basis $e_p, \ p \in \cFC(s,k)$ with multiplication $e_p \cdot e_q$ given by
\begin{align*}
e_p \cdot e_q = \delta_1^{l_1} \dots \delta_s^{l_s} e_{p \cdot q},
\end{align*}  
where $p \cdot q$ is the usual multiplication of partitions and where $l_i, \ i=1,\dots, s$ is the number of loops with label $a_i$.
\end{definition}

Once again, for fixed $s$ the set of generic parameter vectors $(\delta_1,\dots,\delta_s)$ is dense in $\C^s$ and contains the set $[2,\infty)^s$, see \cite[Corollary 2.2.5]{BiJo95}. From the perspective of categories of partitions, a good choice of parameters is $s=2$ and $\delta_1=\delta_2:= \delta, n \geq 4$ as in this case, the Fuss-Catalan algebra $\FC_2(k,\delta) := \FC_2(k,\delta,\delta) $ is isomorphic to the $k$-diagram algebra of the category $\cC_{H^+}$, see e.g. \cite{BBC07}.

Again, at a generic parameter and for arbitrary $s\geq 2$, we get an embedding of semisimple algebras $\FC_s(k,\delta_1,\dots, \delta_s) \subset \FC_s(k+1,\delta_1,\dots, \delta_s)$ by adding $s$ through strings of type $a_1,\dots,a_s$ on the right of every diagram. As before, we denote the inductive limit algebra by $\FC_s(\infty,\delta_1,\dots, \delta_s)$. As the results that will follow will not depend on the parameters (as long as they are generic), we will write $\FC_s(k), \ k=0,1,\dots, \infty$ for short. The branching graph $ \Gamma^s := \Gamma^{FC_s}$ of the inductive sequence $\FC_s(1,\delta_1,\dots, \delta_s) \subset \FC_s(2,\delta_1,\dots, \delta_s) \subset \dots$ at a generic parameter was computed in \cite{BiJo95} for $s=2$ and \cite{La01} for $s \geq 3$. It is the pascalization of the following tree.

\begin{definition} \label{def.FCtree}
The $s$\emph{-Fuss-Catalan tree} $\T^s$ is the tree constructed according to the following inductive procedure. Let $\emptyset$ be the root of the tree and label it with a $1$. If a descendent $v$ at distance $k, \ k\geq 0$ from $\emptyset$ with label $i$ is given, create $i$ children of $v$ and label them $s,s-1,\dots,s-i+1$.
\end{definition}

Note that the label of a vertex $v$ of $\T^s$ is simply its number of descendants.
For $s=2$, the tree $\T^s$ is exactly the Fibonacci tree $\FT$ discussed in the previous section. If we define the map $\psi^2: \{\emptyset,a,b\} \to \{1,2 \}, \ \emptyset \mapsto 1,\  b \mapsto 1, \ a \mapsto 2$, the label of a word $w \in \FT_n$ is simply $\psi(l(w))$, where $l(w)$ is its last letter. More generally, note that for any $s$, we can uniquely identify any vertex $v$ of $\T^s$ with a word in $s$ letters $a_1,\dots, a_s$ by identifying $a_k$ with the label $s-k+1$ through a map $\psi^s: a_k \mapsto s-k-1, \ (\emptyset \mapsto 1)$ and recording the labels on the unique path from the first level $\T^s_1$ to $v$. We will call any word appearing in this manner $s$-\emph{admissible}, so that we can identify $\T^s_n$ with the set of $s$-admissible words of length $n$. Since the branching graph $\Gamma^{s} = \cP(\T^s)$ is the pascalization of $\T^s$, any vertex $v \in \Gamma^{s}_n$ is described by a pair $(n,w)$ with an $s$-admissible word $w \in \T^s_k,$ of length $k \leq n, \ k \equiv n \mod 2$.

\subsubsection{Walk counting on (derooted) Fuss-Catalan trees and dimension formulas for $\FC_2(\infty,\delta)$ and $A_{(\cB^{\#},\delta)}(\infty)$} \label{subsec.walkcount}

The dimension $\dim_{\Gamma^{s}} (n,w)$ of the irreducible representation of $\FC_s(k)$ indexed by a vertex $(n,w)$ was computed in \cite[Section 9]{La01} using the underlying combinatorics of the tree $\T^s$ and Raney's recurrence relation for Fuss-Catalan numbers laid out in \cite[p. 360ff]{GKP94}. In this section, we will recall these results and extend them to the derooted Fuss-Catalan trees. 

Before we state Landau's formula, we will introduce the following notations which we will use throughout this section. 
\begin{itemize}
\item The derooted $s$-Fuss-Catalan tree $\widetilde{\T}^{s}$ is the $s$-Fuss-Catalan tree whose root (and root-adjacent edge) have been deleted. The root of the derooted $s$-Fuss-Catalan tree thus corresponds to the unique first level vertex of $\T^s$. 
\item As before we will refer to the length of an $s$-admissible word $w$ by $|w|$. For two vertices $v,w$ on a tree, we will write $d(v,w)$ for their distance, so that on the Fuss-Catalan tree $d(\emptyset,w) = |w|$.
\item For two $s$-admissible words $v,w$ we denote by $s(v,w)$ the longest word such that there exist words $v',w'$ with $v = s(v,w)v'$ and $w = s(v,w)w'$. Note that $s(v,w)$ is the vertex on the $s$-Fuss-Catalan tree where the paths from $v$ and $w$ to the root meet for the first time. 
\item We define the label sum of a word $w=a_{i_1}a_{i_2} \dots a_{i_n}$ as $r(w):= 1 + \sum_{k=1}^n \psi^s(a_{i_k})$. Equivalently, $r(w)$ is the sum $\sum_{k=0}^n l(v_k)$  of labels $l(v_k)$ of the vertices on the unique path $v_0=\emptyset,v_1,\dots,v_n=w$  from $\emptyset$ to $w$. Recall that $l(v)$ is nothing but the number of descendants of $v$.
\item For a vertex $v$ on the tree $\T^s$, the number of loops of length $2n$ starting and ending at $v$ that never get closer to the root than $v$ only depends on $l(v)$. We will denote this number by $C^{s,i}_n$ for $1 \leq l(v)=i \leq s$.
\item The number of walks of length $m$ on the tree $\T$ from $v$ to $w$ will be denoted by $\Wk_{\T}(m,v,w)$.
\item For an $s$-admissible word $w$ and $n \geq |w|/2$, define the quantity
\begin{align*}
{n \brack w} = \frac{r(w)}{(s+1)(n-\lceil \frac{|w|}{2} \rceil ) + r(w)} \binom{(s+1)(n-\lceil \frac{|w|}{2} \rceil ) + r(w)}{n-\lceil \frac{|w|}{2} \rceil}.
\end{align*}  
\end{itemize}

Recall that by Lemma \ref{lem.pathwalkident}, $\dim_{\Gamma^{s}} (m,w) = \Wk_{\T^s}(m,\emptyset,w)$ for every $s$-admissible word $w$ and $m\geq 1$.

\begin{theorem}[\cite{La01}] \label{thm.dimLandau}
We have 
\begin{align*}
{n \brack w} = \begin{cases} \dim_{\Gamma^{s}} (2n,w) \qquad &\text{if } |w| \text{ is even}, \\   \dim_{\Gamma^{s}} (2n-1,w) \qquad &\text{if } |w| \text{ is odd}. \end{cases}
\end{align*}
If $|w| \neq m \mod 2$, then $\dim_{\Gamma^{s}} (m,w)=0.$
\end{theorem}

\begin{remark} \label{rem.generatingfunction}
\begin{enumerate}
\item[(1)] The number of rooted loops of length $2n$ on the $s$-Fuss-Catalan tree $\T^s$ is the Fuss-Catalan number 
\[ \dim_{\Gamma^s}(2n,\emptyset) = C^s_n := \frac{1}{(s+1)n + 1} \binom{(s+1)n + 1}{n}. \] 
\item[(2)] The ratio test shows that the generating function 
\[ G_s(z) \ = \ \sum_{n=0}^{\infty} C^s_n z^n  \]
has radius of convergence $ \tfrac{s^{s}}{(s+1)^{s+1}}$. Using Stirling's formula, it follows that also for the critical value $ \tfrac{s^{s}}{(s+1)^{s+1}}$, the series $\sum_{n=0}^{\infty} C^s_n \left( \tfrac{s^{s}}{(s+1)^{s+1}} \right)^n < \infty$ converges. In fact, explicit formulas for the probability measures that have $(C^s_n)_{n\geq 0}$ as their moment sequence have been computed. These measures are supported on $\left[0,\tfrac{(s+1)^{s+1}}{s^{s}}\right]$ and explicit formulae can be found in \cite{MPZ13}.
\item[(3)] By \cite[p. 360ff]{GKP94}, the function $G_s$ satisfies the recursion $G_s(z) = z G_s(z)^{s+1} +1$. Moreover, it is shown there that the $n$-th coefficient in the series expansion of $G_s^l$ is given by
\begin{align*}
[z^n]G_s(z)^l = \frac{1}{(s+1)n + l} \binom{(s+1)n + l}{n}.
\end{align*}
\end{enumerate}
\end{remark}

The proof of Theorem \ref{thm.dimLandau} can be adapted to the derooted $s$-Fuss-Catalan tree, yielding in particular an explicit formula for the dimensions of the irreducible representations of the algebras $A_{(\cB^{\#},\delta)}(n), \ n \geq 0$ at the generic parameter. Note that the distance to the root in $\widetilde{\T}^s$ of a word $w$ is $|w|-1$. Let 
\begin{align*}
{n \brack w}_{\#} = \frac{r(w)-1}{(s+1)(n-\lceil \frac{|w|-1}{2} \rceil ) + r(w)-1} \binom{(s+1)(n-\lceil \frac{|w|-1}{2} \rceil ) + r(w)-1}{n-\lceil \frac{|w|-1}{2} \rceil}.
\end{align*}
 for $n \geq \frac{|w|-1}{2}.$
 
\begin{proposition}
In the derooted $s$-Fuss-Catalan tree $\widetilde{\T}^{s}$, we have 
\begin{align*}
{n \brack w}_{\#} = \begin{cases} \Wk_{\widetilde{\T}^{s}}(2n,\emptyset_{\widetilde{\T}^{s}},w)\qquad &\text{if } |w| \text{ is even}, \\   \Wk_{\widetilde{\T}^{s}}(2n-1,\emptyset_{\widetilde{\T}^{s}},w) \qquad &\text{if } |w| \text{ is odd}. \end{cases}
\end{align*}
In particular, the dimension of the irreducible representations $A_{(\cB^{\#},\delta)}(2n)$ indexed by a word $(2n,w), \ |w|$ odd, is 
\begin{align*}
\frac{r(w)-1}{3(n-\lceil \frac{|w|-1}{2} \rceil ) + r(w)-1} \binom{3(n-\lceil \frac{|w|-1}{2} \rceil ) + r(w)-1}{n-\lceil \frac{|w|-1}{2} \rceil}.
\end{align*}
The same formula holds for the dimension of $(2n-1,w)$ with $|w|$ even.
\end{proposition} 

\begin{proof}
By the same argument as in the proof of \cite[Theorem 11]{La01}, $\Wk_{\widetilde{\T}^{s}}(2n,\emptyset,w)$ is the coefficient of $z^{n-\lceil \frac{|w|-1}{2} \rceil}$ in the $r(w)-1$ power of the generation function $G_s$ of Remark \ref{rem.generatingfunction}. The formula then follows from the third part of Remark \ref{rem.generatingfunction}.
\end{proof}

\begin{remark}
In the special case, when $w=a$ is the root of the derooted Fibonacci tree, another formula for the dimension of $(2n,a)$ is given in \cite[Proposition 4.3]{We13} and it is easy to check that both formulas yield the same number.
\end{remark}



\subsubsection{The minimal boundary for random walks on the Fibonacci tree} \label{sec.deFintheorem}

The observations of Section \ref{subsec.walkinterpretation} and Theorem \ref{thm.boundarytraces} turn the problem of classifying extremal traces on $\FC_s(\infty)$ into a classification problem for random walks on Fuss-Catalan trees. For homogeneous trees, central ergodic random walks have been classified in \cite{VM15}. However, the symmetric nature of the homogeneous trees renders the combinatorics of that problem more simple than for the Fuss-Catalan trees and we will need to make regular use of the walk counting formulas of Subsection \ref{subsec.walkcount}. An interesting result that we prove along the way is the law of large numbers mentioned in the introduction (Theorem \ref{thmB}).  Note that laws of large numbers for central measures on branching graphs have also been proven in other cases, see e.g. \cite{Me17} for the Young graph. The combinatorics of these examples are quite different to the situation we face in this section.

Let us start this section by introducing a few notations that we will use down the road.
Let $\T$ be a locally finite rooted tree in which every node has at least one successor. An infinite path $(\emptyset=t_1,t_2,\dots)$ on a tree $\T$ is typically called an \emph{end} of $\T$. We will denote the set of ends of $\T$ by $\partial \T$. Note that is no violation of our previous notation $\partial \Gamma$ for the minimal boundary of a branching graph $\Gamma$. In fact, if $\Gamma = \T$, the set of ends of $\T$ is in bijection with the minimal boundary of $\T$ as we can associate to every end $t$ the Dirac measure $\delta_t$ and all ergodic central measures on the space of ends are of this form. With this identification in mind, we  interpret $\partial \T$ as a topological space (with the weak topology). Also note that for every vertex $v$ of $\T$ and every end $t= (\emptyset=t_1,t_2,\dots)$, there exists a unique vertex $t_m$ of $t$ such that $v$ is a descendant or equal to $t_m$. The \emph{geodesic ray} $[v,t\rangle$ from $v$ to $t$ is the unique path $(v,\dots,t_m,t_{m+1},\dots)$ on the tree that leads from $v$ to $t_m$ and then continuous on $t$.

\begin{definition} \label{def.convergetoend}
\begin{itemize}
\item We say that a sequence $(x_n)_{n\geq 0}, \ x_n \in \T$ converges to an end $t \in \partial \T^s$ if for every vertex $v \in \T$, the length of the common part $[v,x_n] \cap [v,t\rangle$ of the geodesic path $[v,x_n]$ from $v$ to $x_n$ and the geodesic ray $[v,t\rangle$ from $v$ to $t$ tends to infinity as $n \to \infty$. 
\item Similarly, we say that a path $(n,x_n)_{n \geq 0}$ in the pascalized graph $\cP(\T)$ converges to $t$ if $(x_n)_{n \geq 0}$ converges to $t$ in the sense above.
\end{itemize}
\end{definition}

\begin{definition} \label{def.tdirected}

Let $\T$ be a locally finite rooted tree, let $t \in \partial \T$ be an end and let $v,w$ be neighboring vertices of $\T$.
\begin{itemize} 
\item We will call the (ordered) pair $(v,w)$ $t$-\emph{directed} if $w$ lies on the unique geodesic from $v$ to $t$. 
\item Similarly, we will call an edge $((k,v),(k+1,w))$ in the pascalized graph $\cP(\T)$ $t$-directed if $(v,w)$ is $t$-directed in $\T$.
\end{itemize}
\end{definition}

Note that for every pair of neighboring vertices $v,w$ of $\T$, either $(v,w)$ or $(w,v)$ is $t$-directed.

\begin{proposition} \label{prop.randomwalkcentral}
Let $S$ be a random walk on a locally finite tree $\T$ starting at the root. Then $S$ is central if and only if there exists a constant $\eta \geq 0$ such that the transition probabilities of $S$ satisfy 
\begin{align} \label{eq.crossingprobabilities}
 p(v,w)p(w,v) = \eta 
\end{align}
for every pair of neighboring vertices $v,w$.
\end{proposition}

\begin{proof}
Assume first that Equation \ref{eq.crossingprobabilities} holds for all pairs of neighboring vertices $v,w$. Let $x$ be a vertex of $\T$ and let $n \equiv |x| \mod 2$. We need to show that $S$ takes all $n$-step walks from $\emptyset$ to $x$ with the same probability. Let $\emptyset = t_0,t_1,\dots, t_{|x|} = x$ be the unique path on $\T$ from $\emptyset$ to $x$. We can then decompose any $n$-step walks from $\emptyset$ to $x$ as
\begin{center}
(a loop on $t_0$) $(t_0 \rightarrow t_1)$ (a loop on $t_1$) $(t_1 \rightarrow t_2) \cdots (t_{|x|-1} \rightarrow x)$ (a loop on $x$).
\end{center}
Since $\T$ is a tree, every edge on one of the loops has to be traversed equally many times in both directions. Since by assumption $p(v,w)p(w,v) = \eta$ is constant for all edges $(v,w)$, it follows that every $n$-step walk from $\emptyset$ to $x$ is taken with probability 
\[ \eta^{n-|x|} \prod_{i=0}^{|x|-1} p(t_i,t_{i+1}). \]   
Conversely, it suffices to note the following for a given vertex $x$ with $|x|=n$: by centrality the loop of length $2n$ obtained by following the geodesic from $\emptyset$ to $x$ and back has the same probability as the loop obtained by only following this geodesic up to the father of $x$, returning and taking $2$ steps between the root and level $1$. The claim then follows by combining this fact with an induction over the distance to the root. 
\end{proof}

Given an end $t \in \partial \T^s$ and a parameter $\eta \in [0, \tfrac{s^{s}}{(s+1)^{s+1}}]$, we will now define a random walk $S_{(t,\eta)}$ starting at the root of the $s$-Fuss-Catalan tree as follows. If $a$ is the unique vertex connected to the root, we define the transition probabilities $p_{(t,\eta)}(\emptyset,a) = 1, \ p_{(t,\eta)}(a,\emptyset) = \eta.$ Next consider an edge $(v,w)$ that does not lie on the end $t$. For such an edge, we define the transition probabilities to be 
\begin{align} \label{eq.transprob}
p_{(t,\eta)}(v,w) = \begin{cases} G_s(\eta)^{-l(v)} \quad &\text{if } (v,w) \text{ is } t\text{-directed,} \\ \eta \cdot G_s(\eta)^{l(w)} \quad &\text{else},  \end{cases}
\end{align}
where $G_s(z)  = \sum_{n=0}^{\infty} C^s_n z^n$ is the generating function of the $s$-Fuss-Catalan numbers discussed in Remark \ref{rem.generatingfunction}. Since $\eta \leq \tfrac{s^{s}}{(s+1)^{s+1}}$, the number $G_s(\eta)$ is well-defined. In addition, $G_s(\eta) \geq 1$, so that the first value for $p_{(t,\eta)}(v,w)$ is bounded by $1$. The fact that also $\eta \cdot G_s(\eta)^{l} \leq 1$ for all $l=1,\dots,s$ follows directly from the recursion formula $G_s(z) = zG_s(z)^{s+1} +1$, see Remark \ref{rem.generatingfunction}.

So far, we have defined all transition probabilities except for those between neighboring vertices $t_i, t_{i+1}, i \geq 1$ on our prescribed end $t=(t_i)_{i \geq 0}$.

\begin{lemma} \label{lem.welldefinedwalk}
Let $t=(t_i)_{i \geq 0} \in \partial \T^s$ and $\eta \in [0, \tfrac{s^{s}}{(s+1)^{s+1}}]$. Then there exists a unique central random walk $S_{(t,\eta)}$ on $\T^s$ such that
\begin{itemize}
\item $p_{(t,\eta)}(\emptyset,a) = 1, \ p_{(t,\eta)}(a,\emptyset) = \eta$ for the unique level $1$ vertex $a$;
\item For any edge $(v,w)$ that does not lie on $t$, the transition probablities $p_{(t,\eta)}(v,w),p_{(t,\eta)}(w,v)$ are given by Equation \ref{eq.transprob};
\item For any edge $(v,w)$, we have $p_{(t,\eta)}(v,w) p_{(t,\eta)}(w,v) = \eta$.
\end{itemize}
\end{lemma} 

\begin{proof}
Let $x$ be a vertex that does not lie on $t$. We have already shown above that $0 \leq p_{(t,\eta)}(x,w) \leq 1$ for every neighbor $w$ of $x$. Let $v$ be the unique neighbor of $x$ such that $(x,v)$ is $t$-directed. Since $x \neq t_i$ for all $i \geq 0$, we must have $|v| = |x|-1$. Therefore 
\[ \sum_{w \text{ nb. of } x} p_{(t,\eta)}(x,w) = 1 \]
is equivalent to 
\begin{align*}
G_s(\eta)^{l(x)}  = 1 + \sum_{w \text{ descend. of } x} \eta G_s(\eta)^{l(x)+ l(w)} 
\end{align*}
and this equality follows from the computation
\begin{align*}
1 + \sum_{w \text{ descend. of } x} \eta G_s(\eta)^{l(x)+ l(w)} &= 1 + \sum_{i=1}^{l(x)} \eta G_s(\eta)^{l(x)+ s-i+1}\\
&= 1 + \sum_{i=0}^{l(x)-1} \eta G_s(\eta)^{s+1 + i} \\
&= 1 + \sum_{i=0}^{l(x)-1} G_s(\eta)^{i} (G_s(\eta) -1) \\
&= G_s(\eta)^{l(x)},
\end{align*}
where we use the relation $G_s(z) = z G_s(z)^{s+1} +1$. The fact that the missing transition probabilities $p_{(t,\eta)}(t_i,t_{i+1})$ and $p_{(t,\eta)}(t_{i+1},t_i)$ are uniquely defined now follows by induction over $i$.
 The fact that $S_{(t,\eta)}$ is central follows from Proposition \ref{prop.randomwalkcentral}.
\end{proof}

Note that when $\eta =0$, the random walk $S_{(t,0)}$ converges deterministicly to the end $t$.




Now for $t \in \partial \T^s$ and $\eta \in [0, \tfrac{s^{s}}{(s+1)^{s+1}}]$ define the Markov measure $\nu_{(t,\eta)}$ on the space of infinite paths on $\cP(\T^s)$ by 
\begin{align*}
p_{\nu_{(t,\eta)}}((n,v),(n+1,w)) : = p_{(t,\eta)}(v,w),
\end{align*}
so that $\pi(\nu_{(t,\eta)})$ is the distribution of the random walk $S_{(t,\eta)}$. By definition of the transition probabilities, the measure $\nu_{(t,\eta)}$ is time-homogeneous.




We are now ready to state a more precise version of Theorem \ref{thmA}. The rest of this section will be devoted to its proof. Most of the partial results leading up to the proof, will be formulated for general Fuss-Catalan trees. 

\begin{theorem} \label{thm.classificationFC}
The set of all ergodic central measures on the space $(\Omega_{\cP(\FT)}, \cF_{\cP(\FT)})$ of infinite paths on $\cP(\FT)$ conincides with the family of Markov measures
\begin{align*}
\Sigma := \{ \nu_{(t,\eta)} \ ; \ t \in \partial \FT, \ \eta \in [0,4/27] \}.
\end{align*}
\end{theorem}

Denote by $\tilde{S}_{(t,\eta)}$ the random walk starting at the root of $\widetilde{\FT}$ with the same transition probabilities as $S_{(t,\eta)}$ outside of the fixed end $t \in \partial \widetilde{\FT} \cong \partial \FT$ and denote by $\tilde{\nu}_{(t,\eta)}$ its pullback to $\cP(\widetilde{\FT})$.

\begin{corollary} \label{cor.classFC}
\begin{itemize}
\item[(a)] A full list of extremal traces on the infinite Fuss-Catalan algebra $\FC_2(\infty,\delta)$ is given by
\begin{align*}
\{ \tau_{(t,\eta)} \ ; \ t \in \partial \FT, \ \eta \in [0,4/27] \}
\end{align*} 
where
\begin{align*}
\tau_{(t,\eta)}(x) = \sum_{(n,v_i) \in \Gamma^2_n} \nu_{(t,\eta)}\left(X_n = (n,v_i)\right) \frac{\tau_{(n,v_i) }(x)}{\dim_{\Gamma^2}(n,v_i)} \qquad (x \in \FC_2(n,\delta)).
\end{align*}
Here $\tau_{(n,v_i) }$ denotes the (unnormalized) trace on the simple direct summand indexed by $(n,v_i) $ in the decomposition of $\FC_2(n,\delta)$.
\item[(b)] The same statement holds for the extremal traces on the algebra  $A_{(\cB^{\#},\delta)}(\infty)$ if $\nu_{(t,\eta)}$ is replaced by $\tilde{\nu}_{(t,\eta)}$.
\end{itemize}
\end{corollary}

\begin{proof}
The corollary follows by combining Theorem \ref{thm.classificationFC} with Theorem \ref{thm.boundarytraces}.
\end{proof}




Recall that a Markov chain on a tree \cite{VM15} is \emph{transient} if its number of returns to $\emptyset$ is finite with probability one.

\begin{lemma} \label{lem.transient}
Let $\nu$ be an ergodic central measure on the space $(\Omega_{\Gamma^s},\cF_{\Gamma^s})$ of infinite paths on the branching graph $\Gamma^s = \cP(\T^s)$. Then the Markov chain $\pi(\nu)$ on $\T^s$ is transient. 
\end{lemma}

\begin{proof}
Consider an ergodic central measure $\nu$ on $(\Omega_{\Gamma^s},\cF_{\Gamma^s})$ and suppose that $\pi(\nu)$ is recurrent, i.e. the Markov chain returns to the root $\emptyset$ infinitely many times, with positive probability. Translating this statement into a statement for the branching graph $\Gamma^s = \cP(\T^s)$, we see that the set
\begin{align*}
W= \{(n_i,w_i)_{i\geq 1} \in \Omega_{\Gamma^s} \ ; \ w_i = \emptyset \ \text{for infinitely many i} \}
\end{align*}
has positive probability $\nu(W) >0$. By the ergodic method (Theorem \ref{thm.ergodicmethod}), we know that for almost every path $(n_i,w_i)_{i\geq 1} \in W$,    for every neighbouring pair of vertices $v \in \T^s_k, x\in \T^s_l, \ |k-l| =1$ and any $m  \geq k, m= k \mod 2$, the transition probability from $v$ to $x$ on step $m+1$ is
\begin{align*}
p_{\nu}((m,v),(m+1,x))= \lim_{i \to \infty} \frac{\dim((m+1,x), (n_i,w_i))}{\dim((m,v),(n_i,w_i))}.
\end{align*}
By definition of the sequence $(n_i,w_i)_{i\geq 1} $, there exists a subsequence $(n_{i_k},w_{i_k})$ such that $w_{i_k} = \emptyset$ for all $k$, whence
\begin{align*}
p_{\nu}((m,v),(m+1,x))= \lim_{k \to \infty} \frac{\dim((m+1,x), (n_{i_k},\emptyset))}{\dim((m,v),(n_{i_k},\emptyset))}.
\end{align*}
By the path-walk identification Lemma \ref{lem.pathwalkident}, for $n >m$, $\dim((m,v),(n,\emptyset))$ is exactly the number of $n-m$-step walks on $\T_{\FC_s}$ starting at $v$ and ending at $\emptyset$. Reversing the direction of these walks, and setting $l(n) = \lceil n-m/2 \rceil$, it follows from Theorem \ref{thm.dimLandau} that
\begin{align*}
\dim((m,v),(n,\emptyset)) = \dim((0,\emptyset),(n-m,v)) = { l(n) \brack v}
\end{align*} 
so that 
\begin{align*}
p_{\nu}((m,v),(m+1,x))= \lim_{k \to \infty} \frac{{ l(n_{i_k}-1) \brack x}}{{ l(n_{i_k}) \brack v}}.
\end{align*}
From this, it is not hard to derive explicit formula for the transition probabilities which will be of the form 
\begin{align*}
p_{\nu}((m,v),(m+1,x))= K \cdot\frac{r(x)}{r(v)}
\end{align*}
with a factor $K \geq 0$ that only depends on $s, r(v)-r(w) \in \{ -s,\dots, s \}$ and $|v|-|x| \in \{ -1,1 \}$. In particular, the transition probabilities $p_{\nu}((m,v),(m+1,x)) = p_{\nu}(v,x)$ are independent of $m$. Now, choosing $x=\emptyset$ and $v$ to be the unique level $1$ vertex (for the proof for $\widetilde{\T}^s$ choose any one of the two level 1 vertices), the transition probability  $p_{\nu}((m,v),(m+1,\emptyset))$ takes the value
\begin{align*}
p_{\nu}(v,\emptyset) = \frac{s^{s}}{(s+1)^{s+1}}.
\end{align*} 
Consider the $2n$ step rooted loop that jumps from the root to $v$ and back $n$ times. Our computation then implies that the probability that the Markov chain $\pi(\nu)$ follows this walk in the first $2n$ steps is $\left( \tfrac{s^{s}}{s+1^{s+1}} \right)^n$. By centrality any other $2n$ step rooted loop must have the same probability. As there are exactly $C_n^s$ of these loops, the probability to return to the root after $2n$ steps is
$p_n = C_n^s \left( \frac{s^{s}}{s+1^{s+1}} \right)^n$. By Remark \ref{rem.generatingfunction}, it thus follows that 
\begin{align*}
\sum_{n=0}^{\infty} p_n = \sum_{n=0}^{\infty} C_n^s \left( \frac{s^{s}}{(s+1)^{s+1}} \right)^n < \infty,
\end{align*} 
which implies that $\pi(\nu)$ is not recurrent, contradictory to our assumption.
\end{proof}

\begin{remark}
Note that the argument of the proof of the previous lemma combined with Remark \ref{rem.generatingfunction} also shows that for any time-homogeneous ergodic central measure $\nu$, the transition probability from the first level vertex $v$ to $\emptyset$ must be bounded above by $p_{\nu}(v,\emptyset) \leq \tfrac{s^{s}}{(s+1)^{s+1}}. $
\end{remark}

\begin{lemma} \label{lem.convergingtoend}
Let $\nu$ be an ergodic central measure on $(\Omega_{\Gamma^s},\cF_{\Gamma^s})$. Then there exists an end $t \in \partial \T^s$ to which $\nu$-almost every path on $\cP(\T^s)$ converges. Moreover, for every $\nu$-typical path $(n,w_n)_{n \geq 0}$, the limit
\begin{align*}
\lim_{n \to \infty} \frac{n - |w_n|}{2r(w_n)} \in [0,+\infty]
\end{align*}
exists and is almost surely constant.
\end{lemma}

\begin{proof}
Since $\pi(\nu)$ is transient by Lemma \ref{lem.transient}, $\nu$-almost every path $x=(n,w_n)$ of $\cP(\T^s)$ converges to some end $t_x \in \partial \T^s$. As for every $t \in \partial \T^s$, the set of paths that converge to $t$ is invariant under the tail relation on $\cP(\T^s)$ (i.e. changes of starting paths of finite length), ergodicity implies that $\pi(\nu)$-almost all walks converge to the same end. 

Let $(n,w_n)_{n \geq 0}$ be a $\nu$-typical path and $k \geq 0$ even. Arguing as in Lemma \ref{lem.transient}, by the path-walk identification and the ergodic method, the limit 
\begin{align*}
p_{\nu}((k+1,a),(k+2,\emptyset)) &= \lim_{n \to \infty} \frac{\dim((k+2,\emptyset), (n,w_n))}{\dim((k+1,a),(n,w_n))} \\
&= \lim_{n \to \infty} \frac{\Wk(n-k-2,\emptyset,w_n)}{\Wk(n-k-1,a,w_n)} \\
&= \lim_{n \to \infty} \frac{\Wk(n-k-2,\emptyset,w_n)}{\Wk(n-k,\emptyset,w_n)}
\end{align*}
exists. Let$l=k/2$. If $n=2m$ is even, then 
\begin{align*}
\frac{\Wk(n-k-2,\emptyset,w_n)}{\Wk(n-k,\emptyset,w_n)} &= \frac{{m-l-1 \brack w_{2m}}}{{m-l \brack w_{2m}}}
\end{align*}
and if $n=2m-1$ is odd then 
\begin{align*}
\frac{\Wk(n-k-2,\emptyset,w_n)}{\Wk(n-k,\emptyset,w_n)} &= \frac{{m-l-1 \brack w_{2m-1}}}{{m-l \brack w_{2m-1}}}.
\end{align*}
Hence, in any case 
\begin{align*}
\lim_{n \to \infty} \frac{\Wk(n-k-2,\emptyset,w_n)}{\Wk(n-k,\emptyset,w_n)} &= \lim_{n \to \infty} \frac{{\lceil (n-k-1)/2 \rceil \brack w_{n}}}{{\lceil (n-k)/2 \rceil \brack w_{n}}} \\
&= \lim_{n \to \infty} \left(\frac{n-|w_{n}|}{2 r(w_n)} \right)  \frac{(s(\tfrac{n-|w_{n}|}{2r(w_n)})+1)^{s} }{((s+1)(\tfrac{n-|w_{n}|}{2r(w_n)})+1)^{(s+1)}},
\end{align*}
where we used $\lim_{n\to \infty} r(w_n) = \infty$ which holds since $\pi(\nu)$ is transient. Note that this limit depends no longer on $k$, i.e. $p_{\nu}((k+1,a),(k+2,\emptyset)) =: p_{\nu}(a,\emptyset)$. Set $c_n = \tfrac{n-|w_{n}|}{2 r(w_n)}$ and $f:[0,+\infty) \to [0,+\infty), \ f(x) = \tfrac{x(sx+1)^s}{((s+1)x+1)^{s+1}}$, so that $\eta = \lim_{n \to \infty} f(c_n).$
One checks that $f$ is a strictly increasing continuous function with $\lim_{x \to \infty} f(x) = \tfrac{s^s}{(s+1)^{s+1}}$ so that as a function onto $[0,\tfrac{s^s}{(s+1)^{s+1}})$, $f$ is invertible with continuous inverse $f^{-1}: [0,\tfrac{s^s}{(s+1)^{s+1}}) \to [0,\infty)$. It follows that $\lim_{n \to \infty} c_n = f^{-1}(\eta)$ exists in $[0,+\infty]$, where we set $f^{-1}(\tfrac{s^s}{(s+1)^{s+1}}):= + \infty$. Note that our argument is independent of the choice of typical path $(n,w_n)$ (alternatively, apply ergodicity of $\nu$ to the limiting random variable).
\end{proof}

Since we have shown in the proof of the previous lemma that the product 
\[p_{\nu}((k+1,a),(k+2),\emptyset) =\eta \]
is independent of $k$, we can repeat the proof of Proposition \ref{prop.randomwalkcentral} to get the following.

\begin{lemma} \label{lem.etaMarkovchains}
If $\nu$ is a central ergodic measure on $\cP(\T^s)$, then there is a constant $\eta \in  [0,\frac{s^{s}}{(s+1)^{s+1}}]$ such that
\[p_{\nu}((k,v),(k+1,w) \cdot p_{\nu}((k+1,w),(k+2,v)) =\eta \]
 for every edge $((k,v),(k+1,w))$ with $|w|=|v|+1$ on $\cP(\T^s)$.

\end{lemma}

We call the constant $\eta = \eta_{\nu}$ from Lemma \ref{lem.etaMarkovchains} the \emph{structure constant} of the ergodic central measure $\nu$.

\begin{proposition} \label{prop.firstdirection}
Let $\nu$ be an ergodic central measure on $(\Omega_{\Gamma^s},\cF_{\Gamma^s})$, let $t \in \partial \T^s$ be the end to which $\nu$-almost every path on $\cP(\T^s)$ converges and let $\eta \in [0,\frac{s^{s}}{(s+1)^{s+1}}]$ be the structure constant from the previous lemma. If $\eta < \frac{s^{s}}{(s+1)^{s+1}}$, then $\nu = \nu_{(t,\eta)}$. \\
Moreover, given an end $t$, there exists at most one ergodic central measure $\nu^*_{(t,\tfrac{s^{s}}{(s+1)^{s+1}})}$ with structure constant $\eta = \tfrac{s^{s}}{(s+1)^{s+1}}$ that converges to $t$. This measure is the pull back of (the law of) a random walk $S_{(t,\tfrac{s^{s}}{(s+1)^{s+1}})}$ and thus time-homogeneous.
\end{proposition}

\begin{proof}
If $\eta =0$, it follows inductively that in every step $\pi(\nu)$ jumps to a descendent of its current location with probability one and thus converges deterministically to $t$.  Hence in this case $\nu = \nu_{(t,0)}$.

Therefore assume now $\eta >0$. Let $(n,w_n)_{n \geq 0}$ be a $\nu$-typical path, and let $((k,v),(k+1,w))$ be an edge on $\cP(\T^s)$ with $|w| = |v|+1$. Arguing as in Lemma \ref{lem.convergingtoend}, by the ergodic method
\begin{align*}
\frac{\Wk(n-k-2,v,w_n)}{\Wk(n-k,v,w_n)} &= \frac{\Wk(n-k-2,v,w_n)}{\Wk(n-k-1,w,w_n)} \frac{\Wk(n-k-1,w,w_n)}{\Wk(n-k,v,w_n)} \\
&= \frac{\dim_{\cP(\T^s)}((k+2,v),(n,w_n))}{\dim_{\cP(\T^s)}((k+1,w),(n,w_n))} \cdot \frac{\dim_{\cP(\T^s)}((k+1,w),(n,w_n))}{\dim_{\cP(\T^s)}((k,v),(n,w_n))}\\
&\to \quad p_{\nu}((k+1,w),(k+2,v)) \cdot p_{\nu}((k,v),(k+1,w)) = \eta
\end{align*}
as $n \to \infty$. More generally, for fixed $k,l > 0$
\begin{align*}
\frac{\Wk(n-k-2l,v,w_n)}{\Wk(n-k,v,w_n)} = \prod_{i=1}^l \frac{\Wk(n-k-2i,w,w_n)}{\Wk(n-k-2(i-1),w,w_n)} \quad \to \quad \eta^l.
\end{align*}
Assume now that $(w,v)$ is $t$-directed and does not lie on $t$. Note that this automatically implies $|w| = |v|+1$. Since $(w_n)$ converges to $t$, every walk from $w$ to $w_n$ must pass through $v$ at some point for large enough $n$.  Hence any walk from from $w$ to $w_n$ (for large $n$) of length $m$ must split into a downward loop of length $2l$ rooted at $w$ and a walk of length $m-2l-1$ from $v$ to $w_n$. Since the number of downward loops of length $2l$ rooted at $w$ is $C^{l(w)}_l$, it follows that
\begin{align*}
\frac{\dim_{\cP(\T^s)}((k+1,w),(n,w_n))}{\dim_{\cP(\T^s)}((k,v),(n,w_n))} &= \frac{\Wk(n-k-1,w,w_n)}{\Wk(n-k,v,w_n)} \\
&= \sum_{l=0}^{\tfrac{n-k-d(w,w_n)}{2}} C^{l(w)}_l \frac{\Wk(n-k-2(l+1),v,w_n)}{\Wk(n-k,v,w_n)}.
\end{align*}
On the one hand by the ergodic method this expression converges to $p_{\nu}((k,v),(k+1,w))$, on the other hand by Lemma \ref{lem.convergingtoend} $\tfrac{n-k-d(w,w_n)}{2} \to \infty$ as $n$ goes to infinity.
If $\eta < \tfrac{s^{s}}{(s+1)^{s+1}}$, then we can make use of the uniform convergence of the power series $\sum_{l=0}^{L} C^{l(w)}_l z^l$ to $G_s(z)^{l(w)}$ on $[0,\frac{s^{s}}{(s+1)^{s+1}}]$, to conclude that the right hand side converges to $\eta G_s(\eta)^{l(w)}$, so that
\begin{align*}
p_{\nu}((k,v),(k+1,w)) = \eta G_s(\eta)^{l(w)}
\end{align*}
for all $k \geq |v|$. By Lemma \ref{lem.etaMarkovchains}, also
\begin{align*}
p_{\nu}((k,w),(k+1,v)) = G_s(\eta)^{-l(w)}
\end{align*}
for $k \geq |w|$. Note that since either $(v,w)$ or $(w,v)$ must be $t$-directed, we have determined all transition probabilities outside of the end $t= (t_j)_{j \geq 0}$. Arguing as in Lemma \ref{lem.welldefinedwalk}, it follows by induction over $j$ that the transition probabilities $p_{\nu}((k,t_j),(k+1,t_{j \pm 1}))$ are uniquely determined once the ones outside of $t$ are given and that they do not depend on $k$. Hence all transition probabilities coincide with those of the random walk $S_{(t,\eta)}$, whence $\nu = \nu_{(t,\eta)}$.
Note that if $\eta = \tfrac{s^{s}}{(s+1)^{s+1}}$, our argument also shows that 
\begin{align*}
p_{\nu}((k,v),(k+1,w)) = p_{\nu}(v,w) = \lim_{n\to \infty} \sum_{l=0}^{\tfrac{n-k-d(w,w_n)}{2}} C^{l(w)}_l c_n^l > 0,
\end{align*}
where $c_n = \frac{\Wk(n-2,v,w_n)}{\Wk(n,v,w_n)}$ converges to $\eta$ and $(w,v)$ is $t$-directed. Let $S^*_{(t,\tfrac{s^{s}}{(s+1)^{s+1}})}$ be the random walk with transition probabilities $p_{\nu}(v,w) $. Repeating the last part of the argument, it follows that $\nu = \nu^*_{(t,\tfrac{s^{s}}{(s+1)^{s+1}})}$ is the pullback of the law of $S^*_{(t,\tfrac{s^{s}}{(s+1)^{s+1}})}$ and is thus time-homogeneous with structure constant $\tfrac{s^{s}}{(s+1)^{s+1}}$.
\end{proof}

\subsubsection{The random walks $S_{(t,\eta)}$} \label{subsubsec.randomwalks}

In this section, we will analyse the random walks $S_{(t,\eta)}$ on the Fibonacci tree $\FT = \T^2$. In order to do so,  for $\eta \in (0,4/27]$, let us introduce an auxillary random walk $S^{\eta}$ on the derooted Fibonacci tree $\widetilde{\FT}$. If $x_i, i=1,2$ is the unique first level vertex of $\widetilde{\FT}$ with $l(x_i) = i$, we define the transition probabilities from $\emptyset$ to $x_i$ by $p_{\eta}(\emptyset,x_1) = \tfrac{\eta G_(\eta)}{A(\eta)} = \tfrac{G_(\eta)}{G(\eta)+G(\eta)^2}$ and $p_{\eta}(\emptyset,x_2) = \tfrac{\eta G_(\eta)^2}{A(\eta)} = \tfrac{G_(\eta)^2}{G(\eta)+G(\eta)^2}$ where $A(\eta) = \eta(G(\eta)+G(\eta)^2)$. The remaining transition probabilities are defined by
\begin{align} \label{eq.auxilliaryRW}
p_{\eta}(v,w) = \begin{cases} G(\eta)^{-l(v)} \quad &\text{if } |v| = |w| +1 \\ \eta \cdot G(\eta)^{l(w)} \quad &\text{if } |v| = |w| -1  \end{cases},
\end{align}
where $G(z) = G_2(z)$ is the generating function for the $2$-Fuss-Catalan numbers.
\begin{lemma} \label{lem.recaux}
The random walk $S^{\eta}$ on the derooted Fibonacci tree is recurrent.
\end{lemma}

\begin{proof}
We will prove that $\sum_{n=1}^{\infty} \bP(S^{\eta}_{2n} = \emptyset) = \infty $. First let
\begin{align*}
D_{n,k} := |\{\text{loops of length 2n crossing the edges betw. level 0 and 1 exactly k times} \}|.
\end{align*}
By definition of $S_{\eta}$, we have
\begin{align*}
\bP(S^{\eta}_{2n} = \emptyset) = \sum_{k=1}^n \eta^n A(\eta)^{-k} D_{n,k}.
\end{align*}
Fix a vector $(i_1,\dots,i_k) \in \{1,2\}^k$. and write $e_1 = (\emptyset,x_1), e_1 = (\emptyset,x_1)$. Then the number of loops of length $2n$ crossing the edges $e_1$ and $e_2$ in the order $e_{i_1},e_{i_2},\dots, e_{i_k}$ is 
\begin{align*}
\sum_{n_1+\dots +n_k=n-k} C_{n_1}^{i_1} \dots C_{n_k}^{i_k}
\end{align*}
which is the coefficient of $z^{n-k}$ in $G(z)^{\sum_{j=1}^k i_j}$. Therefore 
\begin{align*}
D_{n,k} = \sum_{(i_1,\dots,i_k) \in \{1,2\}^k} [z^{n-k}]G(z)^{\sum_{j=1}^k i_j} = \sum_{l=0}^{k} \binom{k}{l} [z^{n-k}]G(z)^{2k-l}.
\end{align*}
Since $\sum_{n=0}^{\infty} \eta^n < \infty$, the sum $\sum_{n=1}^{\infty} \bP(S^{\eta}_{2n}= \emptyset)$ is infinite if and only if 
\begin{align*}
\sum_{n=0}^{\infty} \sum_{k=0}^n \sum_{l=0}^k \eta^n A(\eta)^{-k}  \binom{k}{l} [z^{n-k}]G(z)^{2k-l} = \infty.
\end{align*}
Reparametrising this sum yields
\begin{align*}
\sum_{l=0}^{\infty} \sum_{k=l}^{\infty} A(\eta)^{-k} \binom{k}{l} \eta^{k} \sum_{n=k}^{\infty} \eta^{n-k}  [z^{n-k}]G(z)^{2k-l} &= \sum_{l=0}^{\infty} \sum_{k=l}^{\infty} \left(\frac{\eta}{A(\eta)}\right)^k \binom{k}{l} G(\eta)^{2k-l} \\
&= \sum_{k=0}^{\infty} \left(\frac{\eta}{A(\eta)}\right)^k G(\eta)^{2k} \sum_{l=0}^{k} \binom{k}{l} G(\eta)^{-l} \\
&= \sum_{k=0}^{\infty} \left(\frac{\eta}{A(\eta)} (G(\eta)^2 + G(\eta)) \right)^k
\end{align*}
Since $\frac{\eta}{A(\eta)} (G(\eta)^2 + G(\eta)) = 1$, the result follows.
\end{proof}

\begin{corollary} \label{cor.convergenceRW}
The random walk $S_{(t,\eta)}, \ t \in \partial \FT, \ \eta \in [0,4/27]$ converges to $t$ almost surely.
\end{corollary}

\begin{proof}
Because $\eta \leq 4/27$, $S_{(t,\eta)}$ is transient. If $\eta = 0$, $S_{(t,0)}$ converges deterministically to $t$. Therefore let $\eta > 0$. Assume that the event 
\[A = \{ S_{(t,\eta)} \text{ does not converge to } t \} \]
has positive probability and denote by $d(t_k)$ the unique son of $t_k$ that does not lie on $t$ for $k \geq 0, \ l(t_k) = 2$. Then, for every trajectory $\omega = (\omega_n)_{n \geq 0} \in A$, there exists a time $N(\omega)$ and a vertex $v(\omega), \ l(v(\omega))=2$ such that subtree rooted at $v(\omega)$ does not contain any vertex of $t$ and such that $\omega_n$ lies on this subtree for all $n \geq N(\omega)$. But then, the transition probabilities of the random walk $S_{(t,\eta)}$ conditioned on $n \geq N$ coincide with those of $S$. Since $S$ is recurrent by Lemma \ref{lem.recaux} this contradicts the transience of $S_{(t,\eta)}$.
\end{proof}

\subsubsection{ A law of large numbers for exit times} \label{subsubsec.LLN}

In this section, we prove the converse of Proposition \ref{prop.firstdirection} when $s=2$, i.e. the fact that the measures $\nu_{(t,\eta)}$ are ergodic. To do so, we prove a law of large numbers for normalized \emph{exit times} for the Fibonacci tree $\FT = \T^2$.

\begin{definition} \label{def.exittime}
Let $S = (S_n)_{n \geq 0}$ be a transient random walk on $\FT$ converging to an end $t = (t_k)_{k\geq 0}$. The \emph{exit time} at $t_k$ is the random variable $N_k$ defined as the last moment of passage at $t_k$, that is to say the unique nonnegative integer such that $S_{N_k} = t_k$ and $|S_m| > k$ for all $m \geq N_k$. Since $S$ converges to $t$, $N_k < \infty$ is almost surely well-defined. 

\end{definition}

Let us recall the statement of Theorem \ref{thmB}, now that we have clarified our notation. 
 
\begin{theorem}[Theorem \ref{thmB}] \label{thm.llnexit}
Let $t= (t_k)_{k \geq 0} \in \partial \FT$ an end and let $\eta \in [0,4/27]$. Denote by $N_k$ the exit time at $t_k$ for the random walk $S_{(t,\eta)}$. Then the limit 
\begin{align*}
f(\eta) = \lim_{k \to \infty} \frac{N_k -k}{r(t_k)} \in [0,\infty]
\end{align*}
exists (and is almost surely constant). It is finite if and only $\eta < 4/27$ and the function $f: [0,4/27) \to [0,+\infty), \ \eta \mapsto f(\eta)$ is injective.
\end{theorem} 

\begin{proof}
Let us first define discrete random variables $Y_{\eta}^{(1)}, Y_{\eta}^{(2)}$ on $2\N$ that are distributed according to
\begin{align*}
\bP[ Y_{\eta}^{(j)} = 2n] = \frac{C_n^{j} \cdot \eta^n}{G^{j}(\eta)}.
\end{align*}
Note that  $\E[Y_{\eta}^{(j)}] = \frac{2 \eta (G^{j})'(\eta)}{G^j(\eta)}$ when $\eta \in [0,4/27)$ and $\E[Y_{4/27}^{(j)}] = \infty$.
We observe that with probability one, between $N_{k-1} +1$ and $N_k$, the random walk $S_{(t,\eta)}$ performs a loop starting and ending at $t_k$ that always stays strictly below $t_{k-1}$. Since the number of such loops of length $2n$ is $C_n^{l(t_k)}$ and since every loop of length $2n$ is taken by $S_{(t,\eta)}$ with probability $\eta^n$, it follows that 
\[N_k - N_{k-1} -1 \sim Y_{\eta}^{(l(t_k))}. \] 
Moreover, thanks to the Markov property of the random walk $S_{(t,\eta)}$, the random variables $N_k - N_{k-1} -1, \ k \geq 1$ are independent. Write
\begin{align*}
F_j(t_k) = |\{ 0 \leq i \leq k \ ; \ l(t_i) =j  \}| \qquad \text{for } j=1,2. 
\end{align*}
Then $r(t_k) = F_1(t_k)+ 2 F_2(t_k)$ and $F_1(t_k) + F_2(t_k) = k+1$. Note that for all $k \geq 0$, $F_2(t_k) \geq k/2$ and $ k +1 \leq r(t_k) \leq 2k+1$. Let first $\eta = 4/27$. Then  
\begin{align*}
\frac{N_k -k}{r(t_k)} \geq \sum_{l=1 \atop l(t_l) = 2}^k (N_l - N_{l-1}-1) \sim \frac{F_2(t_k)}{r(t_k)} \frac{\sum_{l=1}^{F_2(t_k)} Y^{(2)}_l}{F_2(t_k)} \geq \frac{1}{3} \frac{\sum_{l=1}^{F_2(t_k)} Y^{(2)}_l}{F_2(t_k)}
\end{align*}
and by the law of large numbers the righthand side converges to $\infty$.
If $\eta < 4/27$, we write
\begin{align*}
N_k -k = \sum_{l=1}^k (N_l - N_{l-1}-1) &= \sum_{l=1 \atop l(t_l) = 1}^k (N_l - N_{l-1}-1) + \sum_{l=1 \atop l(t_l) = 2}^k (N_l - N_{l-1}-1) \\
&\sim \sum_{l=1}^{F_1(t_k)} Y^{(1)}_l + \sum_{l=1}^{F_1(t_k)} Y^{(2)}_l.
\end{align*}
Therefore 
\begin{align*}
\frac{N_k -k}{r(t_k)} \sim \frac{F_1(t_k)}{r(t_k)} \frac{\sum_{l=1}^{F_1(t_k)} Y^{(1)}_l}{F_1(t_k)} + \frac{F_2(t_k)}{r(t_k)} \frac{\sum_{l=1}^{F_2(t_k)} Y^{(2)}_l}{F_2(t_k)}.
\end{align*}
If $\sup_k F_1(t_k)< \infty$, by the law of large numbers for independent random variables, $\frac{N_k -k}{r(t_k)}$ converges a.s. to $f(\eta)=\tfrac{\E[Y^{(2)}_{\eta}]}{2} < \infty$. If $\sup_k F_1(t_k) = \infty$, we rewrite
\begin{align*}
\frac{N_k -k}{r(t_k)} \sim \frac{\sum_{l=1}^{F_1(t_k)} Y^{(1)}_l}{F_1(t_k)} + \frac{F_2(t_k)}{r(t_k)} \left( \frac{\sum_{l=1}^{F_2(t_k)} Y^{(2)}_l}{F_2(t_k)} -2 \frac{\sum_{l=1}^{F_1(t_k)} Y^{(1)}_l}{F_1(t_k)} \right)
\end{align*}
and we observe that
\begin{align*}
\E[Y^{(2)}_{\eta}]- 2 \E[Y^{(1)}_{\eta}] = \frac{2 \eta (G^{2})'(\eta)}{G^2(\eta)} - 2\frac{2 \eta G'(\eta)}{G(\eta)} = 0.
\end{align*} 
Hence, by the law of large numbers for independent random variables, $\tfrac{N_k -k}{r(t_k)} $ converges a.s. to $ \E[Y^{(1)}_{\eta}] = \tfrac{\E[Y^{(2)}_{\eta}]}{2} = f(\eta)$ as well.
We have established that 
\[ f(\eta) = \left(\frac{2z \tfrac{d}{dz}G^2(z)}{G^2(z)} \right) \bigg|_{z= \eta} = \frac{4\eta G'(\eta)}{G(\eta)}  \]
It  therefore remains to show that the function $f:[0,4/27) \to [0,+\infty), \ \eta \mapsto \frac{4\eta G'(\eta)}{G(\eta)}$ is injective. To prove this, we will once again recall that  the generating function $G$ satisfies 
\[ G(z) = z G^3(z) +1,  \]
so that 
\[ G'(z) = \frac{G^3(z)}{1-3zG^2(z)}. \]
Note that $ G(z) = z G^3(z) +1$, also implies that the denominator of the expression above is always nonzero.
Differentiating the function $f$, we see that
\[ f'(\eta) = \frac{G'(\eta)(G(\eta)-\eta G'(\eta))+\eta G''(\eta)}{G^2(\eta)} = \frac{G'(\eta)+\eta G''(\eta)}{G^2(\eta )} \]
using the formula for $G'(z)$ and $G(z) = z G^3(z) +1$ once more. This clearly implies that $f'(\eta)>0$ for all $\eta \in (0,4/27)$, so that $f$ is injective.
\end{proof}

\begin{remark}
Note that the previous theorem also holds for the random walk $S^*_{(t,4/27)}$, since we have only used the structure constant but no explicit transition probabilities in the proof.
\end{remark}

We are now ready to finish the proof of Theorem \ref{thmA}/Theorem \ref{thm.classificationFC}.

\begin{proof}[Proof of Theorem \ref{thm.classificationFC}]
Let $\nu$ be an ergodic central measure. By Lemmas \ref{lem.convergingtoend} and \ref{lem.etaMarkovchains}, $\nu$ converges to an end $t \in \partial \FT$ and has structure constant $\eta \in [0,4/27]$. If $\eta < 4/27$, then by Proposition \ref{prop.firstdirection}, $\nu = \nu_{(t,\eta)}$, if $\eta = 4/27$, then $\nu = \nu^*_{(t,4/27)} $. We need to prove that these measures are in fact ergodic and that $\nu^*_{(t,4/27)} = \nu_{(t,4/27)}$. If for given $t$, we show that $\nu_{(t,\eta)}$ is ergodic for $\eta < 4/27$, then $\nu_{(t,4/27)}$ is ergodic as well since the simplex of ergodic central measures is weakly closed and $\nu_{(t,\eta)} \to \nu_{(t,4/27)}$ as $\eta \to 4/27$. As there is at most one ergodic measures with structure constant $4/27$ converging to $t$, it follows that $\nu^*_{(t,4/27)} = \nu_{(t,4/27)}$.

Thus let us assume that $\eta < 4/27$. To show ergodicity of $\nu_{(t,\eta)}$, we argue as in \cite[Proof of Proposition 5.1]{VM15}. We already know that the set of ergodic measures is a subset of \begin{align*}
\tilde{\Sigma} := \{ \nu_{(t,\eta)} \ ; \ t \in \partial \FT, \ \eta \in [0,4/27) \} \cup \{\nu^*_{(t,4/27)}, \ t \in \partial \FT \}.
\end{align*}
Therefore, by ergodic decomposition, we know that $\nu_{(t,\eta)}$ can be decomposed as 
\[\nu_{(t,\eta)} = \int_{\tilde{\Sigma}} \xi \ d\rho(\xi) \]
for some probability measure $\rho$ on $\tilde{\Sigma}$. Let $c = \lim_{k \to \infty} \frac{N_k -k}{r(t_k)}$ the limit of Theorem \ref{thm.llnexit} for $\nu_{(t,\eta)}$. Again, using  Corollary \ref{cor.convergenceRW} and Theorem \ref{thm.llnexit}, it follows that,  if $\rho \neq \delta_{\nu_{(t,\eta)}}$ is not the Dirac measure at $\nu_{(t,\eta)}$, then the set of paths
\[A= \{ \omega \in \Omega_{\cP(\FT)}, \ \lim_{n \to \infty} w_n = t \text{ and } \lim_{k \to \infty} \tfrac{N_k(\omega) -k}{r(t_k)} = c  \} \]
has probability $\nu_{(t,\eta)}(A) < 1$, a contradiction. Thus, $\nu_{(t,\eta)}$ is ergodic. 
\end{proof}

\end{document}